\documentclass[10pt,a4paper,reqno]{amsart}
\usepackage{amsmath,amsfonts,mathrsfs,amssymb,amsthm,mathtools}
\usepackage[shortlabels]{enumitem}

\usepackage[foot]{amsaddr}
\usepackage[toc, page]{appendix}

\usepackage[margin=1in]{geometry}

\usepackage{comment}

\usepackage[dvipsnames]{xcolor}
\definecolor{MyColor}{HTML}{0047AB}
\usepackage{fancyhdr}
\usepackage{hyperref}
\hypersetup{
colorlinks=true, 
linktoc=all,     
linkcolor=blue,  
citecolor=blue,
}

\theoremstyle{definition}
\newtheorem{anytheorem}{Theorem}[section] 
\newtheorem{definition}[anytheorem]{Definition}
\newtheorem{theorem}[anytheorem]{Theorem}
\newtheorem{corollary}[anytheorem]{Corollary}
\newtheorem{lemma}[anytheorem]{Lemma}
\newtheorem{proposition}[anytheorem]{Proposition}
\newtheorem{remark}[anytheorem]{Remark}

\newtheorem{assumption}[anytheorem]{Assumption}

\allowdisplaybreaks

\DeclareMathOperator*{\argmin}{arg\,min}

\DeclareMathOperator*{\tr}{Tr}

\newcommand{\E}{\ensuremath{\mathbb{E}}}  
\newcommand{\KL}{\text{KL}}  
  
\newcommand{\R}{\ensuremath{\mathbb{R}}}

\def\cL{\mathcal{L}}
\def\cM{\mathcal{M}}

\def\cO{\mathcal{O}}
\def\cP{\mathcal{P}}

\def\cT{\mathcal{T}}

\def\d{{ \mathrm{d}}}

\def\btau{{\boldsymbol{\tau}}}

\def\sE{{\mathbb{E}}}

\def\sN{{\mathbb{N}}}
\def\sP{\mathbb{P}}

\def\sR{{\mathbb R}}

\newcommand{\lc}
{\mathrel{\raise2pt\hbox{${\mathop<\limits_{\raise1pt\hbox
{\mbox{$\sim$}}}}$}}}

\newcommand{\gc}
{\mathrel{\raise2pt\hbox{${\mathop>\limits_{\raise1pt\hbox{\mbox{$\sim$}}}}$}}}

\newcommand{\ec}
{\mathrel{\raise2pt\hbox{${\mathop=\limits_{\raise1pt\hbox{\mbox{$\sim$}}}}$}}}

\def\bb{\begin{equation}} \def\ee{\end{equation}}
\def\bbn{\begin{equation*}} \def\een{\end{equation*}}

\def\beqn{\begin{eqnarray}}  \def\eqn{\end{eqnarray}}

\def\beqnx{\begin{eqnarray*}} \def\eqnx{\end{eqnarray*}}

\def\bn{\begin{enumerate}} \def\en{\end{enumerate}}

\def\bd{\begin{description}} \def\ed{\end{description}}

\makeatletter
\@namedef{subjclassname@2020}{\textup{2020} Mathematics Subject Classification}
\makeatother

\begin{document}
 
\title{Entropy annealing for  policy  mirror descent in continuous time and space }

\author{Deven Sethi$^1$}
\email{D.Sethi-1@sms.ed.ac.uk}
\author{David \v{S}i\v{s}ka$^1$}
\email{d.siska@ed.ac.uk}
\author{Yufei Zhang$^2$}
\email{yufei.zhang@imperial.ac.uk}

\address{$^1$School of Mathematics, University of Edinburgh, United Kingdom}
\address{$^2$Department of Mathematics, Imperial College London, United Kingdom}

\keywords{Mirror descent, policy gradient method,   stochastic relaxed control, convergence rate analysis, 
entropy regularization, annealing}

\subjclass[2020]{Primary 93E20; Secondary 49M29, 68Q25, 60H30, 35J61}

\begin{abstract}
Entropy regularization has been widely used in policy optimization algorithms to 
enhance exploration and the robustness of the optimal control; however
it also introduces
an additional regularization bias.
This work quantifies the impact of entropy regularization on the convergence of policy gradient methods for stochastic exit time control problems.
We analyze a continuous-time policy mirror descent dynamics, which updates the policy based on the gradient of an entropy-regularized value function
and adjusts the strength of entropy regularization as the algorithm progresses.
We prove that with a fixed entropy level, the mirror descent dynamics converges exponentially to the optimal solution of the regularized problem. 
We further   show that when the entropy level decays at suitable polynomial rates, the annealed flow converges to the solution of the unregularized problem at a rate of  $\mathcal O(1/S)$ for discrete action spaces and, under suitable conditions, at a rate of $\mathcal O(1/\sqrt{S})$ for general action spaces, with $S$ being the gradient flow running time.
The technical challenge lies in analyzing  the gradient flow in the infinite-dimensional space of Markov kernels for nonconvex objectives.
This paper explains how entropy regularization improves policy optimization, even with the true gradient, 
from the perspective of convergence rate.
\end{abstract}

\maketitle


\section{Introduction}
\label{sec:intro}

Policy gradient (PG) method and its variants have proven highly effective in seeking optimal feedback policies for stochastic control problems (see, e.g., \cite{konda1999actor, sutton1999policy, kakade2001natural, schulman2015trust, han2016deep, schulman2017proximal, tomar2020mirror, hamdouche2022policy, pham2022mean}). These algorithms parameterize the policy as a function of the system state and
seek the optimal policy parameterization based on gradient descent of the control objective. 
When coupled with appropriate  function approximations, they  can   handle high-dimensional systems with continuous state and action spaces \cite{van2012reinforcement,  manna2022learning}.

Despite the practical success of PG   methods,   a mathematical theory that guarantees their convergence  has remained elusive, especially  for continuous-time control problems. This challenge arises from the inherent nonconvexity of the objective function
with respect to policies,  a feature that persists even in linear-quadratic (LQ) control settings \cite{fazel2018global, giegrich2024convergence}. 
Most existing theoretical works   concentrate on discrete-time Markov decision processes (MDPs),
 leveraging specific problem structures to circumvent the nonconvexity of the loss function 
(see e.g., \cite{fazel2018global, agarwal2020optimality, mei2020global, hambly2021policy, kerimkulov2023fisher}).
However, certain structural properties crucial for analysing PG methods in MDPs are intrinsically tied to 
the time and spatial discretization scales  of the underlying system, 
making them inapplicable in continuous-time and state problems 
  \cite{tallec2019making, giegrich2024convergence}. 
Hence  new analytical techniques are necessary for designing and analyzing PG methods in continuous time and space.

In particular, the entropy-regularized relaxed control formulation has emerged as a promising approach for designing efficient algorithms in continuous-time and state control problems 
\cite{wang2020reinforcement, wang2020continuous,   jia2022policy,  guo2022entropy, szpruch2024optimal}. 
This approach regularizes the objective with an additional entropy term, presenting a natural extension of well-established regularized MDPs (see e.g.,   \cite{ahmed2019understanding, mei2020global, kerimkulov2023fisher}) into the continuous domain. Entropy regularization guarantees the existence of the optimal 
stochastic 
policy,
which facilitates extending existing PG algorithms for MDPs with softmax policies to the continuous-time setting \cite{  wang2020continuous, jia2022policy}.
It also 
 ensures the Lipschitz stability with respect to  the underlying model \cite{reisinger2021regularity}, which  is critical for analyzing the sample complexity of algorithms  \cite{ basei2022logarithmic, szpruch2021exploration, guo2023reinforcement}.

 Despite the recent increased interest in entropy-regularized control problems,  to the best of our knowledge there is   no theoretical work quantifying the impact of entropy regularization on the convergence of PG methods. A high degree of entropy regularization convexifies the optimization landscape  \cite{ahmed2019understanding}  but also introduces a larger regularization bias.
This yields  the following natural  question: 

\vspace{0.3cm}
\emph{
How should    the strength of entropy  regularization be set  for a convergent policy gradient algorithm in continuous-time control problems?
} 
\vspace{0.3cm}

This work takes an initial step toward
  answering this question   in the context of  exit time control problems.
  We propose a policy mirror descent algorithm based on the gradients of entropy-regularized value functions.
  The strength of entropy regularization is determined by balancing the resulting regularization bias and the optimization error, thereby optimizing the convergence rate of the proposed algorithm.

\subsection{Outline of main results}  
In the sequel, we provide a road map of the key ideas and contributions of this work without introducing needless
technicalities. The precise assumptions and statements of the results can be found in Section \ref{sec:main_results}.

\subsubsection*{Exit time control problem}
We conduct a thorough analysis  for   exit time relaxed control problems with drift control;
see   Section \ref{sec:control_diffusion}  for  extension to controlled diffusion coefficients.
Let $A$ be a metric space  representing 
the action space,
and let $\cP(A|\sR^d)$ be the space of probability  kernels 
 representing   all stochastic polices. 
For each $x\in \sR^d$ and $\pi\in \cP(A|\sR^d)$, 
let 
the state process  $X^{x,\pi}$ be the unique weak solution 
to 
the following  dynamics:
\begin{align}
\label{eqn:state_intro}
dX_t=\left(\int_A b(X_t,a)\pi(da | X_t)\right)dt+\sigma(X_t)dW_t,\enspace t> 0; \quad X_0=x,
\end{align}
where $b:\sR^d\times A\to \sR^d$ and $\sigma:\sR^d\to\sR^{d\times d'}$ are given    measurable  functions,
and $(W_t)_{t\ge 0}$ is a $d'$-dimensional   standard Brownian motion
defined on a  filtered probability space $(\Omega,\mathcal{F},\mathbb{F},\mathbb{P}^{x,\pi})$. 
Let $\cO $ be a given bounded open subset of   $  \sR^d$
and 
consider  the value function 
\begin{align}
\label{eq:value_intro}
 v^\pi_0(x)
    \coloneqq &\E^{\sP^{x,\pi}}\left[\int_0^{\tau_\cO} \Gamma^\pi_t \left(\int_Af(X_t,a)\pi(da | X_t) \right)dt+ \Gamma^\pi_{\tau_\cO}  g(X_{\tau_\cO})\right],
\end{align}
where 
$\tau_{\mathcal O}$ is  the first exit time of  $X^{x,\pi}$ from 
 $\cO   $,
$  \Gamma^\pi_t = \exp\left(-\int_0^t \int_A c(X_s^{x,\pi},a)\pi(da | X_s^{x,\pi}) ds \right)$ 
is  the controlled discount factor,
and 
$f:\sR^d\times A\to \sR$, $c:\sR\times A\to \sR$ and $g:\sR^d\to\sR $ are given measurable   functions.
The  optimal value function is 
 defined by 
\begin{align}
\label{eq:value_optimal_intro}
v^*_0(x)
\coloneqq &\inf_{\pi\in \cP(A|\sR^d) }  v^\pi_0(x),\quad x\in \overline{\cO}\,.
\end{align}
Note that 
in \eqref{eq:value_intro} and hereafter, 
we denote expectations of quantities defined on the space $(\Omega, \mathbb{F},\mathbb{P}^{x,\pi})$ with the superscript $\mathbb{P}^{x,\pi}$, indicating their dependence on $x$ and $\pi$.
 Precise assumptions on $\cO$, $b$, $\sigma$, $c$, $f$ and $g$  
are given in Assumptions \ref{ass:data} and \ref{assum:lsc}.
 
 \subsubsection*{Policy mirror descent }
We now derive a (continuous-time) mirror descent algorithm for \eqref{eq:value_optimal_intro},
which 
is  analogous to the mirror descent algorithm
in \cite{kerimkulov2023fisher}  for discrete-time MDPs.
 The algorithm  relies on  three key components: 
 1)  It optimizes the value function  \eqref{eq:value_intro}
over the class of   Gibbs policies;
2) It computes the policy gradient by regularizing the value function \eqref{eq:value_intro} with an entropy term;
3) It gradually reduces the strength of entropy regularization as the algorithm progresses.

 More precisely, 
 let $\cP(A)$ be the space of probability measures on $A$,
 let  $\mu\in \mathcal{P}(A)$ be a prescribed   reference measure,
 and consider  the  following  class $\Pi_{\mu}$ of Gibbs policies:
\begin{align}
\label{eqn:policy_intro}
\Pi_{\mu}\coloneqq \left\{\pi\in \cP(A|\sR^d)\mid \pi=\boldsymbol{\pi}(Z)\; \textnormal{for some $Z\in B_b(\cO\times A)$} \right\},
\end{align}
where   $B_b(\cO\times A)$ 
is  the space of bounded measurable functions, 
and 
the map 
$\boldsymbol{\pi}:B_b(\cO\times A)\to \cP(A|\sR^d)$
is defined by
\begin{align}
\label{eq:operator_pi_intro}
\boldsymbol{\pi}(Z)(da|x) \coloneqq 
\frac{e^{Z(x,a)}}{\int_Ae^{Z(x,{a}')}\mu(d{a}')}\mu(da),\quad x\in \cO;
\quad  
\boldsymbol{\pi}(Z)(da|x)\coloneqq  \mu(da),\quad x\not\in \cO\,.
\end{align}

Each   policy $\boldsymbol{\pi}(Z)$ in $\Pi_{\mu}$ is parameterized  by the  feature function $Z$, which extends the softmax policies  for discrete state and action spaces \cite{agarwal2020optimality, mei2020global} to the present setting with continuous state space and general action space.
The policy parameterization~\eqref{eq:operator_pi_intro}
is also inspired by 
the form of the optimal policy for  an entropy-regularized control problem,
where the feature $Z$ corresponds to the   Hamiltonian of the control problem; see Proposition~\ref{prop:verification} and \cite{jia2023q}.

Note that the map $Z\mapsto v^{\boldsymbol{\pi}(Z)}_0$ is generally nonconvex,  
even in   a stateless   bandit setting  as pointed out in \cite[Proposition 1]{mei2020global}.
This presents the main technical challenge in designing convergent policy optimization algorithms.

In the paper, we seek a (nearly) optimal   policy $\boldsymbol{\pi}(Z)$
for \eqref{eq:value_optimal_intro} 
 by optimizing the feature  $Z$ 
 via
 the following 
mirror descent flow:  
given an initial feature $Z_0 \in  B_b(\cO\times A)$, and consider
\begin{equation}
\label{eq:md_intro}
\partial_s {Z_s}(x,a) =-\left( \overline{ \mathcal L}^a v^{\boldsymbol{\pi}(Z_s)}_{\btau_s} (x)+ f(x,a)  + \btau_s Z_s(x,a)\right)  \,,\quad   (x,a)\in \cO\times A,\,s>0; \quad Z\vert_{s=0}=Z_0\,,
\end{equation}
where  
 $\overline{\mathcal{L}}^a$  is the differential operator such that 
\begin{align}
\label{eq:first_order_generator}
(\overline{\mathcal{L}}^a v)(x) = b(x,a)^\top Dv(x) -c(x,a)v(x),
\quad \forall v\in C^{1}(\cO)\,,
\end{align}
$\btau:[0,\infty)\to (0,\infty) $ is a prescribed   scheduler for the  regularization parameter,
and 
  for each $\pi\in \Pi_\mu$ and  $\tau>0$, 
  $v^\pi_\tau$ is the entropy-regularized value function defined by 
\begin{align}
\label{eq:value_tau_intro}
 v^\pi_\tau(x)
    \coloneqq & v^\pi_0(x)+\tau  \E^{\sP^{x,\pi}}\left[\int_0^{\tau_\cO} \Gamma^\pi_t  \textrm{KL}\left(\pi |\mu \right)(X_t) dt \right]\,,
    \end{align}
with $ v^\pi_0(x)$ being defined as in \eqref{eq:value_intro},
and      $  \textrm{KL}\left(\pi  |\mu \right)(X_t)$ 
being the Kullback--Leibler (KL) divergence
of $\pi (\cdot|X_t )$ with respect to $\mu$.
The additional KL divergence  in \eqref{eq:value_tau_intro}
 is crucial for ensuring the   convergence of \eqref{eq:md_intro} with continuous action spaces, as we shall discuss in detail later.
Note that for $\pi \in \Pi_\mu$ and $x\in \mathcal O$ we have $\operatorname{KL}(\pi(\cdot|x)|\mu) < \infty$ and later, once all assumptions on $b$, $\sigma$ and $c$ are stated, we will see that $\E^{\sP^{x,\pi}}\left[\int_0^{\tau_\cO} \Gamma^\pi_t  \textrm{KL}\left(\pi |\mu \right)(X_t) dt \right] <\infty$ as well.

The flow \eqref{eq:md_intro} is a continuous-time limit of a
mirror descent algorithm with  regularized gradient directions.  
Indeed, one can show that    for all $\pi,\pi'\in \Pi_\mu$ and $\tau>0$  (see  Lemma~\ref{lemma:performance_diff}),
 \begin{align*}
 \lim_{\varepsilon\searrow 0}\frac{v_\tau^{\pi+\varepsilon(\pi'-\pi)}(x)-v_\tau^{\pi}(x)}{\varepsilon}  
 &=\left\langle \pi'-\pi, \overline{\mathcal{L}}^\cdot v_\tau^{\pi}+f 
+ \tau\ln\frac{  d \pi}{  d\mu}  \right \rangle_\pi \,,
\end{align*}
where  
$\langle \cdot,\cdot  \rangle_\pi: b\cM(A\mid \sR^d)\times B_b(\cO\times A)\to \sR$ 
is a  (policy-dependent) dual   pairing defined by 
$\langle
\tilde{\pi},h  \rangle_\pi \coloneqq \E^{\mathbb P^{x,\pi}}
\int_0^{\tau_\cO} \Gamma^\pi_t \int_A 
h(X_t,a) \tilde{\pi} \left(da|X_t\right)  dt$,
and $b\cM(A\mid \sR^d)$ is the space of bounded signed kernels. 
Heuristically,  the function 
\begin{equation}
\label{eq first variation intro}
(x,a)\mapsto \frac{\delta v_\tau^{{\pi}}}{\delta \pi}(x,a)  \coloneqq     (\overline{\mathcal{L}}^a v_\tau^{\pi})(x)+f (x,a)
+ \tau\ln\frac{  d \pi}{  d\mu}(a|x)
\end{equation}
 can be interpreted 
 as a  derivative (first variation) of $  v^\cdot_\tau$ at $\pi$ relative to the pairing $\langle
\cdot,\cdot  \rangle_\pi $ (see e.g., \cite{kerimkulov2023fisher}). 
 Now consider the following policy mirror descent update: 
 let $\pi_0\in \Pi_\mu$, and for all $n\in \sN\cup\{0\}$, 
given $\tau_n>0$, define  
 \begin{align*}
    \pi_{n+1} (da|x) 
    &=\argmin_{m\in\mathcal{P}(A)}\bigg(\int_A
        \frac{\delta v_{\tau_n}^{{\pi_n}}}{\delta \pi}(x,a) 
 m(da) + \frac{1}{\lambda}\KL(m|\pi_n)(x)\bigg)
 =\frac{e^{- {\lambda}    \frac{\delta v_{\tau_n}^{{\pi_n}}}{\delta \pi}(x,a) }}{\int_A e^{- {\lambda}    \frac{\delta v_{\tau_n}^{{\pi_n}}}{\delta \pi}(x,a')  }\pi_n(da'|x)} \pi_n(da|x)\,,
\end{align*}
which
 optimizes the first-order approximation of $\pi\mapsto v^\pi_{\tau_n}$ 
 around $\pi_n$, and uses the KL divergence 
  to ensure optimization within a sufficiently small domain.
 A straightforward computation shows that 
 this is equivalent to 
 setting $\pi_{n}=\boldsymbol{\pi}(Z_{n})$ 
 for all $n\in \sN$, 
 and  updating  $Z_n$  by   
 \begin{align*}
    \frac{Z_{n+1}(x,a)-Z_n(x,a)}{\lambda} =   -\left(\overline{\mathcal{L}}^av^{\pi_n}_{\tau_n}(x) + f(x,a)+\tau_n Z_n(x,a)\right) \,,
\end{align*}
from which, by interpolating and  letting 
  $\lambda\rightarrow 0$,  we obtain the flow  \eqref{eq:md_intro}.

The first variation of 
value function with respect to policies 
has been used to design policy mirror descent for discrete-time control problems 
in \cite{kerimkulov2023fisher}.
In that case, it is represented as the (discrete-time) $Q$-function along with the log-density of the policy.  
The first variation 
$  \frac{\delta v_\tau^{{\pi}}}{\delta \pi}$ in~\eqref{eq first variation intro}
is a continuous-time analogue of 
the representation in~\cite{kerimkulov2023fisher},
and 
the function
$(x,a)\mapsto (\overline{\mathcal{L}}^a v_\tau^{\pi})(x)+f (x,a)$ 
has been referred to as the $q$-function in \cite{jia2023q}.

 \subsubsection*{Our contributions}

This work analyzes    
the  mirror descent flow \eqref{eq:md_intro} with different  choices of  schedulers $\btau$. 
 
  \begin{itemize}
  \item 
We show that the flow \eqref{eq:md_intro} with a continuous scheduler $\btau$ admits  a unique  solution  when  the state process \eqref{eqn:state_intro} has nondegenerate noise (Theorem \ref  {thm:Existence_Solutions}).
  Moreover, 
    regularized value functions decrease along the flow   if $\btau$ is continuously differentiable and decreasing
(Theorem \ref{ref:cost_decrease_along_flow}).
  
 \item
We prove that for a constant   $\btau\equiv \tau$,  the solution to  \eqref{eq:md_intro}  converges at a global
exponential rate to the optimal solution of the $\tau$-regularized problem \eqref{eq:value_tau_intro} (Corollary \ref{cor:convergence_of_GF}).

\item 

We analyze the convergence of \eqref{eq:md_intro} to the unregularized problem \eqref{eq:value_optimal_intro} using a constant   
 $\btau$ chosen  based on a prescribed   running horizon  $S$ of \eqref{eq:md_intro}. The resulting flow yields   an error of 
$\cO(1/S)$  for discrete action spaces  (Theorem \ref{ex:discrete_action_space}), and under suitable conditions, achieves   a comparable rate for general action spaces, albeit with an additional logarithmic factor  (Theorem \ref{thm:conv_rate_continuous_constant}).
 
\item 

We examine the annealed flow \eqref{eq:md_intro} with a decaying $\btau$. For discrete action spaces, $\btau_s = 1/(s+1)$ achieves $\cO(1/S)$ convergence  to the unregularized problem \eqref{eq:value_optimal_intro} as  the  running horizion  $S\to \infty$ 
(Theorem  \ref{thm:conv_discrete_anneal}). For general action spaces, $\btau_s = 1/\sqrt{s+1}$ yields $\cO(1/\sqrt{S})$ convergence, up to a   logarithmic term (Theorem \ref{thm:conv_general_anneal}).

  \end{itemize} 
  
  To the best of our knowledge, this is the first theoretical work on the precise impact  of the entropy scheduler on the convergence rate of a PG method for continuous-time control problems.

 \subsubsection*{Our approaches and the importance of   entropy regularization}

The key idea of the convergence analysis of    \eqref{eq:md_intro} 
is to balance the optimization error and the regularization bias. 
In particular, let $(Z_s)_{s\ge 0}$ be the solution to  \eqref{eq:md_intro}  with a given  scheduler $\btau$, 
 we   decompose the error of  $\boldsymbol{\pi}(Z_s)$  into 
 \begin{equation}
\label{eq:error_decomposition}
0\le v^{{\boldsymbol{\pi}(Z_s)}}_0-v^*_0 = (v^{{\boldsymbol{\pi}(Z_s)}}_0-v^{{\boldsymbol{\pi}(Z_s)}}_{\btau_s})+(v^{{\boldsymbol{\pi}(Z_s)}}_{\btau_s} -v^*_{\btau_s})
+(v^*_{\btau_s}
- v^*_0)
\,,
\end{equation}
where
$v^*_{\btau_s}=\inf_{\pi\in \cP(A|\sR^d)}v^\pi_{\btau_s}$ is the optimal  ${\btau_s}$-regularized value function. 
The first term in \eqref{eq:error_decomposition} is negative due to the positivity of the KL divergence.
The second term in \eqref{eq:error_decomposition}  represents the optimization error of  \eqref{eq:md_intro} for a regularized problem.
The third term in \eqref{eq:error_decomposition} is the  regularization  bias resulting from    the additional KL divergence in  \eqref{eq:value_tau_intro}.

We 
establish explicit   bounds for  the 
optimization error and  regularization bias in terms of $\btau$,
and 
 optimize the overall error \eqref{eq:error_decomposition} by selecting appropriate   constant or time dependent 
schedulers 
$\btau$. 
For 
  the optimization error, we derive an explicit upper bound of   $v^{{\boldsymbol{\pi}(Z_s)}}_{\btau_s} -v^*_{\tau}$   in terms of the scheduler $\btau$ and $\tau>0$ (Proposition \ref{thm:convergence_of_GF} and Theorem \ref{cor:extend_conv_GF}). 
This is achieved by showing $s\mapsto \mathbb E^{\mathbb P^{x,\pi^\ast_\tau} } \int_0^{\tau_{\cO}} \Gamma^{\pi^*_\tau}_t \text{KL}(\pi^\ast_\tau | {\boldsymbol{\pi}(Z_s)} )( X_t)\,dt$ serves as a differentiable Lyapunov function,
with $\pi^*_\tau$ being the optimal $\tau$-regularized policy, 
and by overcoming the    non-convexity of $Z\mapsto v^{\boldsymbol{\pi}(Z)}_\tau$
using a performance difference lemma  (Lemma~\ref{lemma:performance_diff}).
We further  prove that   the regularization bias $v^*_{\tau}
- v^*_0$ vanishes as $\tau\to 0$ for any sufficiently exploring      reference measure $\mu\in\cP(A)$ (Theorem \ref{thm:convergece_tau}).
An explicit  decay rate in terms of $\tau$
 is identified by deriving  
  precise  asymptotic expansions of the regularized Hamiltonians.

It is worth pointing out that in cases  with continuous action spaces $A$, computing the policy gradient using a regularized value function in \eqref{eq:md_intro} is essential for  the convergence analysis of the   flow. 
Indeed, 
consider the flow with unregularized gradient ($\btau \equiv 0$ in   \eqref{eq:md_intro}): 
\begin{equation}
\label{eq:md_unregularized_intro}
\partial_s {Z_s}(x,a) =-\left( \overline{\mathcal{L}}^a v^{\boldsymbol{\pi}(Z_s)}_0 (x)+ f(x,a)    \right)  \,,\quad   (x,a)\in \cO\times A,\,s>0\,; \quad Z\vert_{s=0}=Z_0. 
\end{equation}
Suppose  that  the unregularized problem \eqref{eq:value_optimal_intro}
has an  optimal  policy $\pi^*_0$.
Formally   differentiating $s\mapsto \mathbb E^{\mathbb P^{x,\pi^\ast_0} } \int_0^{\tau_{\cO}} \Gamma^{\pi^*_0}_t \text{KL}(\pi^\ast_0 | {\boldsymbol{\pi}(Z_s)} )( X_t)\,dt$ suggests that  
 \begin{equation}
 \label{eq:error_unregulraised}
 v^{\boldsymbol{\pi}(Z_s)}_0  (x)-v^*_0 (x)\leq \frac{1}{s}
\mathbb E^{\mathbb P^{x,\pi^\ast_0} } \int_0^{\tau_{\cO}} \Gamma^{\pi^\ast_0}_t \text{KL}(\pi^\ast_0|{\boldsymbol{\pi}(Z_0)} )( X_t) \,dt, \quad s>0\,,
\end{equation}
which can be viewed as the limiting case of \eqref{eq:convergence_gf_constant}   as $\tau\to 0$ (see also \cite[Theorem 4.1]{zhang2021wasserstein}).
However, the error bound 
\eqref{eq:error_unregulraised} does not imply  \eqref{eq:md_unregularized_intro} converges with a rate   $\cO(1/s)$
as the constant in \eqref{eq:error_unregulraised}   is generally  infinite. 
In fact, since $\pi^*_0(\cdot|x)$ is typically   a Dirac measure and ${\boldsymbol{\pi}(Z_0)}\in \Pi_\mu$, 
it is   infeasible to prescribe  a reference measure $\mu$ such  that 
$\text{KL}(\pi^\ast_0|\boldsymbol{\pi}(Z_0) )(X_t)<\infty$ for all $t$. 

We overcome this difficulty by utilizing a regularized policy gradient in  \eqref{eq:md_intro}, which allows for introducing   the regularized value function as an intermediate step for the convergence analysis. By carefully balancing the optimization error and regularization bias, we achieve similar error bounds   up to a logarithmic term.

\subsection{Most related works}
Mirror descent is a well-established optimization algorithm, whose convergence has been  extensively analyzed for static optimization problems over both Euclidean spaces \cite{bubeck2015convex, beck2003mirror, lu2018relatively} and spaces of measures \cite{aubin2022mirror, nitanda2021particle}. Recently, mirror descent has been adapted to design policy optimization algorithms for discrete-time MDPs, owing to its equivalence with the natural policy gradient method \cite{raskutti2015information}. 
It has been shown that  policy mirror descent achieves linear convergence to the global optimum for MDPs whose action spaces are finite  sets  \cite{lan2023policy, zhan2023policy, alfano2024novel}, subsets of  Euclidean spaces \cite{lan2022policy}, and  general Polish   spaces \cite{kerimkulov2023fisher}.

For continuous-time control problems,
most existing works
on PG algorithms
focus on algorithm design rather than convergence analysis. 
The primary approach involves applying existing algorithms for MDPs after discretizing both time and state spaces, and then sending     discretization parameters  to zero  \cite{munos2006policy, munos1998reinforcement, munos2000study, han2016deep, pham2022mean, hamdouche2022policy}. 
Recently, \cite{ jia2022policy, zhao2024policy} extend PG   methods  to continuous-time problems without time and space discretisation, in order to develop algorithms that are robust across different time and spatial discretization scales. Yet, very little is known regarding the convergence rate  of   these algorithms.

In fact, convergence analysis of gradient-based algorithms for continuous-time control problems is fairly limited. Works such as \cite{siska2024gradient,sethi2024modified, kerimkulov2022convergence} have established convergence rates for gradient flows involving open-loop controls, which are functions of  the system's underlying noise.
Incorporating open-loop controls  avoid the complexities of nonlinear feedback in state processes, thereby substantially simplifying the analysis. 
For Markov controls, existing studies  typically rely on uniform derivative estimates of policy   iterates to assure algorithmic convergence. For example, \cite{giegrich2024convergence} analyzes natural PG methods  for LQ control problems and demonstrates the uniform Lipschitz continuity of policies by leveraging the inherent LQ structure. 
This analysis is extended to nonlinear drift control problems by \cite{reisinger2023linear}, where similar Lipschitz estimates are established  under sufficiently convex cost functions. 
Moreover,    \cite{zhou2023policy, zhou2024solving} establish the convergence of PG methods for general control problems  under the a-priori assumption that the policy iterates have  uniformly bounded derivatives up to the fourth order, although they   do not provide conditions that guarantee these prerequisites.

In contrast to previous works,  this work introduces a weak formulation of the control problem, which facilities  working with merely measurable policies. This eliminates the need for uniform derivative estimates as required in previous works and allows for considering control problems with more irregular coefficients, broadening the applicability   of our analysis.

\subsection{Notation}

Given topological  spaces $E_1$ and $E_2$,
we denote by $\mathcal{B}(E_1;E_2)$ the space of Borel measurable functions $\phi:E_1\to E_2$,
and by $B_b(E_1; \sR^k)$ the space of bounded Borel measurable functions $\phi:E_1\to \sR^k$ equipped with the supremum norm $\|\phi\|_{B_b(E_1)}=\sup_{x\in E_1}|\phi(x)|$. 
We will write $B_b(E_1)=B_b(E_1; \sR^k)$ when the range is clear from the context.
Given $E\subset\mathbb{R}^n$ and $p\in[1,\infty)$ 
let $L^p(E)$ be the space of Borel measurable functions $f:E\rightarrow\mathbb{R}$ such that $\int_E|f|^pdx<\infty$
equipped with the norm $\|f\|_{L^p(E)}
\coloneqq \left(\int_E|f|^pdx\right)^{\frac{1}{p}}$.
Also for $k\in\mathbb{N}$ and $p\in[1,\infty)$ 
let $W^{k,p}(E)$ be the space 
of all Borel measurable $f:E\rightarrow\mathbb{R}$ 
whose generalized derivatives up to order $k$ exists and are in $L^p(E)$ and is equipped with the norm $\|f\|_{W^{k,p}(E)}=\left(\sum_{|\alpha|\leq k}\|D^\alpha f\|_{L^p(E)}^p\right)^{\frac{1}{p}}$ where $\alpha$ is a multi-index and $D^\alpha$ the generalized derivative. 

For a given domain 
$\cO\subset\mathbb{R}^d$ and $k\in \mathbb{N}$, $C^{k}(E)$ denotes the space of $k$-times continuously differentiable functions in $\overline{\cO}$
equipped with the norm $\|f\|_{C^k(\overline{\cO})}=\sum_{i=0}^k\|D^if\|_{C^0(\overline{\cO})}$, where 
$\|f\|_{C^0(\overline{\cO})}=\sup_{x\in\overline{\cO}}|f(x)|$.
Given $\alpha\in(0,1)$ and a function $u:\cO\rightarrow\mathbb{R}$
we define the H\"older semi-norm $[u]_{\alpha}=\sup_{x,y\in\overline{\cO}}\frac{|u(x)-u(y)|}{|x-y|^\alpha}$ and the  H\"older space $C^{k,\alpha}(\overline{\cO})$ 
is the space of all functions in $C^k(\overline{\cO})$ such that $\|u\|_{C^{k,\alpha}(\overline{\cO})}=\|u\|_{C^k(\overline{\cO})}+[D^ku]_{\alpha}<\infty$.

Given normed vector spaces $(X,\|\cdot\|_Y)$ and $(Y,\|\cdot\|_Y)$ we denote by $\mathcal{L}(X,Y)$ the space of bounded linear operators $T:X\rightarrow Y$ equipped with the operator norm $\|T\|_{\mathcal{L}(X,Y)}=\sup_{\|x\|_X\leq 1}\|Tx\|_Y$.
Given a Banach space $(X, \|\cdot\|_X)$ and a constant $S>0$, let $C^1([0,S];X)$
be the space of continuously  (Fr\'{e}chet) differentiable functions $f:[0,S]\rightarrow X$
equipped with the norm
$\|f\|_{C^1([0,S];X)}=\sup_{s\in[0,S]}\|f(s)\|_X+\sup_{s\in[0,S]}\left\|\frac{d}{ds}f(s)\right\|_X $.

For a given $E\subset\mathbb{R}^d$ we denote by $\mathcal B(E)$ the Borel $\sigma$-algebra and by
$\mathcal{M}(E)$ the space all finite signed measures $\mu$ on $E$ endowed with the total variation norm $\|\mu\|_{\mathcal{M}(E)}=|\mu|(E) = \mu^+(E) + \mu^-(E)$ where for any $B\in \mathcal B(E)$ we have $\mu^+(B) \coloneqq  \sup_{A\in \mathcal B(E), A\subset B} \mu(A) $ and $\mu^-(B) \coloneqq  -\inf_{A\in \mathcal B(E), B \subset A}\mu(A)$. 
Denote by $\mathcal{P}(E)\subset\mathcal{M}(E)$ the space of all probability measures on $E$, again endowed with the total variation norm.
Given $\mu,\nu\in\mathcal{P}(E)$ we write $\nu\ll\mu$ if $\nu$ is absolutely continuous with respect to $\mu$ and define the Kullback--Liebler (KL) divergence of $\nu$ with respect to $\mu$ by $\KL(\nu|\mu)=\int_A\ln\frac{d\nu}{d\mu}(x)\nu(dx)$ if $\nu\ll\mu$ and $+\infty$ otherwise. 

Given $E_1\subset\mathbb{R}^d$ and a separable metric space $(E_2,d_2)$,
$b\mathcal{M}(E_2|E_1)$ denotes the Banach space of bounded signed kernels
$\pi:E_1\rightarrow\mathcal{M}(E_2)$ 
endowed with the norm 
$\|\pi\|_{b\mathcal{M}(E_2|E_1)}=\sup_{x\in E_1}\|\pi(x)\|_{\mathcal{M}(E_2)}$,
i.e.
for each $B\in\mathcal{B}(E_1)$ the map $x\mapsto \pi(B|x)$ is measurable and for each fixed $x$, $\pi(da|x)\in\mathcal{M}(E_1)$.
For a fixed positive measure 
$\mu\in\mathcal{M}(E_2)$ and $\pi\in b\mathcal{M}(E_2|E_1)$ we will write $\pi\ll\mu$ if for each $x\in E_1$, $\pi(\cdot|x)\ll\mu$.
For $\pi,\pi' \in b\mathcal M(E_2|E_1)$ and $x\in E_1$ we define $\operatorname{KL}(\pi|\pi')(x) \coloneqq  \operatorname{KL}(\pi(\cdot|x)|\pi'(\cdot|x))$ and note that $E_1\ni x\mapsto \operatorname{KL}(\pi|\pi')(x)\in \mathbb R \cup \{+\infty\}$ is measurable.
Let $\mathcal{P}(E_1|E_2)\subset b\mathcal{M}(E_1|E_2)$ be the space of bounded probability kernels,
that is $\pi\in b\mathcal{M}(E_1|E_2)$ such that $\pi(d\bar{x}|x)\in\mathcal{P}(E_1)$ for all $x\in E_2$.

\section{Problem formulation and main results}
\label{sec:main_results}

This section summarizes the model assumptions and presents the main results.

\subsection{Relaxed control problem}

The following standing  assumptions  on the coefficients are imposed throughout this paper.

\begin{assumption}\label{ass:data}
Let $A$ be a separable metric space,
$d,d'\in \sN$,  
and
let $\cO\subset \sR^d$ be a bounded domain  (i.e. a connected open set)   
whose boundary $\partial \cO$ is of the class $C^{1,1}$. 
Let 
$b\in B_b(\sR^d\times A; \sR^d)$,
$c\in B_b(\sR^d\times A  ; [0,\infty) )$, $f\in B_b(\sR^d\times A;\sR)$
and
let $\sigma \in  B_b( \sR^d; \sR^{d\times d'})\cap C(\sR^d; \sR^{d\times d'})$ 
satisfy $ \lambda\coloneqq \inf_{x\in \sR^d , u\in \sR^{d'}\setminus \{0\}}\frac{|\sigma (x)u|^2}{|u|^2}>0$. 
Let    
$g\in W^{2,p^*}(\cO)$ with some $p^*\in (d,\infty)\cap [2,\infty)$.

\end{assumption}

Under Assumption \ref{ass:data},  
we consider a weak formulation of the exit time problem.
For each $x\in \sR^d$ and $\pi\in \cP(A|\sR^d)$, 
let
the state process  $X^{x,\pi}$ be the unique weak solution 
to
\begin{align}
\label{eqn:state}
dX_t=\left(\int_A b(X_t,a)\pi(da | X_t)\right)dt+\sigma(X_t)dW_t,\enspace t> 0; \quad X_0=x,
\end{align}
where    $(W_t)_{t\ge 0}$ is a $d'$-dimensional   standard Brownian motion
defined      on a  filtered probability space $(\Omega,\mathcal{F},\mathbb{F},\mathbb{P}^{x,\pi})$. 
Since the diffusion coefficient $\sigma$ is non-degenerate,
the weak solution $X^{x,\pi}$ to \eqref{eqn:state} exists  and  is unique in the sense of probability law 
(see  e.g.,~\cite[Theorem 7.2.1]{stroock1997multidimensional}).
Let $\tau_{\mathcal O}\coloneqq \inf\{t\ge 0\mid X_t^{x,\pi}\not \in \cO\}$ be  the first exit time of  $X^{x,\pi}$ from the domain $\cO$,
and let 
$\Gamma^{\pi}=(\Gamma^{ \pi}_t)_{t\in [0,\tau_\cO)}$ 
be  the controlled discount factor given by
$\Gamma^{\pi}_t \coloneqq \exp\left(-\int_0^t \int_A c(X_s^{x,\pi},a)\pi(da | X_s^{x,\pi}) ds \right)$.
We define the value function 
\begin{equation}
\label{eq:value}
v^\pi_\tau(x)
\coloneqq \E^{\sP^{x,\pi}}\left[\int_0^{\tau_\cO} \Gamma^\pi_t \left(\int_Af(X_t,a)\pi(da | X_t) + \tau\textrm{KL}(\pi |\mu)(X_t )\right)dt+ \Gamma^\pi_{\tau_\cO}  g(X_{\tau_\cO})\right]\,,
\end{equation}
where 
$\tau\geq 0$ is a given regularising weight and
$\mu \in \cP(A)$ is a given reference measure.
Note that  $v^\pi_\tau (x) $ in \eqref{eq:value} is a well-defined extended-real number, due to the boundedness of $c$, $f$ and $g$,  
$\operatorname{KL}(\nu|\mu)\ge 0$
and $\sE^{\sP^{x,\pi}}[\tau_{\cO}]<\infty$ 
(see~\cite[Ch.~2, Sec.~2, Theorem 4, p.~54]{krylov2008controlled}).
Define the optimal value function 
$v^*_\tau:\overline{\cO}\to \sR\cup\{\infty\}$ 
by
\begin{align}
\label{eq:value_optimal}
v^*_\tau(x)
\coloneqq &\inf_{\pi\in \cP(A|\sR^d) }  v^\pi_\tau(x)\,.
\end{align}

To facilitate the presentation, we provide a unified formulation of exit time control problems for any regularization parameter $\tau \geq 0$. As alluded to in Section \ref{sec:intro}, our goal is to analyze the convergence of the mirror descent flow \eqref{eq:md_intro}, guided by an appropriately defined   entropy scheduler $\btau:[0,\infty)\to (0,\infty)$,
to the unregularized problem  \eqref{eq:value_optimal} with $\tau=0$ (i.e.,  \eqref{eq:value_optimal_intro}). This analysis will be conducted in three   steps: 
(1) establishing the well-posedness of~\eqref{eq:md_intro}; 
(2) quantifying the convergence rate of~\eqref{eq:md_intro} to the optimal solution of the regularized problem; (3) quantifying the resulting regularization bias and optimizing the total error over the regularization weight.

\subsection{Well-posedness of the mirror descent flow}

We start by showing that the flow  \eqref{eq:md_intro} admits a unique solution for any   entropy scheduler $\btau\in C([0,\infty); (0,\infty))$. The essential step is to analyze the regularity of the
nonlinearity  $Z\mapsto \overline{\cL}^a v^{\boldsymbol{\pi}(Z)}_\tau$ in the flow  \eqref{eq:md_intro}.

To this end, recall that     the class $\Pi_\mu$ of  Gibbs policies is defined by:  
\begin{align}\label{eqn:policy}
\Pi_{\mu}\coloneqq \left\{\pi\in \cP(A|\sR^d)\mid \pi=\boldsymbol{\pi}(Z)\; \textnormal{for some $Z\in B_b(\cO\times A)$} \right\}\,,
\end{align}
with  
$\boldsymbol{\pi}:B_b(\cO\times A)\to \cP(A|\sR^d)$ given by  
\begin{align}
\label{eq:operator_pi}
\boldsymbol{\pi}(Z)(da|x) \coloneqq 
\frac{e^{Z(x,a)}}{\int_Ae^{Z(x,{a}')}\mu(d{a}')}\mu(da),\quad x\in \cO;
\quad  
\boldsymbol{\pi}(Z)(da|x)\coloneqq  \mu(da),\quad x\not\in \cO\,. 
\end{align}   
As it suffices to determine the policy for the state variable  inside the domain $\cO$,  
we simply extend the policy outside the domain   by $\mu$. 
For each $\pi\in \Pi_\mu$, we introduce the so-called on-policy-Bellman equation:
\begin{align}
\label{eq:on_policy_bellman}
\int_A  \Big((\mathcal{L}^a v)(x)
+f(x,a)\Big)\,\pi(da|x) + \tau \textrm{KL}(\pi |\mu)(x) =0,  \,\,\, \textnormal{a.e.~$x\in \cO$};  \,\,\, 
v(x)=g(x),  \,\,\, x \in \partial \cO\,,
\end{align}
where
for each $a\in A$,   $\mathcal{L}^a:W^{2,p^*}(\cO)\to L^{p^*}(\cO)$ is the operator defined by
\begin{align}
\label{eq:generator}
(\mathcal{L}^a v)(x) = \frac{1}{2}\tr(\sigma(x)\sigma(x)^\top D^2v(x) )+b(x,a)^\top Dv(x) -c(x,a)v(x)\,.
\end{align}

The following proposition characterizes the regularized value function \eqref{eq:value}
as the solution of \eqref{eq:on_policy_bellman}.

\begin{proposition} 
\label{prop:Bellman_PDE_wellposedness}
Suppose Assumption~\ref{ass:data} holds and $\tau>0$.
Let $\pi\in \Pi_{\mu}$,
and let $v^{\pi }_\tau$ be the associated value function given by~\eqref{eq:value}. 
Then $v^{\pi}_\tau$ satisfies the Dirichlet problem~\eqref{eq:on_policy_bellman}, $v^\pi_\tau\in W^{2,p^*}(\cO)$ 
with $p^*$ from Assumption~\ref{ass:data}, and 
$\tr(\sigma \sigma^\top D^2 v^{\pi}_\tau) \in L^\infty(\cO)$.

\end{proposition}

The proof of  Proposition \ref{prop:Bellman_PDE_wellposedness} is given in Appendix~\ref{sec:well-posedness of HJB}.
It follows from standard regularity results of linear PDEs and  It\^{o}'s formula  for Sobolev functions \cite[Theorem 1, p.~122]{krylov2008controlled}. 
Note that Assumption~\ref{ass:data}
only assumes the drift \( b \) and cost function \( f \)   to be measurable, and hence 
 we do not expect \( v^\pi_\tau \) to exhibit \( C^2 \) interior regularity.

The following theorem proves the well-posedness of \eqref{eq:md_intro} with a continuous scheduler.

\begin{theorem}
\label{thm:Existence_Solutions}
Suppose Assumption \ref{ass:data} holds.
For each   $Z_0\in B_b(\cO\times A)$ and
$\boldsymbol{\tau}\in C([0,\infty); (0,\infty))$, 
there exists a unique $Z\in \cap_{S>0}C^{1}([0,S];B_b(\cO\times A))$  satisfying~\eqref{eq:md_intro}.
\end{theorem}

The proof of Theorem \ref{thm:Existence_Solutions} is given in   Appendix~\ref{sec:existence_of_sols}.
The argument  begins by  leveraging elliptic PDE theory to prove  that the map 
$B_b(\cO\times A)\ni Z\mapsto \overline{\cL}^\cdot v^{\boldsymbol{\pi}(Z)}_\tau\in B_b(\cO\times A)$
is    locally Lipschitz continuous (Proposition \ref{prop:W^{2,p} bound}) and of linear growth (Lemma \ref{lem:DV_s_bound}).
A priori estimate further shows that  a   solution of \eqref{eq:md_intro}
will not blow up on a   finite interval, hence 
applying a truncation argument and  the Banach fixed point theorem on each finite interval
yield the desired conclusion.

\subsection{Convergence of mirror descent for the regularized problem}

We proceed to quantify the accuracy  of   \eqref{eq:md_intro}
for approximating a regularized problem.
In particular, we shall prove that  \eqref{eq:md_intro} with a constant scheduler $\btau\equiv \tau$ converges exponentially to  the $\tau$-regularized problem  \eqref{eq:value_optimal}.

We first prove that
the   regularized value function decreases  along the flow   \eqref{eq:md_intro},
if  the entropy scheduler $\btau$ is continuously differentiable and decreases in time.  

\begin{theorem}
\label{ref:cost_decrease_along_flow}
Suppose Assumption~\ref{ass:data} holds.
Let $Z_0\in B_b(\cO\times A)$, 
$\btau\in C^1([0,\infty);(0,\infty))$,
and 
$Z\in \cap_{S>0}C^{1}([0,S];B_b(\cO\times A))$  be the solution to  \eqref{eq:md_intro}.
Then for all $x\in \cO$, 
$[0,\infty) \ni s\mapsto  v^{\boldsymbol{\pi}(Z_s)}_{\btau_s}(x)\in\mathbb{R}$ is differentiable and
for all $s>0$,
\begin{equation}
\label{eq:decreasing_along_Z}
\begin{split}
&\partial_s v^{\boldsymbol{\pi}(Z_s)}_{\btau_s}(x)
\\
&=-\E^{\mathbb{P}^{x,\boldsymbol{\pi}(Z_s)}}\int_0^{\tau_\cO}\Gamma^{\boldsymbol{\pi}(Z_s)}_t\int_A\left(\mathcal{L}^av_{\btau_s}^{\boldsymbol{\pi}(Z_s)}(X_t)+f(X_t,a)+\btau_s \ln\frac{d\boldsymbol{\pi}(Z_s)}{d\mu}(a|X_t)\right)^2\boldsymbol{\pi}(Z_s)(da|X_t)dt\\
&\enspace\enspace+(\partial_s\btau_s)\E^{\mathbb{P}^{x,\boldsymbol{\pi}(Z_s)}}\int_0^{\tau_{\cO}}\Gamma^{\boldsymbol{\pi}(Z_s)}_t\KL(\boldsymbol{\pi}(Z_s)|\mu)(X_t)dt\,.
\end{split}
\end{equation}
Consequently, if $\btau\in C^1([0,\infty);(0,\infty))$ is decreasing, then $\partial_s v^{\boldsymbol{\pi}(Z_s)}_{\btau_s}(x)\le 0$ for all $s>0$ and $x\in \cO$. 
\end{theorem}

The proof of Theorem~\ref{ref:cost_decrease_along_flow} is given in Section~\ref{sec:convergence}.
It relies on the Hadamard differentiability of the map 
$(Z,\tau) \mapsto v^{\boldsymbol{\pi}(Z )}_{\btau }(x)$ established in Proposition \ref{prop:funct_derivative_rho}.

We then quantify the error 
$ v^{\boldsymbol{\pi}(Z_s)}_{\btau_s}(x)-v^*_{\tau}(x)$
for any given $\tau>0$. 
To this end, for each $\tau>0$, let  
$H_\tau : \cO \times   \sR\times \sR^{d}\to \sR$ be the regularized Hamiltonian  
such that for all 
$(x,u, p )\in \cO \times   \sR\times \sR^{d}$, 
\begin{align}
\begin{split}
\label{eq:Hamiltonian_tau}
H_\tau(x,u,p) &\coloneqq  \inf_{m\in \mathcal P(A)} \left(\int_A [b(x,a)^\top p - c(x,a)u+f(x,a)]\,m(da)+\tau\operatorname{KL}(m|\mu)  \right)
\\
&=-\tau \ln\left(\int_A \exp\left(-\frac{b(x,a)^\top p -c(x,a)u+f(x,a)}{\tau }\right) \mu(da)\right)\,. 
\end{split}
\end{align}
The  Hamilton-Jacobi-Bellman (HJB) equation 
associated to the regularized problem is given by:
\begin{equation}
\label{eq:semilinear}
\frac{1}{2}\tr(\sigma(x)\sigma(x)^\top D^2v(x) )+ H_\tau (x,v(x),Dv(x)) =0,\quad \textnormal{a.e.~$x\in \cO$}\, ;  \quad 
v(x)=g(x),  \quad x \in \partial \cO\,.
\end{equation}

The following proposition characterizes the optimal  regularized value function
and the optimal regularized policy using 
the solution to \eqref{eq:semilinear}.

\begin{proposition}
\label{prop:verification}
Suppose Assumption~\ref{ass:data} holds and $\tau>0$.
Then 
\eqref{eq:semilinear} admits a unique solution $v \in  W^{2,p^*}(\cO)$ with $p^*$ as in Assumption~\ref{ass:data},
and $v (x)=v^*_\tau (x) $ for all $x\in \overline{\cO}$,
where  $v^*_\tau $ 
is  the optimal value function 
defined    in \eqref{eq:value_optimal}. 
Moreover, if  $\pi^*_\tau \in \cP(A|\sR^d)$ satisfies  for all $x\in \cO$ that
\[
\pi^*_\tau (da|x) = \frac{e^{-\frac{1}{\tau  } Z^*_\tau (x,a)  }}{\int_A  e^{ -\frac{1}{\tau  } Z^*_\tau (x,a')  } \mu(da') } \mu(da)\,\quad 
\textnormal{with $ Z^*_\tau (x,a) \coloneqq  b(x,a)^\top  Dv^*_\tau (x) - c(x,a)v^*_\tau(x)+f(x,a)$},
\]
then
$\pi^*_\tau$ is an optimal policy of   \eqref{eq:value_optimal} in the sense that 
$v^{\pi^*_\tau}_\tau (x) = v^*_\tau (x) $ for all $x\in \overline{\cO}$. 
\end{proposition}

The proof of Proposition
\ref{prop:verification} is given in Appendix~\ref{sec:well-posedness of HJB}.
The crucial step is establishing that \eqref{eq:semilinear} admits a unique solution in $W^{2,p^*}(\cO)$. 
While the well-posedness of semilinear HJB equations in Sobolev spaces has been  examined in~\cite[Section 3.2.3]{bensoussan1984}, the analysis therein assumes the cost function is uniformly bounded over all actions, and the discount factor is strictly positive.   These conditions are not fulfilled by \eqref{eq:semilinear} since $\cP(A)\ni m\mapsto \textrm{KL}(m|\mu)\in \sR\cup \{\infty\}$ is unbounded, and the discount factor $c$ can be zero.
In Appendix~\ref{sec:well-posedness of HJB}, 
we   provide a self-contained proof based on 
the Leray--Schauder Theorem~\cite[Theorem 11.3]{gilbarg1977elliptic}.

We  now   state the   error bound  of  \eqref{eq:md_intro}
for a given regularized problem, whose proof is given   in Section~\ref{sec:convergence}.

\begin{proposition}
\label{thm:convergence_of_GF}
Suppose Assumption \ref{ass:data} holds. 
Let   $Z_0\in B_b(\cO\times A)$,   let
$\btau\in C^1([0,\infty);(0,\infty))$ be decreasing, and 
let $Z\in \cap_{S>0}C^1([0,S];B_b(\cO\times A))$ be the solution to \eqref{eq:md_intro}. Then for all $s>0$, $x\in \overline\cO$ and $\tau>0$,
\begin{align}
\label{eq:convergence_gf}
\begin{split}
v^{{\boldsymbol{\pi}(Z_s)}  }_{\btau_{s}}(x)-v^{*}_{{\tau}}(x)
&\leq
\frac{1}
{\int_0^s e^{\int_0^{s'}\btau_{r}dr} d{s'}}\E^{\mathbb{P}^{x,\pi^*_{{\tau}}}}\int_0^{\tau_\cO}\Gamma^{\pi^*_{{\tau}}}_t\KL(\pi^*_{{\tau}}|\boldsymbol{\pi}(Z_0))(X_t)dt
\\
&\quad 
+\frac{\int_0^s(\btau_{s'}-{\tau})^+ e^{\int_0^{s'}\btau_{r}dr} ds'}{\int_0^s e^{\int_0^{s'}\btau_{r}dr} ds'}\E^{\mathbb{P}^{x,\pi^*_{\tau}}}\int_0^{\tau_\cO}\Gamma^{\pi^*_{{\tau}}}_t\KL(\pi^*_{{\tau}}|\mu)(X_t)dt\,.
\end{split}
\end{align}

\end{proposition}

Proposition~\ref{thm:convergence_of_GF} quantifies the precise   impact  of  a time-dependent scheduler $\btau$
on the convergence rate of   \eqref{eq:md_intro}.
To see it, suppose that one aims to solve the regularized problem with a fixed 
$\tau>0$.
In this case,   the first term in the estimate \eqref{eq:convergence_gf} represents the acceleration resulting from using a higher entropy regularization in   \eqref{eq:md_intro},
while the second term in the estimate \eqref{eq:convergence_gf} quantifies the error  caused by using a time-dependent   scheduler.

Note that 
the error bound in 
Proposition~\ref{thm:convergence_of_GF} depends on  the integrated KL divergence up to the exit time $\tau_{\cO}$. 
By further analyzing the behavior of the KL divergence in relation to     $\tau$
and the cardinality of action space $A$,
the  following theorem
provides a more explicit upper bound on 
$v^{{\boldsymbol{\pi}(Z_s)}  }_{\btau_{s}} -v^{*}_{{\tau}}$
in terms of   the entropy scheduler $\btau$.
It will be used to optimize the scheduler $\btau$  
for the unregularized problem. 

\begin{theorem}
\label{cor:extend_conv_GF}
Suppose Assumption~\ref{ass:data} holds, and let   $Z_0\in B_b(\cO\times A)$.
Then there exists  $C>0$ such that 
for all decreasing $\btau\in C^1([0,\infty);(0,\infty))$,
the solution $Z\in \cap_{S>0}C^1([0,S];B_b(\cO\times A))$   to \eqref{eq:md_intro}
satisfies for all 
$s>0$, $x\in \overline \cO$ and $\tau>0$,
\begin{align}
\label{eqn:modified_conv}
\begin{split}
v^{{\boldsymbol{\pi}(Z_s)}  }_{\btau_{s}}(x)-v^{*}_{{\tau}}(x)
&\leq
\frac{C}{\tau }\left( \frac{1+\tau}
{\int_0^s e^{\int_0^{s'}\btau_{r}dr} d{s'}}
+\frac{\int_0^s(\btau_{s'}-{\tau})  e^{\int_0^{s'}\btau_{r}dr} ds'}{\int_0^s e^{\int_0^{s'}\btau_{r}dr} ds'} \right)\,.
\end{split}
\end{align}
Assume further that  $A$ is of finite cardinality. Then  
for all
$x\in \mathcal O$ and $s>0$,
\begin{align}
\label{eqn:modified_conv_fin_A}
\begin{split}
v^{{\boldsymbol{\pi}(Z_s)}  }_{\btau_{s}}(x)-v^{*}_{{\tau}}(x)
&\leq
C \left( \frac{1 }
{\int_0^s e^{\int_0^{s'}\btau_{r}dr} d{s'}}
+\frac{\int_0^s(\btau_{s'}-{\tau})  e^{\int_0^{s'}\btau_{r}dr} ds'}{\int_0^s e^{\int_0^{s'}\btau_{r}dr} ds'} \right)\,.
\end{split}
\end{align}

\end{theorem}

The proof of Theorem~\ref{cor:extend_conv_GF}  is given   in Section~\ref{sec:convergence}.
Note that for general action spaces, the upper bound \eqref{eqn:modified_conv}
with $\btau\equiv \tau$ 
blows up as $ {\tau}\to 0$. 
In other words, there is no  uniform polynomial or exponential convergence  
rate  of $v_\tau^{\boldsymbol{\pi}(Z_s)}(x)-v^*_\tau (x)$     with respect to $s>0$ and $\tau>0$ (cf.~\eqref{eq:convergence_gf_constant}).

In fact, the following corollary shows that \eqref{eq:md_intro} with a constant scheduler converges exponentially  
to the regularized value function,   extending \cite[Theorem 2.7]{kerimkulov2023fisher} to the present continuous-time setting. 
The proof follows directly from  
Proposition \ref{thm:convergence_of_GF}.
Exponential convergence of policies to the optimal regularized policy can be established using 
similar arguments as in 
\cite{kerimkulov2023fisher}.

\begin{corollary}
\label{cor:convergence_of_GF}
Suppose Assumption \ref{ass:data} holds. Let  $\tau > 0$, $Z_0\in B_b(\cO\times A)$,
and  
$Z \in \bigcap_{S>0}C^1([0,S];B_b(\cO\times A))$ be the solution to \eqref{eq:md_intro}
with $\btau \equiv \tau$.
Then for all  $s>0$ and $x\in \overline\cO$,
\begin{align}
\label{eq:convergence_gf_constant}
0\leq v_\tau^{\boldsymbol{\pi}(Z_s)}(x)-v^*_\tau (x)\leq\frac{\tau}{e^{\tau s} - 1}\E^{\mathbb{P}^{x,\pi^*_\tau}}\int_0^{\tau_\cO} \Gamma^{\pi^*_\tau}_t\KL(\pi_\tau^*|\boldsymbol\pi(Z_0))(X_t)\,dt\,.
\end{align}
\end{corollary}

\subsection{Convergence of mirror descent with constant schedulers}
\label{sec:md_constant}

In this section, we  characterize the convergence rate of the flow  \eqref{eq:md_intro}  to  the unregularized problem \eqref{eq:value_optimal_intro} by employing an appropriately chosen constant scheduler $\btau$, whose value is determined based on the desired accuracy. 
The key step is to  quantify   the regularization bias $v^*_\tau-v^*_0$ for any fixed   $\tau>0$.

To this end, define the  unregularized Hamiltonian   
$H : \cO \times \sR\times \sR^{d}\to \sR$ such that  for all   
$(x,u, p )\in \cO \times   \sR\times \sR^{d}$, 
\begin{equation}
\label{eq:Hamiltonian_unregularized}
H (x,u,p)\coloneqq  \inf_{a\in A} \left( b(x,a)^\top p -c(x,a)u+f(x,a)  \right)\,. 
\end{equation}
Note that under Assumption~\ref{ass:data}, 
the function $H$ is well-defined due to the boundedness of  $b, c$ and $f$,
but $H$ may not   be Borel measurable  (see e.g., \cite{bertsekas1996stochastic}).
Assume further that  $H$ is 
Borel measurable,  then the HJB for the unregularized control problem \eqref{eq:value_optimal} (with $\tau=0$) is given by 
\begin{align}
\label{eq:hjb_unregualarised_1}
\inf_{a\in A}  \Big((\mathcal{L}^a v)(x)
+f(x,a)\Big)  =0,  \quad \textnormal{a.e.~$x\in \cO$};  \quad 
v(x)=g(x),  \quad x \in \partial \cO\,,
\end{align}
which can be equivalently written as 
\begin{align}
\label{eq:hjb_unregualarised}
\frac{1}{2}\tr(\sigma(x)\sigma(x)^\top D^2v(x) )+ H (x,v(x),Dv(x)) =0,  \quad \textnormal{a.e.~$x\in \cO$};  \quad 
v(x)=g(x),  \quad x \in \partial \cO\,.
\end{align}

We now provide sufficient conditions under which we prove the Borel  measurability of $H$,    characterize the optimal unregularized value function $v^*_0$ as the solution to  \eqref{eq:hjb_unregualarised},
and establish  the convergence of   $(v^*_\tau)_{\tau>0}$ 
to  $v^*_0$  as $\tau \to 0$. 

\begin{assumption}\label{assum:lsc}
\begin{enumerate}[(1)]
\item \label{item:measurability}
$A$ is a nonempty, compact and separable metric space. For all $x\in \cO $,
$ b(x,\cdot)$,  $c(x,\cdot)$  and $f(x,\cdot) $ are continuous on $A$.
\item \label{item:dense}
If a set $\mathcal C\subset A$ satisfies  $\mu(\mathcal C)=1$, then $\mathcal C$ is dense in $A$. 

\end{enumerate}

\end{assumption} 

Assumption \ref{assum:lsc} Item \ref{item:dense} requires the  reference measure $\mu$ in \eqref{eq:value} to   explore the entire action space $A$.
This  condition holds 
if $\mu$ assigns a nonzero probability to any nonempty open ball in $A$, and is satisfied by commonly used references measures such as   uniform measures for discrete action spaces  \cite{mei2020global,reisinger2021regularity}, and    Gaussian measures \cite{giegrich2024convergence}  and  uniform measures \cite{jia2022policy} for continuous action spaces.

\begin{theorem}
\label{thm:convergece_tau}
Suppose  Assumptions  \ref{ass:data}  and \ref{assum:lsc} hold. Then
$v^*_0$ given by~\eqref{eq:value_optimal} is the unique solution to \eqref{eq:hjb_unregualarised}
in  $  W^{2,p^*}(\cO)$ with $p^*$ as in Assumption~\ref{ass:data},
and there exists $C\ge 0$ such that for all $\tau>0 $, 
\[
0 \le  v^{\pi^*_\tau}_0  - v^*_0  \le {   v_\tau^*  - v_0^\ast  \leq }  C\| \left( H_\tau(\cdot, v^*_0(\cdot), Dv^*_0(\cdot))- H(\cdot, v^*_0(\cdot), Dv^*_0(\cdot))\right)^+\|_{L^{p^*}(\cO)}\,.
\]
Moreover, $\lim_{\tau \to 0}\| ( H_\tau(\cdot,  {v}^*_0(\cdot), D {v}^*_0(\cdot))- H(\cdot,  {v}^*_0(\cdot), D {v}^*_0(\cdot)))^+\|_{L^{p^*}(\cO)}=0$
and consequently, $\lim_{\tau \to 0} v^{\pi^*_\tau}_0 = v^*_0$ uniformly on $\overline{\cO}$. 
\end{theorem}

The proof of Theorem~\ref{thm:convergece_tau} is given in Section~\ref{sec:van_reg_proofs}.

Theorem \ref{thm:convergece_tau}   indicates that 
for a sufficiently small $\tau>0$,
the optimal regularized policy $\pi^*_\tau$ is   $\varepsilon$-optimal    for the unregularized problem \eqref{eq:value_optimal_intro}.
To the best of our knowledge, 
this is the first time such a consistency result has been established for general action spaces and reference measures.
It extends similar consistency results previously established for discrete action spaces where $\mu$ is the uniform measure \cite{reisinger2021regularity}, as well as for finite-dimensional action spaces where $\mu$ is the Lebesgue measure \cite{tang2022exploratory}.

Theorem \ref{thm:convergece_tau} also bounds    the decay rate of the regularization bias     $ ( v^*_\tau-v^*_0)_{\tau>0}  $ using the  convergence rate of 
$(H_\tau-H)_{\tau>0}$. 
To obtain explicit bound of $v^*_\tau-v^*_0$,   we distinguish two different cases for the action space:
(1) $A$ is of finite cardinality;
(2) $A$ is a general space.

For discrete action spaces, 
$H_\tau-H$ can be bounded using    the cardinality of $A$.
This subsequently allows for 
bounding $v^*_\tau-  v^*_0 $
and further proving 
the flow \eqref{eq:md_intro} 
with a constant  $\btau$  achieves an error     of order $\cO(1/S)$ at time $S$. 

\begin{theorem}
\label{ex:discrete_action_space}

Suppose $A=\{a_1,\ldots, a_{N}\}$ for some $N\in \sN$ and $\mu\in \cP(A)$ is the uniform distribution over $A$.
Then for all $\tau>0$ and $(x,u,p)\in \cO\times \sR\times \sR^d$, 
$0\le H_\tau(x, u, p)- H(x, u, p) \le (\ln N) \tau$.

Assume further that Assumptions~\ref{ass:data} and~\ref{assum:lsc} hold, and  let   $Z_0\in B_b(\cO\times A)$.
Then  there exists $C\ge 0$ such that for all $S>1$, by taking $\btau \equiv {1}/{S}$, 
the solution   $ Z\in \cap_{S>0}C^1([0,S];B_b(\cO\times A))$ to 
\eqref{eq:md_intro} satisfies 
\begin{equation*}
0\leq v_0^{\boldsymbol{\pi}(Z_s)}(x) - v_0^\ast(x) 
\leq   \frac{C}{S},\quad \forall x\in \overline \cO, s\ge S\,.   
\end{equation*}

\end{theorem}

Theorem \ref{ex:discrete_action_space}
follows directly from \eqref{eqn:modified_conv_fin_A} in 
Theorem \ref{cor:extend_conv_GF},  
Theorem \ref{thm:convergece_tau} and the well-known inequality that 
$
0\le \max_{1\le k\le N}a_k - \ln \frac{1}{N}\sum_{k=1}^N \exp(a_k) \le  \ln N
$ for all $ (a_k)_{k=1}^N\in \sR^N$, the details are omitted.

For general action spaces $A$,
analyzing the error 
$v^*_{\tau_0} - v^*_0$
is more technically involved,  as  the convergence rate of $(H_\tau(x,u,p)-H(x,u,p))_{\tau > 0}$ may depend  on  $(x,u,p)$. In the sequel, we optimize the scheduler $\btau$ under the following assumption on $(H_\tau -H)_{\tau >0}$.

\begin{assumption}
\label{assum:H_tau}
There exists $\alpha\ge 0$, $C\ge 0$ and $ {\tau}_{\max}\in (0,1)$  such that for all $\tau\in (0,{\tau}_{\max}]$ and $x\in \cO$, 
$$
H_\tau(x, v_0^*(x), Dv_0^*(x))- H(x, v_0^*(x), Dv_0^*(x))\le C\tau \left(\ln \frac{1}{\tau}\right)^\alpha\,,
$$
where  $v_0^*\in W^{2,p^*}(\cO)$ is defined by \eqref{eq:value_optimal}.

\end{assumption}

Assumption \ref{assum:H_tau} relaxes the uniform bound of $H_\tau-H$ in Theorem \ref{ex:discrete_action_space} into 
a local bound that depends on the unregularized value function $v_0^*$.  
Before presenting sufficient conditions  for Assumption \ref{assum:H_tau}, we first demonstrate that Assumption \ref{assum:H_tau} permits the choice of a scheduler $\btau$ in \eqref{eq:md_intro} that achieves an error  of order $\cO\left( {(\ln S)^{\alpha+1}}/{S}\right)$ at time $S$.

\begin{theorem}
\label{thm:conv_rate_continuous_constant}
Suppose  Assumptions~\ref{ass:data},~\ref{assum:lsc} and  \ref{assum:H_tau} hold,
and  let   $Z_0\in B_b(\cO\times A)$.
Then there exists $C\ge 0$ and $S_0>1$ such that for all $S\ge S_0$, by taking $\btau \equiv {\ln(S+1)}/{S}$, 
the solution   $  Z\in \cap_{S>0}C^1([0,S];B_b(\cO\times A))$ to 
\eqref{eq:md_intro} satisfies 
\begin{equation*}
0\leq v_0^{\boldsymbol{\pi}(Z_s)}(x) - v_0^\ast(x) 
\leq  C \frac{(\ln S)^{\alpha+1} }{S},\quad \forall x\in \overline \cO, s\ge S\,. 
\end{equation*}

\end{theorem}

Theorem \ref{thm:conv_rate_continuous_constant} 
follows directly from~\eqref{eqn:modified_conv},
Theorem \ref{thm:convergece_tau}, and the specific choice of $\btau$.
Detailed proofs are omitted for brevity.

We conclude this section by providing sufficient conditions for Assumption \ref{assum:H_tau}.
This   is closely related to obtaining precise Laplace asymptotics  for the integral in $H_\tau$,
\emph{uniformly} with respect to $(x,u,p)$
(see \cite[Chapter 7]{wong2001asymptotic}).
Quantifying the precise convergence rate in $\tau$ is challenging 
and typically has to be performed in a problem-dependent manner.
For instance, a rate of $\cO(\tau \log (1/\tau))$ is established 
for exploratory  temperature control problem
in \cite[Corollary 4.7]{tang2022exploratory}, 
and exploratory optimal stopping problem
\cite[Theorem 3.7]{dong2024randomized}, both with one-dimensional action spaces.
In Propositions \ref{ex:unique_min_in_conv_comp_set} and \ref{prop:EG_LQR},
we provide two different scenarios under which the same rate can be achieved.
The proofs are given in Section \ref{proof:lapace_asymptotic}.

The first example concerns the scenario where  the action space $A$ is finite-dimensional and the optimal actions of the unregularized   problem are unique and  achieved at the interior of the action set $A$.

\begin{proposition}
\label{ex:unique_min_in_conv_comp_set}
Suppose  Assumptions~\ref{ass:data} and~\ref{assum:lsc}  hold,
$A\subset \sR^k$ is a nonempty convex and compact set, and $\mu\in \cP(A)$ is the uniform distribution on $A$.
Let $v_0^*\in W^{2,p^*}(\cO)$ be the unique solution to \eqref{eq:hjb_unregualarised_1}.
Assume further that
$b\in C(\overline{\cO}\times A;\sR^d)$, 
$c\in C(\overline{\cO}\times A;\sR)$
and   $f\in C(\overline{\cO}\times A;\sR)$ are such that for all $x\in \overline{\cO}$,
$$A\ni  a\mapsto  h(x,a)\coloneqq b(x,a)^\top Dv^*_0(x) -c(x,a)v^*_0(x)+f(x,a)  \in \sR$$
admits a unique minimiser in the interior of $A$ and 
is twice  differentiable
with   derivative $D^2_{aa}h \in C(\overline{\cO}\times A;\sR^{k\times k})$.  
Then Assumption \ref{assum:H_tau} holds with $\alpha=1$.  

\end{proposition}

In general, optimal actions of   
the unregularized   problem  \eqref{eq:value_optimal_intro} 
may be achieved   both at the interior and the boundary of the action space $A$. In such cases, it is crucial to analyze when the minimum value $H$ is attained at the boundary of $A$ and to quantify its impact on the convergence rate \emph{uniformly} with respect to $(x,u,p)$.
The following example delves  into the specific scenario where the action space is one-dimensional, the drift coefficient $b$ and discount factor $c$ are linear in $a$, and the running cost $f$ is quadratic in $a$.  A comprehensive  analysis for general action spaces  is left for future work.

\begin{proposition}
\label{prop:EG_LQR}
Suppose Assumptions~\ref{ass:data} and~\ref{assum:lsc} hold,
$A=[\alpha,\beta]$ for some $-\infty<\alpha<\beta<+\infty$,  and $\mu\in \cP(A)$ is the uniform distribution on $A$.
Assume further that   
there exist     $\bar{b},\widehat{b}\in B_b(\sR^d;\sR^d)$,
$\bar{c},\widehat{c}\in B_b(\sR^d;\R)$,
$\bar{f},\tilde{f},  \widehat{f}\in B_b(\sR^d;\sR)$  such that for all $x\in \cO$ and $a\in A$, 
$$
b(x,a)=\bar{b}(x)+\widehat{b}(x)a,\quad    c(x,a)=\bar{c}(x)+\widehat{c}(x)a,\quad 
f(x,a) = \bar{b}(x)+\tilde{f}(x)a +   \widehat{f}(x) a^2\,,  
$$
and $\inf_{x\in \cO}\widehat{f}(x)>0$. Then   Assumption \ref{assum:H_tau} holds with $\alpha=1$.

\end{proposition}

\subsection{Convergence of mirror descent with annealing    schedulers}
\label{sec:md_anneal}

In this section, we  analyze the convergence rate  of the annealed flow \eqref{eq:md_intro} 
guided by a time-dependent scheduler $\btau$  that converges to $0$.
The analysis is more intricate than that for constant schedulers in Section \ref{sec:md_constant},
as one has to balance both terms in the estimates~\eqref{eqn:modified_conv} of~\eqref{eqn:modified_conv_fin_A},
as well as the
regularization bias $v^*_\tau-v^*_0$.
As before, we treat the cases with discrete action spaces and general action spaces separately.

For discrete action spaces, the following theorem shows that setting $\btau_s=1/(1+s)$ in  \eqref{eq:md_intro}
yields a rate of $\cO(1/s)$ as $s\to \infty$. 
The proof  is given in Section \ref{sec:proof_anneal}.

\begin{theorem}
\label{thm:conv_discrete_anneal}
Suppose    that Assumptions~\ref{ass:data} and~\ref{assum:lsc} hold, and $A$ is of finite cardinality. 
Let
$Z_0\in B_b(\cO\times A)$ and
let  $\btau\in C^1([0,\infty); (0,\infty))$ be such that $\btau_s =1/(1+s)$ for all $s>0$.
Then  there exists $C>0$ such that  the solution 
$Z\in \cap_{S>0}C^1([0,S];B_b(\cO\times A))$   to 
\eqref{eq:md_intro} satisfies 
\begin{equation*}
0\leq v_0^{\boldsymbol{\pi}(Z_s)}(x) - v_0^\ast(x) 
\leq   \frac{C}{s},\quad \forall x\in \overline \cO , s>1\,.
\end{equation*}

\end{theorem}

Note that the flow \eqref{eq:md_intro} with entropy annealing is an anytime algorithm \cite{lattimore2020bandit}, meaning it does not require determining the running horizon 
in advance, and the error bound in Theorem \ref{thm:conv_discrete_anneal} holds for all 
large $s>0$. In contrast, the flow \eqref{eq:md_intro} with a constant scheduler, as described in Theorem \ref{ex:discrete_action_space}, is not an anytime algorithm since the choice of scheduler depends on the horizon 
$S$, and results in a non-zero regularization bias. 

For general action spaces, 
the following theorem proves that the scheduler 
$\btau_s=1/\sqrt{1+s}$  ensures the convergence of the flow \eqref{eq:md_intro} as  $s\to \infty$,  and further achieves a convergence rate of $\cO(1/\sqrt{s})$ 
 under Assumption   \ref{assum:H_tau}. 

\begin{theorem}
\label{thm:conv_general_anneal}
Suppose    that Assumptions~\ref{ass:data} and \ref{assum:lsc} hold. 
Let
$Z_0\in B_b(\cO\times A)$ and
let  $\btau\in C^1([0,\infty); (0,\infty))$ be such that $\btau_s =1/\sqrt{1+s}$ for all $s>0$.
Then  
the solution 
$Z\in \cap_{S>0}C^1([0,S];B_b(\cO\times A))$   to 
\eqref{eq:md_intro} satisfies 
$\lim_{s\to \infty}  v_0^{\boldsymbol{\pi}(Z_s)}(x) =v_0^\ast(x)$
for all $x\in \overline{\cO}$.

Assume further that 
Assumption  \ref{assum:H_tau} holds. Then
there exists $C>0$ such that  
\begin{equation*}
0\leq v_0^{\boldsymbol{\pi}(Z_s)}(x) - v_0^\ast(x) 
\leq   \frac{C (\ln s)^\alpha}{\sqrt{s}},\quad \forall x\in \overline \cO , s>1\,.
\end{equation*}

\end{theorem}

The proof of Theorem \ref{thm:conv_general_anneal}  is given in Section \ref{sec:proof_anneal}.

\begin{remark}
\label{remark other schedulers}
It is not clear how to choose an alternative annealing scheduler that would improve the anytime rate $\mathcal{O}(1/\sqrt{s})$
given by Theorem \ref{thm:conv_general_anneal}.
For the scheduler $\btau_s = 1/\sqrt{1+s}$, the last term in \eqref{eqn:modified_conv} dominates the optimization error, which is approximately $\cO(1/\sqrt{s})$ (see \eqref{eq:dominating_term_general}). This matches the regularization bias up to a logarithmic term. A faster decaying $\boldsymbol{\tau}$ would reduce the regularization bias but increase the optimization error. For instance, if we take the scheduler $\btau_s = \ln(1+s) / (1+s)$,  as suggested in Theorem \ref{thm:conv_rate_continuous_constant}, then a careful asymptotic analysis reveals that the last term in \eqref{eqn:modified_conv} is approximately $\mathcal{O}(1/\ln s)$, resulting in an overall error of a worse rate $\mathcal{O}(1/\ln s)$ as $s\to \infty$.

In  Figure \ref{fig error plot annealing},  
we 
illustrate the behavior 
of the   overall error 
  $v_0^{\boldsymbol{\pi}(Z_S)}(x) - v_0^\ast(x)$
  under annealing schedules of the form 
  $\boldsymbol{\tau}_s = 1/(1+s)^\beta$,
  considering different values of  $\beta\in (0,1)$ and 
   different  running horizons $S$ of the gradient flow.
For a given $\beta$
and $S$,
we plot the upper bound of $v_0^{\boldsymbol{\pi}(Z_S)}(x) - v_0^\ast(x)$
  implied by  
\eqref{eqn:modified_conv}, Theorem~\ref{thm:conv_rate_continuous_constant} and Assumption~\ref{assum:H_tau} with $\alpha=1$.
The result shows that the optimal parameter, which minimizes the error, depends on the running horizon.
In our numerical results, the optimal \(\beta\) does not stabilize within the examined range of \(S\), leaving it unclear whether a scheduler within this parametric family optimizes the asymptotic performance as  \(S \to \infty\). Unfortunately, evaluating the error for longer time horizons is infeasible due to the exponentials in \eqref{eqn:modified_conv} becoming too large to handle with floating-point arithmetic.
\end{remark}

\begin{figure}
\begin{center}
\includegraphics[width=0.8\textwidth]{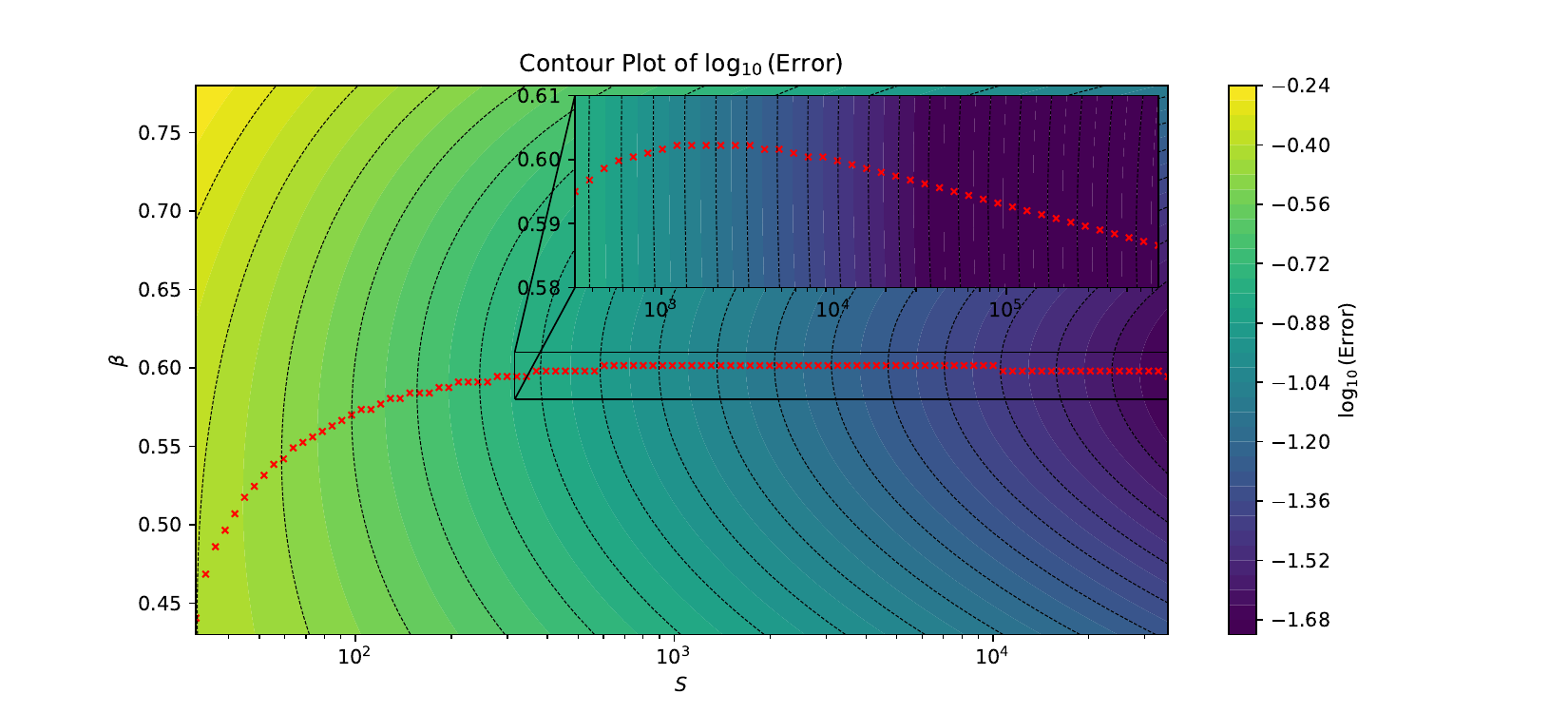}
\caption{
The overall error 
  $v_0^{\boldsymbol{\pi}(Z_S)}  - v_0^\ast $
  with   annealing schedulers   $\boldsymbol{\tau}_s = 1/(1+s)^\beta$,
for  different  $\beta\in (0,1)$ and running horizon $S$.}
\label{fig error plot annealing}
\end{center}	
\end{figure}

\subsection{Discussion: controlled diffusion coefficients}  
\label{sec:control_diffusion}

An analogue  mirror descent flow  can be proposed for exit time   problems
with controlled diffusions. 
For each $\pi\in \Pi_\mu$, 
consider the state process 
\begin{align*}
dX_t = \left(\int_Ab(X_t^\pi,a)\pi(da|X_t)\right)dt + \left(\int_A\sigma\sigma^\top (X_t,a)\pi(da|X_t)\right)^{\frac{1}{2}}dW_t\,,\enspace X_0=x\,,\enspace t\geq 0\,.
\end{align*}
A  similar   argument  as that in Section  \ref{sec:intro} shows  that 
the corresponding mirror descent flow is  
\begin{equation}
\label{eq:md_diffusion}
\partial_s {Z_s}(x,a) =-\left(  \mathcal L^a v^{\boldsymbol{\pi}(Z_s)}_{\btau_s} (x)+ f(x,a)  + \btau_s Z_s(x,a)\right)  \,,\quad   (x,a)\in \cO\times A,\,s>0 \,,
\end{equation}
with the   operator $ \mathcal L^a$   given by
\begin{align*}
(\mathcal{L}^a v)(x) = \frac{1}{2}\tr(\sigma(x,a)\sigma(x,a)^\top D^2v(x) )+b(x,a)^\top Dv(x) -c(x,a)v(x)\,.
\end{align*}
Compared with \eqref{eq:md_intro}, \eqref{eq:md_diffusion} involves a second-order differential operator due to the controlled diffusion coefficient.
Assume that \eqref{eq:md_diffusion} has a (sufficiently regular) solution $Z  $
along which   the map $ s \mapsto v^{\boldsymbol{\pi}(Z_s)}_{\btau_s}$ is differentiable.  Then 
one can extend 
Theorem  \ref{ref:cost_decrease_along_flow} to prove that $ s \mapsto v^{\boldsymbol{\pi}(Z_s)}_{\btau_s}$  decreases and
further quantify 
the optimization error $v^{\boldsymbol{\pi}(Z_s)}_{\btau_s}-v^*_\tau$ as in Proposition 
\ref{thm:convergence_of_GF}. 
The regularization bias 
$v^*_\tau-v^*_0$ can also be estimated 
under  sufficient regularity conditions on the coefficients 
as in 
\cite{reisinger2021regularity, tang2022exploratory}.

However, it remains unclear in which function space the flow \eqref{eq:md_diffusion} admits a solution that allows the value function to be differentiable. This primarily stems from the lack of regularity of 
$Z\mapsto \cL^a v^{\boldsymbol{\pi}(Z)}_\tau$. 
Given $\pi\in \Pi_\mu$, 
the controlled diffusion coefficient is  merely measurable, 
and hence 
standard elliptic regularity results can no longer be   applied to ensure the differentiability of  
$v^{\boldsymbol{\pi}(Z)}_\tau$.
Even when restricting to sufficiently regular $Z$,
it remains unclear  under which norm the map $Z\mapsto  \cL^a v^{\boldsymbol{\pi}(Z)}_\tau$
is continuous, which is essential for applying a fixed point theorem to establish the existence of a solution to   \eqref{eq:md_diffusion}. This lack of continuity   also hinders the application of the techniques developed in this paper to establish the differentiability of 
$Z\mapsto  v^{\boldsymbol{\pi}(Z)}_\tau$ along the flow \eqref{eq:md_diffusion};
see Section \ref{sec:Differentiability_Stability} for more details.

\section{Performance difference and regularity of cost functional}
\label{sec:Differentiability_Stability}
This section
establishes several essential properties of the regularized value function $v^\pi_\tau$, which will be used to analyze the well-posedness and convergence of the flow \eqref{eq:md_intro}.  
For notational simplicity, 
in the sequel,
we write 
$K\coloneqq \max\{ \|b\|_{B_b(\mathbb{R}^d\times A)}, \|c\|_{B_b(\mathbb{R}^d\times A)}, \|f\|_{B_b(\mathbb{R}^d\times A)}, \|\sigma\|_{B_b(\mathbb{R}^d)}\}$,
and 
denote by $C>0$  a generic constant 
which depends only on $d$, $p^*$, $\lambda$, $\cO$ and the modulus of continuity of $\sigma\sigma^\top$, 
and may take a different value at each occurrence.

We first prove the so-called performance difference lemma for any two different Gibbs policies.  
The proof is   based on  the following generalized  Feynman-Kac formula for   linear PDEs established in \cite[Ch.~2., Sec.~10, Theorem 1]{krylov2008controlled}. 

\begin{proposition}\label{prop:Fey_Kac}
Suppose  Assumption \ref{ass:data} holds. 
Let  $u\in W^{2,p^\ast}(\cO)$, $h\in L^{p^\ast}(\cO)$ 
and $\pi\in\Pi_\mu$ satisfy
\begin{equation}
\label{eqn:helper_F_K}
\int_A(\mathcal{L}^a u)(x)\pi(da|x)+h(x) =0\,\,\,\text{a.e.}\,\,\, x \in \cO\,; \quad 
u(x)=0, \, x \in \partial \cO\,.
\end{equation}
Then for all $x\in \cO$, $u(x)=\E^{\mathbb{P}^{x,\pi}}\int_0^{\tau_\cO}\Gamma^\pi_t h(X_t)\,dt$,
where  $(X^{x,\pi}_t)_{t\geq 0}$ is the unique  weak solution to \eqref{eqn:state}.
\end{proposition}

\begin{lemma}[Performance Difference]
\label{lemma:performance_diff}
Suppose Assumption \ref{ass:data} holds. 
Then for all $\pi,\pi'\in\Pi_\mu$, $\tau>0$ and  $x\in \mathcal O$,
\begin{equation}
\label{eqn:performance_diff}
\begin{split}
v_\tau^{\pi}(x)-v_\tau^{\pi'}(x)&=\E^{\mathbb P^{x,\pi}}\int_0^{\tau_\cO}\Gamma^\pi_t\int_A\left(\mathcal{L}^av_\tau^{\pi'}(X_t)+f(X_t,a)+\tau\ln\frac{d\pi'}{d\mu}(a|X_t)\right)\left(\pi-\pi'\right)\left(da|X_t\right)dt\\
&\quad +\tau\E^{\mathbb P^{x,\pi}}\int_0^{\tau_\cO}\Gamma^\pi_t\KL\left(\pi\big|\pi'\right)(X_t)dt\,.
\end{split}
\end{equation}
Moreover,  $\mathcal{L}^a$ in~\eqref{eqn:performance_diff} can be replaced by  $\overline{\mathcal{L}}^a$ defined in \eqref{eq:first_order_generator}.
\end{lemma}
\begin{proof}
By Proposition \ref{prop:Bellman_PDE_wellposedness}, $v_\tau^{\pi}-v_\tau^{\pi'}\in W^{2,p^*}(\cO)\cap W^{1,{p^*}}_0(\cO)$. Define $h\in L^{p^*}(\cO)$ such that for a.e.~$x\in \cO$, 
\begin{align}
\label{eqn:forcing_increment_pde}
h(x)\coloneqq \int_A\left(\mathcal{L}^av_\tau^{\pi'}(x)+f(x,a)+\tau\ln\frac{d\pi'}{d\mu}(a|x)\right)\left(\pi-\pi'\right)(da|x)+\tau\KL\left(\pi\big|\pi'\right)(x)\,.
\end{align}
Then for a.e.~ $x\in \cO$,
\begin{align*}
\int_A\mathcal{L}^a\big(v_\tau^{\pi}-v_\tau^{\pi'}\big)(x)\pi(da|x)+h(x)&=\int_A
\left(\mathcal{L}^a v_\tau^{\pi}(x)+f(x,a)+\tau\ln\frac{d \pi}{d\mu}(a|x)\right)\pi(da|x)\\
&\enspace\enspace-\int_A
\left(\mathcal{L}^a v_\tau^{\pi'}(x)+f(x,a)+\tau\ln\frac{ d\pi'}{d\mu}(a|x)\right)\pi'(da|x)=0\,.
\end{align*}
Applying Proposition~\ref{prop:Fey_Kac} with $u=v_\tau^\pi-v_\tau^{\pi'}$ and $h$ given by~\eqref{eqn:forcing_increment_pde} 
leads to~\eqref{eqn:performance_diff}. 
The fact that $\frac12 \text{tr}(\sigma \sigma^\top D^2 v^{\pi'})$ is independent of $a$ implies that~\eqref{eqn:performance_diff} holds with $\mathcal{L}^a$ replaced by $\overline{\mathcal{L}}^a$.
\end{proof}

Our next aim is to prove that 
for each $x\in\cO$, the map $B_b(\cO\times A)\times (0,\infty)\ni (Z,\tau) \mapsto v_\tau^{\boldsymbol{\pi}(Z)}(x)\in \mathbb R$ is Hadamard differentiable.
We adopt the notion of Hadamard differentiability which is the weakest notion for which a chain rule holds (see the remark below \cite[Proposition 2.47]{Bonnans2000PerturbationAO}).The following definition recalls  the notion of Hadamard differentiability  as given in 
~\cite[Chapter 2.2]{Bonnans2000PerturbationAO}.

\begin{definition}
[Hadamard Derivative] 
\label{def:H_diff}
Let $X,Y$ be Banach spaces. We say $\mathcal{H}:X\rightarrow Y$ is Hadamard differentiable if there exists $\partial\mathcal{H}:X\rightarrow\mathscr{L}(X,Y)$, called the differential of $\mathcal{H}$, such that for all $x,v\in X$, and   all sequences $(h_n)_{n\in\mathbb{N}}\subset(0,1)$ and $(v_n)_{n\in\mathbb{N}}\subset X$ such that $\lim_{n\to \infty}h_n = 0$ and $\lim_{n\to \infty} v_n = v$,
\begin{align*}
\lim_{n\rightarrow\infty}\frac{\mathcal{H}(x+h_nv_n)-\mathcal{H}(x)}{h_n}=\partial\mathcal{H}(x)[v].
\end{align*}

\end{definition}

We then     summarize the Hadamard differentiability 
of some basic functions related to the operator $\boldsymbol{\pi}$ defined in \eqref{eq:operator_pi}. 
These results have been proved by
Propositions 3.6, 3.7 and 3.9 in~\cite{kerimkulov2023fisher}.

\begin{proposition}
\label{prop:pi_derivative}
\begin{enumerate}[(1)]
\item
\label{prop:mirror_map_derivative}
The map $\boldsymbol{\pi}:B_b(\cO\times A) \rightarrow b\mathcal{M}(A|\cO)$ is Hadamard differentiable. 
The differential $\partial\boldsymbol{\pi}:B_b(\cO\times A)\rightarrow\mathscr{L}(B_b(\cO\times A); b\mathcal{M}(A|\cO))$
satisfies for all $Z,Z'\in B_b(\cO\times A)$, 
\begin{align}
\label{eqn:Deriv_MM}
\partial\boldsymbol{\pi}(Z)[Z'](da|x)=\left(Z'(x,a)-\int_AZ'(x,\bar{a})\boldsymbol{\pi}(Z)(d\bar{a}|x)\right)\boldsymbol{\pi}(Z)(da|x)\,, 
\end{align}
and  $\left\|\partial\boldsymbol{\pi}(Z)\right\|_{\mathscr{L}(B_b(\cO\times A);b\mathcal{M}(A|\cO))}\leq 2$.

\item 
\label{prop:ln_pi_diff}
The map $\ln\frac{d\boldsymbol{\pi}}{d\mu}:B_b(\cO\times A)\rightarrow B_b(\cO\times A)$ is Hadamard differentiable. The   differential $\partial\ln\frac{d\boldsymbol{\pi}}{d\mu}:B_b(\cO\times A)\rightarrow\mathscr{L}(B_b(\cO\times A);B_b(\cO\times A))$ satisfies for all $Z,Z'\in B_b(\cO\times A)$, 
\begin{align}
\label{eqn:ln_mirror_map_derivative}
\left(\partial\ln\frac{d\boldsymbol{\pi}(Z)}{d\mu}[Z']\right)(x,a)=Z'(x,a)-\int_AZ'(x,a')\boldsymbol{\pi}(Z)(da'|x)\,,
\end{align}
and 
$\left\|\partial\ln\frac{d\boldsymbol{\pi}(Z)}{d\mu}\right\|_{\mathscr{L}(B_b(\cO\times A);B_b(\cO\times A))}\leq 2$.

\item 
\label{prop:diff_ln_RE}
The map $T:B_b(\cO\times A) \rightarrow B_b(\cO)$ defined by 
$ T(Z)=\ln\int_A e^{Z(\cdot,a)}\mu(da)
$ 
is Hadamard differentiable and for all $Z,Z'\in B_b(\cO\times A)$,
\begin{align}
\label{eqn:derivative}
\partial T(Z)[Z'](x)=\int_A Z'(x,a)\boldsymbol{\pi}(Z)(da|x).
\end{align}

\end{enumerate}

\end{proposition}

We proceed by proving three fundamental properties of the regularized value functions: (i) 
the local  Lipschitz continuity  of  $(Z,\tau)\mapsto v_\tau^{\boldsymbol{\pi}(Z)}$
(Proposition \ref{prop:W^{2,p} bound}),
(ii)
the linear growth of $Z\mapsto v_\tau^{\boldsymbol{\pi}(Z)}$
(Lemma \ref{lem:DV_s_bound}),
and (iii) the   differentiability
of the KL divergence 
(Lemma \ref{lem:entropy_diff}).
These properties serve as key ingredients in proving  the desired Hadamard differentiability of 
$(Z,\tau) \mapsto v_\tau^{\boldsymbol{\pi}(Z)}(x)$ (Proposition \ref{prop:funct_derivative_rho}).
To this end,
recall that   under   Assumption \ref{ass:data}, by the   Sobolev inequality \cite[Section 5.6.3, Theorem 6]{evans2022partial}, there exists a constant $C>0$, depending only on $d$, $p^*$ and $\cO$,
such that for all $u\in W^{2,p^*}(\cO)$,
\begin{align}
\label{eq:sobolev_inq}
\|u\|_{C^{1}(\overline{\cO})} \leq \|u\|_{C^{1,1-{d}/{p^*}}(\overline{\cO})}\leq C\|u\|_{W^{2,p^*}(\cO)}\,.
\end{align}

\begin{proposition}
\label{prop:W^{2,p} bound}
Suppose Assumption \ref{ass:data} holds.
There exists a constant $C>0$ such that
for all $Z,Z'\in B_b(\cO\times A)$ and $\tau, \tau'>0$,
\begin{align}
\|v_\tau^{\boldsymbol{\pi}(Z)}-v_\tau^{\boldsymbol{\pi}(Z')}\|_{W^{2,p^*}(\cO)}&\leq C(1+\tau)(
1+\|v_\tau^{\boldsymbol{\pi}(Z)}\|_{C^1(\overline{\cO})}+\|Z\|_{B_b(\cO\times A)})\|Z-Z'\|_{B_b(\cO\times A)},
\label{eqn:dual_W_2_p_inc}
\\
\|v^{\boldsymbol{\pi}(Z)}_{\tau}-v^{\boldsymbol{\pi}(Z)}_{\tau'}\|_{W^{2,p^*}(\cO)} 
&\leq C|\tau-\tau'|\|Z\|_{B_b(\cO\times A)}\,.
\label{eqn:lip_in_reg}
\end{align}
\end{proposition}

\begin{proof}
Let $\pi=\boldsymbol{\pi}(Z)$  and $\pi'=\boldsymbol{\pi}(Z')$.
We start by   showing that there exists a constant $C>0$ such that
\begin{equation}
\label{eqn:W_2_p_stability}
\begin{split}
\|v_\tau^\pi-v_\tau^{\pi'}\|_{W^{2,p^*}(\cO)}
\leq C\bigg(
\left(1+\left\|v_\tau^\pi\right\|_{C^1(\overline{\cO})}\right)\|\pi-\pi'\|_{b\mathcal{M}(A|\cO)}+\tau\|\KL(\pi|\mu)-\KL(\pi'|\mu)\|_{B_b(\cO)}
\bigg).
\end{split}
\end{equation}
Consider the Dirichlet problem 
\begin{align}
\label{eqn:increment_bellman}
\int_A \mathcal{L}^a w(x)\pi'(da|x)=-h(x),\;  x \in \cO; \quad 
w = 0, \; x \in \partial \cO,
\end{align}
where $h(x)=\int_A\mathcal{L}^av_\tau^{\pi}(x)+f(x,a)[\pi-\pi'](da|x)+\tau\left(\KL(\pi|\mu)(x)-\KL(\pi'|\mu)(x)\right)$.
Since $Z\in B_b(\cO\times A)$, Proposition~\ref{prop:Bellman_PDE_wellposedness} implies that
$v_\tau^\pi\in W^{2,p^*}(\cO)$ and so $h\in L^{p^*}(\cO)$. 
By 
standard elliptic regularity results
(see Lemma~\ref{lemma:elliptic_regularity}),
\eqref{eqn:increment_bellman} admits a unique solution $w\in  W^{2,p^*}(\cO)$
and 
$\|w\|_{W^{2,p^*}(\cO)}\leq C\|h\|_{L^{p^*}(\cO)}$,
for some constant 
$C>0$.
As shown in the proof of Lemma  \ref{lemma:performance_diff}, $v_\tau^\pi-v_\tau^{\pi'}$ satisfies \eqref{eqn:increment_bellman}, which implies $w=v_\tau^\pi-v_\tau^{\pi'}$. 
Hence it remains to   bound $\|h\|_{L^{p^*}(\cO)}$.
To that end note that 
\begin{align*}
|h(x)|&\leq\left\|b\cdot Dv_\tau^\pi-cv_\tau^\pi+f\right\|_{B_b(\cO\times A)}\|\pi-\pi'\|_{b\mathcal{M}(A|\cO)}+\tau\|\KL(\pi|\mu)-\KL(\pi'|\mu)\|_{B_b(\cO)}\\
&\leq K(\|v_\tau^\pi\|_{C^1(\overline{\cO})}+1)\|\pi-\pi'\|_{b\mathcal{M}(A|\cO)}+\tau\|\KL(\pi\|\mu)-\KL(\pi'|\mu)\|_{B_b(\cO)},
\end{align*}
which shows that~\eqref{eqn:W_2_p_stability} holds.

Now we prove \eqref{eqn:dual_W_2_p_inc}. 
Let $Z^\varepsilon=Z'+\varepsilon(Z-Z')$.
From Proposition~\ref{prop:pi_derivative} Item \ref{prop:mirror_map_derivative} we have 
\begin{equation}
\label{eqn:stability_helper}
\begin{split}
\|\pi-\pi'\|_{b\mathcal{M}(A|\cO)}&=\left\|\int_0^1\partial\boldsymbol{\pi}(Z^\varepsilon)[Z'-Z]d\varepsilon\right\|_{b\mathcal{M}(A|\cO)}\leq 2\|Z'-Z\|_{B_b(\cO\times A)}.
\end{split}
\end{equation}
For the entropy term we can write 
\begin{align}
\label{eqn:stability_helper_3}
&\|\KL(\pi|\mu)-\KL(\pi'|\mu)\|_{B_b(\cO)}=\|\KL(\boldsymbol{\pi}(Z)|\mu)-\KL(\boldsymbol{\pi}(Z')|\mu)\|_{B_b(\cO)}\nonumber\\
&=\sup_{x}\left|\int_A\ln\frac{d\boldsymbol{\pi}(Z)}{d\mu}(a|x)\boldsymbol{\pi}(Z)(da|x)-\int_A\ln\frac{d\boldsymbol{\pi}(Z')}{d\mu}(a|x)\boldsymbol{\pi}(Z')(da|x)\right|\nonumber\\
&\leq\sup_{x}\int_A\left|\ln\frac{d\boldsymbol{\pi}(Z)}{d\mu}(a|x)\right||\boldsymbol{\pi}(Z)-\boldsymbol{\pi}(Z')|(da|x)+\int_A\left|\ln\frac{d\boldsymbol{\pi}(Z)}{d\mu}(a|x)-\ln\frac{d\boldsymbol{\pi}(Z')}{d\mu}(a|x)\right|\pi'(da|x)\nonumber\\
&\leq 2\|Z\|_{B_b(\cO\times A)}\|\boldsymbol{\pi}(Z)-\boldsymbol{\pi}(Z')\|_{b\mathcal{M}(\cO\times A)}+\left\|\ln\frac{d\boldsymbol{\pi}(Z)}{d\mu}-\ln\frac{d\boldsymbol{\pi}(Z')}{d\mu}\right\|_{B_b(\cO\times A)}.
\end{align}
From Proposition \ref{prop:pi_derivative} Item \ref{prop:ln_pi_diff} and the mean value theorem
\begin{align*}
\ln\frac{d\boldsymbol{\pi}(Z)}{d\mu}-\ln\frac{d\boldsymbol{\pi}(Z')}{d\mu}=\int_0^1\left(\left[\partial\ln\frac{d\boldsymbol{\pi}(Z^\varepsilon)}{d\mu}\right](Z'-Z)\right)\,d\varepsilon.
\end{align*}
Taking the $\|\cdot\|_{B_b(\cO\times A)}$ norm  and applying the bound on the operator norm of $\partial\ln\frac{d\boldsymbol{\pi}(Z)}{d\mu}$ given in Proposition~\ref{prop:pi_derivative} Item \ref{prop:ln_pi_diff} implies
\begin{align}
\label{eqn:stability_helper_2}
\left\|\ln\frac{d\boldsymbol{\pi}(Z)}{d\mu}-\ln\frac{d\boldsymbol{\pi}(Z')}{d\mu}\right\|_{B_b(\cO\times A)}\leq 2\|Z-Z'\|_{B_b(\cO\times A)}.
\end{align}
Substituting \eqref{eqn:stability_helper} and $\eqref{eqn:stability_helper_2}$ into \eqref{eqn:stability_helper_3} yields
\begin{align}
\label{eqn:stab_helper_4}
\|\KL(\boldsymbol{\pi}(Z)|\mu)-\KL(\boldsymbol{\pi}(Z')|\mu)\|_{B_b(\cO)}\leq 2\left(2\|Z\|_{B_b(\cO\times A)} +1\right)\|Z-Z'\|_{B_b(\cO\times A)}.
\end{align}
Finally substituting \eqref{eqn:stability_helper} and \eqref{eqn:stab_helper_4} into \eqref{eqn:W_2_p_stability} implies there exists a constant $C>0$ 
such that
\begin{align*}
\|v_\tau^\pi-v_\tau^{\pi'}\|_{W^{2,p^*}(\cO)}&\leq C\left[2(1+\|v_\tau^{\boldsymbol{\pi}(Z)}\|_{C^1(\overline{\cO})})+2\tau(2\|Z\|_{B_b(\cO\times A)} + 1)\right]\|Z-Z'\|_{B_b(\cO\times A)}\\
&\leq C(1+\|v_\tau^{\boldsymbol{\pi}(Z)}\|_{C^1(\overline{\cO})}+\tau\|Z\|_{B_b(\cO\times A)}+\tau)\|Z-Z'\|_{B_b(\cO\times A)}.
\end{align*}
This together with~\eqref{eqn:stability_helper} 
proves the inequality \eqref{eqn:dual_W_2_p_inc}.

It remains to prove \eqref{eqn:lip_in_reg}.
Observe that 
$w\coloneqq v^{\boldsymbol{\pi}(Z)}_{\tau}-v^{\boldsymbol{\pi}(Z)}_{\tau'}\in W^{2,p^*}(\cO)$ and satisfies (cf.~\eqref{eq:on_policy_bellman})
\begin{align*}
\int_A \mathcal{L}^aw(x)\boldsymbol{\pi}(Z)(da|x)=-(\tau-\tau')\KL(\boldsymbol{\pi}(Z)|\mu)(x)   \,\,\, \textnormal{a.e.~$x\in \cO$};  \,\,\, 
w(x)=0,  \,\,\, x \in \partial \cO\,.
\end{align*}
This along with  Lemma \ref{lemma:elliptic_regularity} shows that    
\begin{align*}
\|v^{\boldsymbol{\pi}(Z)}_{\tau}-v^{\boldsymbol{\pi}(Z)}_{\tau'}\|_{W^{2,p^*}(\cO)}&\leq C|\tau-\tau'|\|\KL(\boldsymbol{\pi}(Z)|\mu)\|_{L^{p^*}(\cO)}
\le C|\tau-\tau'|\|\KL(\boldsymbol{\pi}(Z)|\mu)\|_{B_b(\cO)}
\\
&\le C|\tau-\tau'|\|Z\|_{B_b(\cO\times A)}\,,
\end{align*}
where the last inequality used  
\eqref{eqn:stab_helper_4}
and $\KL(\boldsymbol{\pi}(\textbf 0)|\mu)=0$.
\end{proof}

An immediate consequence of 
Proposition \ref{prop:W^{2,p} bound}
is the following linear growth of 
$Z\mapsto v_\tau^{\boldsymbol{\pi}(Z)}$.
It follows by taking $Z=0$ 
in 
\eqref{eqn:dual_W_2_p_inc}, 
and 
using 
\eqref{eq:sobolev_inq}
and the bound $\|v_\tau^{\boldsymbol{\pi}(\textbf{0})}\|_{W^{2,p^*}(\cO)}\le C$ due to Lemma \ref{lemma:elliptic_regularity}.
\begin{lemma}
\label{lem:DV_s_bound}
Suppose Assumption~\ref{ass:data} holds. 
There exists  a constant  $C>0$ 
such that
for all $Z\in B_b(\cO\times A)$ and $\tau>0$,
\begin{align*}
\|v_\tau^{\boldsymbol{\pi}(Z)}\|_{C^1(\overline{\cO})}\leq C(1+\tau)(1+\|Z\|_{B_b(\cO\times A)}).
\end{align*}
\end{lemma}

Finally we prove the differentiability of the KL divergence.
 
\begin{lemma}
\label{lem:entropy_diff}
The map $B_b(\cO\times A) \ni Z \mapsto \KL(\boldsymbol{\pi}(Z)|\mu) \in  B_b(\cO)$ 
is Hadamard differentiable and    for all $Z,Z'\in B_b(\cO\times A)$ and $x\in\cO$,
\begin{align*}
\partial\KL(\boldsymbol{\pi}(Z)|\mu)[Z'](x)=\int_A\ln\frac{d\boldsymbol{\pi}(Z)}{d\mu}(a|x)\partial\boldsymbol{\pi}(Z)[Z'](da|x).
\end{align*} 
\end{lemma}

\begin{proof}
Let  $(Z'_n)_{n\in \sN} \subset B_b(\cO\times A)$ and $(h_n)_n\subset(0,1)$ is such that 
$\lim_{n\to\infty} \|Z'_n - Z'\|_{B_b(\cO\times A)}=0$ and  
$\lim_{n\to \infty} h_n = 0$.
Observe that for all $x\in \cO$,
\begin{align*}
\frac{\KL(\boldsymbol{\pi}(Z+h_n Z'_n)|\mu)(x)-\KL(\boldsymbol{\pi}(Z)|\mu)(x)}{h_n}=f_{n}(x)+g_{n}(x),
\end{align*}
where
\begin{align*}
f_{ n}(x)&\coloneqq \int_A\ln\frac{d\boldsymbol{\pi}(Z+h_n Z'_n)}{d\mu}(a|x)\left(\frac{\boldsymbol{\pi}(Z+h_n Z'_n)-\boldsymbol{\pi}(Z)}{h_n}\right)(da|x),\\
g_{ n}(x)&=\int_A\frac{\ln\frac{d\boldsymbol{\pi}(Z+h_n Z'_n)}{d\mu}(a|x)-\ln\frac{d\boldsymbol{\pi}(Z)}{d\mu}(a|x)}{h_n}\pi(Z)(da|x)\,.
\end{align*}
For the convergence of  $( f_{n})_{n\in \sN}$, we have
\begin{align*}
&\left\|f_{n}(x)-\int_A\ln\frac{d\boldsymbol{\pi}(Z)}{d\mu}(a|x)\partial\boldsymbol{\pi}(Z)[Z'](da|x)\right\|_{B_b(\cO)}\\
&\leq\left\|\ln\frac{d\boldsymbol{\pi}(Z+h_n Z'_n)}{d\mu}-\ln\frac{d\boldsymbol{\pi}(Z)}{d\mu}\right\|_{B_b(\cO\times A)}\left\|\frac{\boldsymbol{\pi}(Z+h_n Z'_n)-\boldsymbol{\pi}(Z)}{h_n}\right\|_{b\mathcal{M}(A|\cO)}\\
&\enspace\enspace+\left\|\ln\frac{d\boldsymbol{\pi}(Z)}{d\mu}\right\|_{B_b(\cO\times A)}\left\|\frac{\boldsymbol{\pi}(Z+h_n Z'_n)-\boldsymbol{\pi}(Z)}{h_n}-\partial\boldsymbol{\pi}(Z)[Z']\right\|_{b\mathcal{M}(A|\cO)},
\end{align*}
where converges to zero due to $\sup_{n\in \sN}\| Z_n\|_{B_b(\cO\times A)}<\infty$
and  Proposition \ref{prop:pi_derivative} Item \ref{prop:mirror_map_derivative}.
For the convergence of  $( g_{n})_{n\in \sN}$,
note that 
by Proposition \ref{prop:pi_derivative} Item \ref{prop:ln_pi_diff},
for any $\varepsilon>0$, there exists   $N\in\mathbb{N}$ such that 
for all $n\ge N$ and $(x,a)\in \cO\times A$,
\begin{align*}
Z'(x,a)-\int_AZ'(x,a)\boldsymbol{\pi}(Z)(da|x)-\varepsilon&\leq\frac{1}{h_n}\left(\ln\frac{d\boldsymbol{\pi}(Z+h_n Z'_n)}{d\mu}(a|x)-\ln\frac{d\boldsymbol{\pi}(Z)}{d\mu}(a|x)\right)\\
&\leq Z'(x,a)-\int_AZ'(x,a)\boldsymbol{\pi}(Z)(da|x)+\varepsilon\, ,
\end{align*}
from which by integrating both sides with $\boldsymbol{\pi}(Z)(da|x)$ yields
$\|g_n\|_{B_b(\cO)}\le \varepsilon$ for all $n\ge N$. This implies that $\lim_{n\to \infty }\|g_n\|_{B_b(\cO)}=0$.
\end{proof}

Now we are ready to present the desired Hadamard differentiability of $(Z,\tau) \mapsto v_\tau^{\boldsymbol{\pi}(Z)}(x)$ 
and compute  its Hadamard derivative.

\begin{proposition}
\label{prop:funct_derivative_rho}
Suppose Assumption \ref{ass:data} holds.
For all $x\in\cO$, the map 
$B_b(\cO\times A)\times (0,\infty)\ni(Z,\tau)\mapsto v^{\boldsymbol{\pi}(Z)}_{\tau}(x)\in \mathbb{R}$
is Hadamard differentiable,\footnote{Note that  
the domain $(0,\infty)$ of $\tau \mapsto v^{\boldsymbol{\pi}(Z)}_{\tau}(x)$  is  only a   subset of $\sR$,
and,   strictly speaking, does not align with Definition \ref{def:H_diff}. However, 
it is straightforward to extend Definition \ref{def:H_diff}
to this setting by restricting to all sequences 
$(\tau'_n)_{n\in \sN}\subset \sR $   and $(h_n)_{n\in \sN}\subset (0,1)$ 
such that
$\lim_{n\rightarrow\infty}\tau_n'=\tau'$,
$\lim_{n\rightarrow\infty}h_n=0$, and 
$\tau+h_n\tau'_n\subset (0,\infty)$ for all $n\in \sN$.  
One can show that the chain rule still holds under this relaxation.}   
and for all $Z,  Z' \in B_b(\cO\times A)$, $\tau>0$ and $\tau'\in \sR$,
\begin{equation}
\label{eqn:Had_derivative}
\begin{split}
\partial&v^{\boldsymbol{\pi}(Z)}_\tau(x)[(Z',\tau')]\\
&=
\E^{\mathbb{P}^{x,\boldsymbol{\pi}(Z)}}\int_0^{\tau_\cO}\Gamma^{\boldsymbol{\pi}(Z)}_t\int_A\left(\overline{\mathcal{L}}^a v_\tau^{\boldsymbol{\pi}(Z)}(X_t)+f(X_t,a)+\tau\ln\tfrac{d\boldsymbol{\pi}(Z)}{d\mu}(a|X_t)\right)\partial\boldsymbol{\pi}(Z)[Z']\left(da|X_t\right)dt\\
&\enspace\enspace+\tau'\E^{\mathbb{P}^{x,\boldsymbol{\pi}(Z)}}\int_0^{\tau_{\cO}}\Gamma^{\boldsymbol{\pi}(Z)}_t\KL(\boldsymbol{\pi}(Z)|\mu )(X_t)dt\,,
\end{split}
\end{equation}
where 
$\overline{\mathcal{L}}^a$ is defined in \eqref{eq:first_order_generator}.
Moreover, $\overline{\mathcal{L}}^a$  in~\eqref{eqn:Had_derivative} can be replaced by $\mathcal{L}^a$.
\end{proposition}

\begin{proof}
Fix $Z,Z'\in B_b(\cO\times A)$, $\tau>0$ and $\tau'\in \sR$. 
Let
$(Z'_n)_{n\in\mathbb{N}}\subset B_b(\cO\times A)$, $(\tau_n')_{n\in\mathbb{N}}\subset \sR$
and $(h_n)_{n\in\mathbb{N}}\subset(0,1)$ be sequences such that 
$\lim_{n\rightarrow\infty}Z'_n= Z'$, $\lim_{n\rightarrow\infty}\tau_n'=\tau'$ 
and $\lim_{n\rightarrow\infty}h_n=0$.
For all $n\in \sN$, 
define 
$\boldsymbol{\pi}_n\coloneqq \boldsymbol{\pi}(Z+h_n Z'_n)$ and $\boldsymbol{\pi}_\infty\coloneqq \boldsymbol{\pi}(Z)$.
Note that for all $n\in \sN$, 
\begin{align}
\label{eqn:H_diff_decomp}
\frac{v_{\tau+h_n\tau_n'}^{\boldsymbol{\pi}_n}(x)-v_\tau^{\boldsymbol{\pi}_\infty}(x)}{h_n}=\frac{v^{\boldsymbol{\pi}_n}_{\tau+h_n\tau_n'}(x)-v^{\boldsymbol{\pi}_\infty}_{\tau+h_n\tau_n'}(x)}{h_n}+\frac{v^{\boldsymbol{\pi}_\infty }_{\tau+h_n\tau_n'}(x)-v_\tau^{\boldsymbol{\pi}_\infty}(x)}{h_n}.
\end{align}
Observe that   by Lemma \ref{lemma:performance_diff}, 
the first term on the right-hand side of    \eqref{eqn:H_diff_decomp} can be rewritten as 
\begin{align*}
&\frac{v_{\tau+h_n\tau'_n}^{\boldsymbol{\pi}_n }(x)-v_{\tau+h_n\tau'_n}^{\boldsymbol{\pi}_\infty }(x)}{h_n}\\
&=-\frac{1}{h_n}\bigg[ 
\E^{\mathbb{P}^{x,\boldsymbol{\pi}_\infty}}\int_0^{\tau_\cO}\Gamma^{\boldsymbol{\pi}_\infty}_t\left(\int_A\left[\mathcal{L}^av_{\tau+h_n\tau'_n}^{\boldsymbol{\pi}_n}\left(X_t\right)+f\left(X_t,a\right)\right]
\left( {\boldsymbol{\pi}_\infty - \boldsymbol{\pi}_n} \right)(da|X_t)
\right)dt \\
&\enspace\enspace+(\tau+h_n\tau'_n)\E^{\mathbb{P}^{x,\boldsymbol{\pi}_\infty}}\int_0^{\tau_\cO}\Gamma^{\boldsymbol{\pi}_\infty}_t \bigg( \int_A\ln\frac{d\boldsymbol{\pi}_n}{d\mu}(a|X_t)\left({\boldsymbol{\pi}_\infty-\boldsymbol{\pi}_n} \right)(da|X_t) + \KL(\boldsymbol{\pi}_\infty|\boldsymbol{\pi}_n)(X_t)\bigg)dt\\
&=\E^{\mathbb{P}^{x,\boldsymbol{\pi}_\infty}}\int_0^{\tau_\cO}\Gamma^{\boldsymbol{\pi}_\infty}_t\left(\int_A\left[\mathcal{L}^av_{\tau+h_n\tau'_n}^{\boldsymbol{\pi}_n}\left(X_t\right)+f\left(X_t,a\right)\right]\left[\frac{\boldsymbol{\pi}_n-\boldsymbol{\pi}_\infty}{h_n}\right](da|X_t)\right)dt\\
&\quad +(\tau+h_n\tau'_n) \E^{\mathbb{P}^{x,\boldsymbol{\pi}_\infty}}\int_0^{\tau_\cO}\Gamma^{\boldsymbol{\pi}_\infty}_t\frac{\KL\left(\boldsymbol{\pi}_n|\mu\right)(X_t)-\KL\left(\boldsymbol{\pi}_\infty|\mu\right)(X_t)}{h_n}dt\,,
\end{align*}
which along with \eqref{eqn:H_diff_decomp} implies that 
\begin{align*}
\frac{v_{\tau+h_n\tau_n'}^{\boldsymbol{\pi}(Z+h_n Z'_n)}(x)-v_\tau^{\boldsymbol{\pi}(Z)}(x)}{h_n}=I^1_{n}+I^2_{n}+(\tau+h_n\tau_n')I^3_{n} +\tau_n'\E^{\mathbb{P}^{x,\boldsymbol{\pi}_\infty}}\int_0^{\tau_{\cO}}\Gamma^{\boldsymbol{\pi}_\infty}_t \KL(\boldsymbol{\pi}(Z)|\mu)(X_t)\,dt\,,
\end{align*}
where  $I^1_{n}$, $I^2_{n}$ and $I^3_{n}$ are defined by  
\begin{align*}
I^1_{n}&\coloneqq\E^{\mathbb{P}^{x,\boldsymbol{\pi}_\infty}}\int_0^{\tau_\cO}\Gamma^{\boldsymbol{\pi}_\infty}_t\left(\int_A\mathcal{L}^av_{\tau+h_n\tau_n'}^{\boldsymbol{\pi}_n}\left(X_t\right)\left[\frac{\boldsymbol{\pi}_n-\boldsymbol{\pi}_\infty}{h_n}\right]\left(da|X_t\right)\right)dt\,,
\\
I^2_{n}&\coloneqq \E^{\mathbb{P}^{x,\boldsymbol{\pi}_\infty}}\int_0^{\tau_\cO}\Gamma^{\boldsymbol{\pi}_\infty}_t\left(\int_Af\left(X_t,a\right)\left[\frac{\boldsymbol{\pi}_n-\boldsymbol{\pi}_\infty}{h_n}\right]\left(da|X_t\right)\right)dt
\,,\\
I^3_{n} 
&\coloneqq
\E^{\mathbb{P}^{x,\boldsymbol{\pi}_\infty}}\int_0^{\tau_\cO}\Gamma^{\boldsymbol{\pi}_\infty}_t\frac{\KL\left(\boldsymbol{\pi}_n|\mu\right)(X_t)-\KL\left(\boldsymbol{\pi}_\infty|\mu\right)(X_t)}{h_n}dt\,.
\end{align*}
It suffices to prove that 
\begin{align}
\label{eq I_1_n}
\lim_{n\to \infty} I^1_{n}&=\E^{\mathbb{P}^{x,\boldsymbol{\pi}_\infty}}\int_0^{\tau_\cO}\Gamma^{\boldsymbol{\pi}_\infty}_t\left(\int_A\mathcal{L}^av_\tau^{\boldsymbol{\pi}(Z)}(X_t)\partial\boldsymbol{\pi}(Z)[Z']\left(da|X_t\right)\right)dt\,,\\
\lim_{n\to \infty}I^2_{n}&=\E^{\mathbb{P}^{x,\boldsymbol{\pi}_\infty}}\int_0^{\tau_\cO}\Gamma^{\boldsymbol{\pi}_\infty}_t\left(\int_Af\left(X_t,a\right)\partial\boldsymbol{\pi}(Z)[Z']\left(da|X_t\right)\right)dt\,,\\
\label{eq I_3_n}
\lim_{n\to \infty} I^3_{n}&=\E^{\mathbb{P}^{x,\boldsymbol{\pi}_\infty}}\int_0^{\tau_\cO}\Gamma^{\boldsymbol{\pi}_\infty}_t\left(\int_A\ln\frac{d\boldsymbol{\pi}(Z)}{d\mu}\left(a|X_t\right)\partial\boldsymbol{\pi}(Z)[Z']\left(da|X_t\right)\right)dt\,.
\end{align}

To prove the convergence of    $(I^1_{n})_{n\in\sN}$, note that
\begin{align}
\label{eqn:help}
\begin{split}
&\left|I^1_{n}-\E^{\mathbb{P}^{x,\boldsymbol{\pi}_\infty}}\int_0^{\tau_\cO}\Gamma^{\boldsymbol{\pi}_\infty}_t\left(\int_A\mathcal{L}^av_\tau^{\boldsymbol{\pi}(Z)}(X_t)\partial\boldsymbol{\pi}(Z)[Z']\left(da|X_t\right)\right)dt\right|
\nonumber\\
&=\bigg|\E^{\mathbb{P}^{x,\boldsymbol{\pi}_\infty}}\int_0^{\tau_\cO}\Gamma^{\boldsymbol{\pi}_\infty}_t
\bigg\{ \int_A \bigg( b(X_t,a)^\top Dv_{\tau+h_n\tau_n'}^{\boldsymbol{\pi}_n}(X_t)-c(X_t,a)v_{\tau+h_n\tau_n'}^{\boldsymbol{\pi}_n}(X_t)\bigg) \left[\frac{\boldsymbol{\pi}_n-\boldsymbol{\pi}_\infty}{h_n}\right](da|X_t)\nonumber\\
&\quad -\int_A \bigg( b(X_t,a)\cdot Dv_\tau^{\boldsymbol{\pi}_\infty}(X_t)-c(X_t,a)v_\tau^{\boldsymbol{\pi}_\infty}(X_t)\bigg) \partial\boldsymbol{\pi}(Z)[Z']\left(da|X_t\right)\bigg\}\,dt  \bigg|\nonumber\\
& \leq\left\|b^\top Dv_{\tau+h_n\tau_n'}^{\boldsymbol{\pi}_n}-cv_{\tau+h_n\tau_n'}^{\boldsymbol{\pi}_n}\right\|_{B_b(\cO\times A)}\left\|\frac{\boldsymbol{\pi}_n -\boldsymbol{\pi}_\infty}{h_n}-\partial\boldsymbol{\pi}(Z)[Z']\right\|_{b\mathcal{M}(A|\cO)}\E^{\mathbb{P}^{x,\boldsymbol{\pi}_\infty}}\left[\tau_\cO\right]\nonumber\\
&\quad +\|b^\top (Dv_{\tau+h_n\tau_n'}^{\boldsymbol{\pi}_n}-Dv_\tau^{\boldsymbol{\pi}_\infty})-c(v_\tau^{\boldsymbol{\pi}_\infty}-v_{\tau+h_n\tau_n'}^{\boldsymbol{\pi}_n})\|_{B_b(\cO\times A)}\|\partial\boldsymbol{\pi}(Z)[Z']\|_{b\mathcal{M}(A|\cO)}\E^{\mathbb{P}^{x,\boldsymbol{\pi}_\infty}}\left[\tau_\cO\right]\nonumber\\
& \leq
C  \|v_{\tau+h_n\tau_n'}^{\boldsymbol{\pi}_n}\|_{C^1(\overline{\cO})}\left\|\frac{\boldsymbol{\pi}_n -\boldsymbol{\pi}_\infty}{h_n}-\partial\boldsymbol{\pi}(Z)[Z']\right\|_{b\mathcal{M}(A|\cO)}\E^{\mathbb{P}^{x,\boldsymbol{\pi}_\infty}}\left[\tau_\cO\right] 
\\
&\quad 
+C \|Z'\|_{B_b(\cO\times A)}\|v_{\tau+h_n\tau_n'}^{\boldsymbol{\pi}_n}-v_\tau^{\boldsymbol{\pi}_\infty}\|_{C^1(\overline{\cO})} \E^{\mathbb{P}^{x,\boldsymbol{\pi}_\infty}}\left[\tau_\cO\right],
\end{split}
\end{align}
where the final inequality used    Proposition \ref{prop:pi_derivative} Item \ref{prop:mirror_map_derivative}.
By Lemma \ref{lem:DV_s_bound}
and the uniform boundedness of the sequence $(Z'_n)_{n\in\mathbb{N}}$ in $B_b(\cO\times A)$,  
$\sup_{n\in \sN} \|v_{\tau+h_n\tau_n'}^{\boldsymbol{\pi}_n}\|_{C^1(\overline{\cO})}<\infty$.
Thus, by Proposition \ref{prop:pi_derivative} Item \ref{prop:mirror_map_derivative}, 
$\lim_{n\rightarrow\infty}\|v_{\tau+h_n\tau_n'}^{\boldsymbol{\pi}_n}\|_{C^1(\overline{\cO})}\left\|\frac{\boldsymbol{\pi}_n -\boldsymbol{\pi}_\infty}{h_n}-\partial\boldsymbol{\pi}(Z)[Z']\right\|_{b\mathcal{M}(A|\cO)}=0$.
Moreover, by Proposition~\ref{prop:W^{2,p} bound}, 
the Sobolev embedding~\eqref{eq:sobolev_inq} 
and the convergence of $(Z_n')_{n\in\mathbb{N}}$ and $(\tau'_n)_{n\in \sN}$,
$\lim_{n\to \infty} \|v_{\tau+h_n\tau_n'}^{\boldsymbol{\pi}_n}-v_\tau^{\boldsymbol{\pi}_\infty}\|_{C^1(\overline{\cO})} =0$.
This shows~\eqref{eq I_1_n} holds.
For $I^2_{n}$ we have 
\begin{align*}
&\left|I^2_{n}-\E^{\mathbb{P}^{x,\boldsymbol{\pi}_\infty}}\int_0^{\tau_\cO}\Gamma^{\boldsymbol{\pi}_\infty}_t\left(\int_Af(X_t,a)\partial\boldsymbol{\pi}(Z)[Z'](da|X_t)\right)dt\right|\\
& \leq \E^{\mathbb{P}^{x,\boldsymbol{\pi}_\infty}}\int_0^{\tau_\cO}\Gamma^{\boldsymbol{\pi}_\infty}_t\left(\int_A\left|f\left(X_t,a\right)\right|\left|\frac{\boldsymbol{\pi}_n-\boldsymbol{\pi}_\infty}{h_n}-\partial\boldsymbol{\pi}(Z)[Z']\right|\left(da|X_t\right)\right)dt\\
& \leq  \|f\|_{B_b(\cO\times A)} \E^{\mathbb{P}^{x,\boldsymbol{\pi}_\infty}}\int_0^{\tau_\cO}
\Gamma^{\boldsymbol{\pi}_\infty}_t
\left\|\left(
\frac{\boldsymbol{\pi}(Z+h_n Z'_n)-\boldsymbol{\pi}(Z)}{h_n}-\partial\boldsymbol{\pi}(Z)[Z']\right)\left(\cdot |X_t\right)
\right\|_{\mathcal{M}(A)}dt\\
& \leq C \left\| \frac{\boldsymbol{\pi}(Z+h_n Z'_n)-\boldsymbol{\pi}(Z)}{h_n}-\partial\boldsymbol{\pi}(Z)[Z']\right\|_{b\mathcal{M}(A|\cO)}\E^{\mathbb{P}^x,  {\boldsymbol{\pi}_\infty}}\left[\tau_\cO\right],
\end{align*}
which converges to zero as $n\rightarrow\infty$, 
due to 
$\E^{\mathbb{P}^x,\boldsymbol{\pi}(Z)}\left[\tau_\cO\right]<\infty$ and
Proposition~\ref{prop:pi_derivative} Item~\ref{prop:mirror_map_derivative}.
For $I^3_{n}$,
note that from
\eqref{eqn:stab_helper_4},  
we have that the intergrand is uniformly bounded in $n$. 
The dominated convergence theorem together with Lemma~\ref{lem:entropy_diff} yields~\eqref{eq I_3_n}.
This finishes the proof.
\end{proof}

\section{Proofs of Theorem~\ref{ref:cost_decrease_along_flow}, Proposition~\ref{thm:convergence_of_GF}  and Theorem~\ref{cor:extend_conv_GF}
}
\label{sec:convergence}

\begin{proof}
[Proof of Theorem~\ref{ref:cost_decrease_along_flow}]
As $Z\in C^1([0,S];B_b(\cO\times A))$  and $\tau\in C^1([0,\infty);(0,\infty))$,  the map $s\mapsto (Z_s,\btau_s)$
is differentiable.
From Proposition \ref{prop:funct_derivative_rho} 
the map $(Z,\tau)\mapsto v^{\boldsymbol{\pi}(Z)}_\tau(x)$ is Hadamard differentiable.
Thus by the chain rule, 
\begin{equation}
\label{eqn:pre_energy_function}
\begin{split}
&\partial_s v^{\boldsymbol{\pi}(Z_s)}_{\btau_s}(x)=\partial v^{\boldsymbol{\pi}(Z_s)}_{\btau_s}(x)[(\partial_s Z_s, \partial_s\btau_s)]\\
&=\E^{\mathbb{P}^{x,\boldsymbol{\pi}(Z_s)}}\int_0^{\tau_\cO}\Gamma^{\boldsymbol{\pi}(Z_s)}_t\int_A\left(\mathcal{L}^av_{\btau_s}^{\boldsymbol{\pi}(Z_s)}(X_t)+f(X_t,a)+\btau_s\ln\frac{d\boldsymbol{\pi}(Z_s)}{d\mu}(a|X_t)\right)\partial\boldsymbol{\pi}(Z_s)[\partial_s Z_s](da|X_t)dt\\
&\enspace\enspace+(\partial_s\btau_s)\E^{\mathbb{P}^{x,\boldsymbol{\pi}(Z_s)}}\int_0^{\tau_{\cO}}\Gamma^{\boldsymbol{\pi}(Z_s)}_t\KL(\boldsymbol{\pi}(Z_s)|\mu)(X_t)dt.
\end{split}
\end{equation}
By Proposition \ref{prop:pi_derivative} Item \ref{prop:mirror_map_derivative} and \eqref{eq:md_intro}, 
we have
\begin{align*}
&\partial\boldsymbol{\pi}(Z_s)[\partial_s Z_s]\left(da|X_t\right)=\left(\partial_s Z_s\left(X_t,a\right)-\int_A\partial_s Z_s\left(X_t,a'\right)\boldsymbol{\pi}(Z_s)\left(da'|X_t\right)\right)\boldsymbol{\pi}(Z_s) \left(da|X_t\right)\\
&=\bigg(-\left[\mathcal{L}^av_{\btau_s}^{\boldsymbol{\pi}(Z_s)}(X_t)+f(X_t,a)+{\btau_s}Z_s(X_t,a)\right]\\
&\enspace\enspace+\int_A\left[\mathcal{L}^{a'}v_{\btau_s}^{\boldsymbol{\pi}(Z_s)}(X_t)+f(X_t,a')+{\btau_s} Z_s(X_t,a')\right]\boldsymbol{\pi}(Z_s)\left(da'|X_t\right)
\bigg)\boldsymbol{\pi}(Z_s)\left(da|X_t\right),
\end{align*}
where the second identity used    the fact that the diffusion coefficient is independent of $a$.
By further    adding and subtracting the control-independent term  $\ln\left(\int_A e^{Z_s(X_t,a'')}\mu(da'')\right)$,
\begin{align*}
&\partial\boldsymbol{\pi}(Z_s)[\partial_s Z_s]\left(da|X_t\right)
\\
&
=-\bigg(\mathcal{L}^a v_{\btau_s}^{\boldsymbol{\pi}(Z_s)}(X_t)+ f(X_t,a)+\btau_s\ln\frac{d\boldsymbol{\pi}(Z_s)}{d\mu} (a|X_t)\\
&\quad -\int_A\left[\mathcal{L}^{a'} v_{\btau_s}^{\boldsymbol{\pi}(Z_s)}(X_t)+ f(X_t,a')+{\btau_s}\ln\frac{d\boldsymbol{\pi}(Z_s)}{d\mu} (a'|X_t)\right]\boldsymbol{\pi}(Z_s)\left(da'|X_t\right)\bigg)\boldsymbol{\pi}(Z_s)\left(da|X_t\right)
\\
&=-\left(\mathcal{L}^av_{\btau_s}^{\boldsymbol{\pi}(Z_s)}(X_t)+f(X_t,a)+\btau_s \ln\frac{d\boldsymbol{\pi}(Z_s)}{d\mu}(a|X_t)\right)\boldsymbol{\pi}(Z_s)\left(da|X_t\right)\,,
\end{align*}
where the last identity used the fact that 
$v_{\btau_s}^{\boldsymbol{\pi}(Z_s)}$ satisfies  \eqref{eq:on_policy_bellman}.
Substituting the identity into \eqref{eqn:pre_energy_function} completes the proof.
\end{proof}

To prove Proposition \ref{thm:convergence_of_GF},  
let  $\Phi:B_b(\cO\times A)\rightarrow B_b(\cO )$  by $\Phi(Z)(x)\coloneqq \ln\left(\int_Ae^{Z(x,a)}\mu(da)\right)$, and 
for each $x\in\cO$
and $\tilde{Z}\in B_b(\cO\times A)$, 
define  
$\mathcal{D}_x^{\boldsymbol{\pi}(\tilde{Z})}:B_b(\cO\times A)\times B_b(\cO\times A)\rightarrow\mathbb{R} $ by
\begin{equation*}
\begin{split}
& \mathcal{D}_x^{\boldsymbol{\pi}(\tilde{Z})}(Z,Z')
\\
&\coloneqq \E^{\mathbb{P}^{x,\boldsymbol{\pi}(\tilde{Z})}}\int_0^{\tau_\cO} \Gamma^{\boldsymbol{\pi}(\tilde{Z})}_t\bigg(\Phi(Z)\left(X_t\right)-\Phi(Z')\left(X_t\right)-\int_A \left(Z\left(X_t,a\right)-Z'\left(X_t,a\right)\right)\boldsymbol{\pi}(Z')\ (da|X_t)\bigg)dt.
\end{split}
\end{equation*}

The following lemma characterizes $\mathcal{D}_x^{\boldsymbol{\pi}(\tilde{Z})}$ as an  integrated KL divergence between two Gibbs policies. 
\begin{lemma}
\label{lem:D_to_RE}
For all   $x\in\cO$ and $Z,Z'\in B_b(\cO\times A)$, 
\begin{align*}
\mathcal{D}_x^{\boldsymbol{\pi}(Z')}(Z,Z')= \E^{\mathbb{P}^{x,\boldsymbol{\pi}(Z')}}\int_0^{\tau_\cO}\Gamma^{\boldsymbol{\pi}(Z')}_t\KL(\boldsymbol{\pi}(Z')|\boldsymbol{\pi}(Z))(X_t)dt\,.
\end{align*}
\end{lemma}
\begin{proof}
Note that for all $f,g \in B_b(A)$,
\[
\begin{split}
& \int_A \bigg(\ln \frac{e^{g(a)}}{\int_A e^{g(a')}\mu(da')} - \ln \frac{e^{f(a)}}{\int_A e^{f(a')}\mu(da')} \bigg) \frac{e^{g(a)}}{\int_A e^{g(a')}\mu(da')}\mu(da)\\
& = \int_A \bigg(g (a)- \ln \int_A e^{g(a')}\,\mu(da') -f (a)+ \ln \int_A e^{f(a')}\,\mu(da')\bigg)\frac{e^{g(a)}}{\int_A e^{g(a')}\mu(da')}\mu(da)\,,
\end{split}
\]
which along with the definition of $\mathcal{D}_x$ yields the desired 
conclusion.
\end{proof}

\begin{proof}
[Proof of Proposition \ref{thm:convergence_of_GF}]
Let $\pi_s =\boldsymbol{\pi}(Z_s)$ for all $s>0$ and $\pi^*_\tau = \boldsymbol{\pi}(Z^*_\tau)$.
Using Proposition   \ref{prop:pi_derivative} Item \ref{prop:ln_pi_diff} and Item \ref{prop:diff_ln_RE}, 
and the chain rule we have 
\begin{align*}
\partial_s & \mathcal{D}^{\pi^*_{\tau}}_x(Z_s,Z_\tau^*)
=\E^{\mathbb{P}^{x,\pi_\tau^\ast}}\int_0^{\tau_\cO}\Gamma^{\pi_\tau^\ast}_t\left(\partial_s\Phi(Z_s)\left(X_t\right)-\int_A\partial_sZ_s\left(X_t,a\right)\pi^*_\tau (da|X_t)\right) dt\\
&=\E^{\mathbb{P}^{x,\pi_\tau^*}}\int_0^{\tau_\cO}\Gamma^{\pi_\tau^\ast}_t\left(\partial\Phi(Z_s)[\partial_s Z_s]\left(X_t\right)-\int_A\partial_sZ_s\left(X_t,a\right)
\pi^*_\tau (da|X_t) \right) dt\\
&=\E^{\mathbb{P}^{x,\pi_\tau^*}}\int_0^{\tau_\cO}\Gamma^{\pi_\tau^\ast}_t\int_A\partial_s Z_s\left(X_t,a\right)\left(\pi_s-\pi^*_\tau \right)\left(da|X_t\right)dt\\
&=\E^{\mathbb{P}^{x,\pi_\tau^*}}\int_0^{\tau_\cO}\Gamma^{\pi_\tau^\ast}_t\int_A\left(\partial_s Z_s\left(X_t,a\right)+\btau_s\ln\left(\int_A e^{Z_s(X_t,a')}\mu(da')\right)\right)\left( \pi_s-\pi^*_\tau \right)\left(da|X_t\right)dt,
\end{align*}
where   interchanging   the differentiation and integration   follows from  
the continuous differentiability of $Z$
and the dominated convergence theorem,
and the last identity used the fact that 
$\ln\left(\int_A e^{Z_s(X_t,a')}\mu(da')\right)$ is   independent of $a$. 
This along with the definition of the   flow \eqref{eq:md_intro}, and  Lemmas \ref{lemma:performance_diff} and  \ref{lem:D_to_RE} implies that
\begin{align*}
\partial_s\mathcal{D}^{\pi^*_{\tau}}_x(Z_s,Z^*_\tau)&=\E^{\mathbb{P}^{x,\pi^*_\tau}}\int_0^{\tau_\cO}\Gamma^{\pi_\tau^\ast}_t\int_A\left(\mathcal{L}^av^{\pi_s}_{\btau_s}(X_t)+f(X_t,a)+\btau_s\ln\frac{d\pi_s}{d\mu}(a|X_t)\right)(\pi^*_\tau-\pi_s)(da|X_t)dt\\
&=(v^{\pi^*_\tau}_{\btau_s}-v^{\pi_s}_{\btau_s})(x)-\btau_s\E^{\mathbb{P}^{x,\pi^*_\tau}}\int_0^{\tau_\cO}\Gamma^{\pi_\tau^\ast}_t\KL(\pi^*_\tau|\pi_s)(X_t)dt=(v^{\pi^*_\tau}_{\btau_s}-v^{\pi_s}_{\btau_s})(x)-\btau_s\mathcal{D}^{\pi^*_{\tau}}_x(Z_s,Z^*_\tau).
\end{align*}
Setting $I_s =e^{\int_0^s \btau_r dr}$ and solving the above ODE yields 
\begin{align}
\label{eqn:solved_ODE}
I_s\mathcal{D}^{\pi^*_{\tau}}_x(Z_s,Z^*_\tau)=\mathcal{D}_x(Z_0,Z^*_\tau)+\int_0^sI_{s'}(v^{\pi^*_\tau}_{\btau_{s'}}-v^{\pi_{s'}}_{\btau_{s'}})(x)ds'\,.
\end{align}
This along with  $I_s \ge 0$ and $\mathcal{D}^{\pi^*_{\tau}}_x(Z_s,Z^*_\tau)\ge 0$ (see Lemma \ref{lem:D_to_RE}) implies 
\begin{align*}
\int_0^sI_{s'}(v^{\pi_{s'}}_{\btau_{s'}}-v^{\pi^*_\tau}_{\btau_{s'}})(x)ds'\leq\mathcal{D}^{\pi^*_{\tau}}_x(Z_0,Z^*_\tau).
\end{align*}
Hence by the definition of the regularized value function   \eqref{eq:value},
\begin{equation}
\label{eqn:conv_pf_help_1}
\begin{split}
\int_0^sI_{s'}(v^{\pi_{s'}}_{\btau_{s'}}-v^{\pi^*_\tau}_\tau)(x)ds'&=\int_0^sI_{s'}(v^{\pi_{s'}}_{\btau_{s'}}-v^{\pi^*_\tau}_{\btau_{s'}})ds'+\int_0^sI_{s'}(v^{\pi^*_\tau}_{\btau_{s'}}-v^{\pi^*_\tau}_\tau)(x)ds'\\
&\leq\mathcal{D}_x(Z_0,Z^*_\tau) + \left(\E^{\mathbb{P}^{x,\pi^*_\tau}}\int_0^{\tau_\cO}\Gamma^{\pi_\tau^\ast}_t\KL(\pi^*_\tau|\mu)(X_t)dt\right)\int_0^sI_{s'}(\btau_{s'}-\tau)^+ds'.
\end{split}
\end{equation}
From Theorem \ref{ref:cost_decrease_along_flow} 
the map $s\mapsto v^{\boldsymbol{\pi}(Z_s)}_{\btau_s}(x)$ is decreasing, hence
\begin{align}
\label{eqn:conv_pf_help_2}
(v^{\pi_{s}}_{\btau_s}-v^{\pi^*_\tau}_\tau)(x)\int_0^sI_{s'}ds'\leq\int_0^sI_{s'}(v^{\pi_{s'}}_{\btau_{s'}}-v^{\pi^*_\tau}_\tau)(x)ds'\,.
\end{align}
Combining \eqref{eqn:conv_pf_help_1} and \eqref{eqn:conv_pf_help_2}
yields
\begin{align*}
(v^{\pi_{s}}_{\btau_s}-v^{\pi^*_\tau}_\tau)(x)\leq\frac{\E^{\pi^*_\tau}\int_0^{\tau_{\cO}}\Gamma^{\pi_\tau^\ast}_t\KL(\pi^*_\tau|\pi_0)(X_t)dt}{\int_0^sI_{s'}ds'}+\frac{\int_0^sI_{s'}(\btau_{s'}-\tau)^+ ds'}{\int_0^sI_{s'}ds'}\E^{\mathbb{P}^{x,\pi^*_\tau}}\int_0^{\cO}\Gamma^{\pi_\tau^\ast}_t\KL(\pi^*_\tau|\mu)(X_t)dt,
\end{align*}
which (recalling the definition of $I_s$) concludes the proof.
\end{proof}
\begin{lemma}
\label{lemma:entropy_bounds}
Suppose Assumption~\ref{ass:data} holds.
Let $\pi^0 \in \Pi_\mu$.
Then there exists $C>0$ such that 
for all $\tau>0$,  
$\|\KL\left(\pi^*_\tau|\pi^0\right)\|_{B_b(\cO)} \leq C(1+\tau^{-1})$
and 
$\|\KL\left(\pi^*_\tau|\mu\right)\|_{B_b(\cO)}\le C/\tau$.  

If $A$ is of finite cardinality,  then $ 
\sup_{\tau>0}(\|\KL\left(\pi^*_\tau|\pi^0\right)\|_{B_b(\cO)}+\|\KL\left(\pi^*_\tau|\mu\right)\|_{B_b(\cO)}) < \infty\,.	
$
\end{lemma}
\begin{proof}
As $\pi_0 \in \Pi_\mu$, there is $Z_0 \in B_b(\mathcal O \times A)$ such that $\pi_0 = \boldsymbol{\pi}(Z_0)$. Then 
\begin{equation}
\label{eq:per_x_KL_est}
\begin{split}
\KL\left(\pi^*_\tau|\pi^0\right)(x)&=\int_A\left(\ln\frac{d\pi^*_\tau}{d\mu}(a|x)-\ln\frac{d\boldsymbol{\pi}(Z_0)}{d\mu}(a|x)\right) \pi^*_\tau(da|x)\\
&\leq\int_A\left|\ln\frac{d\pi^*_\tau}{d\mu}(a|x)\right|\pi^*_\tau(da|x)+\left\|\ln\frac{d\boldsymbol{\pi}(Z_0)}{d\mu}\right\|_{B_b(\cO\times A)}.	
\end{split}
\end{equation}	
From Proposition \ref{prop:verification} we have that $\pi^*_\tau(da|x)=\boldsymbol{\pi}(Z^*_\tau)(da|x)$ and $\left|\ln\frac{d\pi^*_\tau}{d\mu}(a|x)\right|\leq 2\|Z^*_\tau\|_{B_b(\cO\times A)}$.
Recalling the definition of $Z^*_\tau$ and using the  Sobolev embedding we have
\begin{equation}
\label{eqn:sup_x_KL}
\|Z^*_\tau\|_{B_b(\cO\times A)}
\leq\frac{2K}{\tau}(1+\|v^*_\tau\|_{C^1(\cO)})
\leq\frac{C}{\tau}(1+\|v^*_\tau\|_{W^{2,p^*}(\cO)})\,.
\end{equation}
From Proposition~\ref{prop:verification} $v^*_\tau$ is the unique strong solution to \eqref{eq:semilinear}, therefore using the estimate provided in Lemma \ref{lemma:hjb_a_priori} 
(with $\eta = 1$) there exists a constant $C$, independent of $\tau$, such that  $\|v^*_\tau\|_{W^{2,p^*}(\cO)}\leq C(1+\|g\|_{W^{2,p^*}(\cO)})$.
Hence there exists   $C\ge 0$ such that for all $\tau>0$ and $x\in \cO$,
\[
\begin{split}
\KL\left(\pi^*_\tau|\pi^0\right)(x) \leq C + 2\int_A\|Z^*_\tau\|_{B_b(\cO\times A)}\pi^*_\tau(da|x) \leq C + \frac{C}{\tau}(1+\|g\|_{W^{2,p^*}(\cO)}) \leq C(1+\tau^{-1})\,,
\end{split}
\]
which proves the first statement. 

To prove the second statement assume $A = \{a_1,\ldots,a_n\}$. 
Since $\pi^0 = \boldsymbol{\pi}(Z_0)$ we have that $\pi^0(a_i|x) \in (0,1)$ for all $i=1,\ldots,n$.
Then for all $x \in \mathcal O$ we have $\pi_\tau^\ast(\cdot|x)\ll\mu$ and hence
\[
\begin{split}
\operatorname{KL}(\pi^\ast_\tau|\pi^0)(x) & = \sum_{a_i\in A} \bigg( \ln \frac{\pi^\ast_\tau(a_i|x)}{\mu(a_i)} - \ln\frac{\pi^0(a_i|x)}{\mu(a_i)}\bigg)\pi^\ast_\tau(a_i|x)\\
& \leq \sum_{a_i\in A} (\ln \pi^\ast_\tau(a_i|x) - \ln \mu(a_i) )\pi^\ast_\tau(a_i|x) + \sum_{a_i\in A}  \bigg|\ln\frac{\pi^0(a_i|x)}{\mu(a_i)}\bigg|\pi^\ast_\tau(a_i|x)\\ 
& \leq \sum_{a_i\in A} \bigg( |\ln \mu(a_i)| + \bigg|\ln\frac{\pi^0(a_i|x)}{\mu(a_i)}\bigg|\bigg) \pi^\ast_\tau(a_i|x) \leq \max_{i=1,\ldots,n} \bigg( |\ln \mu(a_i)| + \bigg|\ln\frac{\pi^0(a_i|x)}{\mu(a_i)}\bigg|\bigg) 
=: C\,.	
\end{split}
\] 
This concludes the proof.
\end{proof}

\begin{proof}[Proof of Theorem~\ref{cor:extend_conv_GF}]
From Proposition \ref{thm:convergence_of_GF} and the non-negativity of the KL-divergence we have that
\begin{align*}
\label{eqn:bound_no_kl}
v^{\boldsymbol{\pi}(Z_s)}_{\btau_{s}}(x)-v^{\pi^*_{{\tau}}}_{{\tau}}(x)
&\leq
C\left(\frac{1}
{\int_0^s I_{s'}ds'}\|\KL(\pi^*_\tau|\boldsymbol{\pi}(Z_0))\|_{B_b(\mathcal O)} + \frac{\int_0^s(\btau_{s'}-{\tau}) I_{s'}ds'}{\int_0^s I_{s'}ds'}\|\KL(\pi^*_\tau|\mu)\|_{B_b(\mathcal O)}\right)\,,
\end{align*}
where $C\coloneqq \sup_{  \tau> 0}\E^{\mathbb{P}^{x,\pi^*_{{\tau}}}}\int_0^{\tau_\cO}\Gamma^{\pi^*_{{\tau}}}_t dt 
< \sup_{\pi} \E^{\mathbb{P}^{x,\pi}}\int_0^{\tau_\cO}\Gamma^{\pi}_t dt
< \infty$ by~\cite[Ch.~2, Sec.~2, Theorem 4, p.~54]{krylov2008controlled}.
The conclusion then follows from Lemma~\ref{lemma:entropy_bounds}. 
\end{proof}

\section{Proof of Theorem~\ref{thm:convergece_tau}}
\label{sec:van_reg_proofs}

We first prove that the unregularized HJB equation  \eqref{eq:hjb_unregualarised} admits a unique strong solution and quantify the difference between solutions of \eqref{eq:semilinear} and \eqref{eq:hjb_unregualarised} using the difference
of two Hamiltonians $H_\tau - H$.
The proof is given in 
Appendix~\ref{sec:well-posedness of HJB}.

\begin{proposition}
\label{prop:wp_unregularized}
Suppose Assumption   \ref{ass:data} holds and the function  $H : \cO \times   \sR\times \sR^{d}\to \sR$ in  \eqref{eq:Hamiltonian_unregularized} is measurable.  
Then~\eqref{eq:hjb_unregualarised} admits a unique solution  $\bar{v} \in  W^{2,p^*}(\cO)$ with $p^*$ as in Assumption~\ref{ass:data}.
Moreover,   
there exists $C\ge 0$ such that for all $\tau>0$, 
\[
\|v^*_\tau - \bar{v} \|_{W^{2,p^*}(\cO)} \le C\| \left( H_\tau(\cdot, \bar{v}(\cdot), D\bar{v}(\cdot))- H(\cdot, \bar{v}(\cdot), D\bar{v}(\cdot))\right)^+\|_{L^{p^*}(\cO)}\,,
\]
where $v^*_\tau$ is the solution to \eqref{eq:semilinear}.
\end{proposition}

\begin{proof}[Proof of Theorem~\ref{thm:convergece_tau}]
Under Assumption~\ref{assum:lsc}, 
for all $(x,u,p)\in \cO\times  \sR\times \sR^d$,
$a\mapsto b(x,a)^\top p -c(x,a)u+f(x,a)$ is   continuous,
and hence by
\cite[Theorem 18.19]{aliprantis2006infinite}, 
$H$ is Borel measurable and 
there exists a Borel measurable function  
$\phi:  \cO \times   \sR\times \sR^{d}\to A$ such that  
\begin{equation}
\label{eq:optimal_action}
\phi  (x,u,p) \in \argmin_{a\in A}  \left( b(x,a)^\top p -c(x,a)u+f(x,a)  \right),\quad
\forall  (x,u, p )\in \cO \times   \sR\times \sR^{d}\,.
\end{equation}
This along with  Proposition~\ref{prop:wp_unregularized} implies that  the HJB equation~\eqref{eq:hjb_unregualarised} admits a unique strong solution $\bar{v}\in W^{2,p^*}(\cO)$. 
Define the candidate control {$\pi^*_0 \in \mathcal P(A|\sR^d)$ such that  $\pi^*_0(x) = \delta_{\phi(x,\bar{v}(x), D\bar{v}(x))}$} for all $x\in \cO$. 
Then using  the generalised It\^{o}'s formula~\cite[Theorem 1, p.~122]{krylov2008controlled}
and  standard verification arguments    (see e.g.,~\cite[Theorem 2.2]{reisinger2021regularity}),
one can show that  $v^*_0 \equiv \bar{v} $ and $\pi^*_0$ is an optimal control for the unregularized problem.

Now observe that  
for all $x\in \overline{\cO}$, 
\begin{align}
\label{eq:monotone_values}
0\le  v^{\pi^*_\tau}_0(x) - v^*_0(x) \le v^*_\tau (x) - v^*_0(x)\,, 
\end{align}
where we used  
$v^{\pi^*_\tau}_0 \le  v^{\pi^*_\tau}_\tau = v^*_\tau$
since $\textrm{KL}({\pi_\tau^\ast}|\mu)(x)\ge 0$ for all $x\in \sR^d$. 
Using that $v_0^\ast = \bar v$, Proposition \ref{prop:wp_unregularized} and  the Sobolev embedding theorem~\cite[Theorem 7.26]{gilbarg1977elliptic}, there exists  $C\ge 0$ such that  for all $\tau>0$,
\begin{align*}
\|v^*_\tau -  {v}^*_0 \|_{C(\overline{\cO})} \le C\| \left( H_\tau(\cdot,  {v}^*_0(\cdot), D {v}^*_0(\cdot))- H(\cdot,  {v}^*_0(\cdot), D {v}^*_0(\cdot))\right)^+\|_{L^{p^*}(\cO)}\,,
\end{align*}
which along with \eqref{eq:monotone_values} implies that 
there exists $C\ge 0$ such that for all $\tau>0$ and $x\in \overline{\cO}$,
\begin{align*}
0\le  v^{\pi^*_\tau}_0(x) - v^*_0(x) \le C\| \left( H_\tau(\cdot,  {v}^*_0(\cdot), D {v}^*_0(\cdot))- H(\cdot,  {v}^*_0(\cdot), D {v}^*_0(\cdot))\right)^+ \|_{L^{p^*}(\cO)}\,.
\end{align*}

It remains to prove 
$\lim_{\tau \to 0}\| \left( H_\tau(\cdot,  {v}^*_0(\cdot), D {v}^*_0(\cdot))- H(\cdot,  {v}^*_0(\cdot), D {v}^*_0(\cdot))\right)^+ \|_{L^{p^*}(\cO)}=0$. 
We first claim that for all $(x,u,p)\in \cO\times  \sR\times \sR^d$,
$\lim_{\tau \to 0} H_\tau(x, u, p) = H(x,  u, p)$. 
To see it, let $(x,u,p)\in \cO\times  \sR\times \sR^d$ be fixed, and 
recall that  for any measure space  $(E,\mathcal A, \nu)$ with $\nu(A)<\infty$
and 
any  bounded measurable function $g:E\to \sR$, 
$\lim_{p\to \infty}\left(\int_E |g(x) |^p\nu(\d x)\right)^{1/p}=\|g\|_{L^\infty(E,\nu)} $,
where $\|g\|_{L^\infty(E,\nu)} =\inf\{C\ge 0||g(x)|\le C,\;  \textnormal{for $\nu$-a.s.~$x\in E$}\} $ is the essential supremum of $g$ with respect to $\nu$.
Hence setting   $g: A\to \sR$ with  
$g(a) =  \exp\left(-({b(x,a)^\top p -c(x,a)u+f(x,a)} )\right)$ for all $a\in A$   and $p=1/\tau$ yields
\begin{align*}
\lim_{\tau\to 0}\left(\int_A |g(a) |^{1/\tau} \mu(\d a)\right)^{\tau}= \|g\|_{L^\infty(A,\mu)}\,,
\end{align*}
which along with
$g(a)>0$ for all $a\in A$ and the continuity of $(0,\infty)\ni x\mapsto \ln x\in \sR$ implies that 
\begin{align}
\label{eq:lim_H_tau}
\begin{split}
\lim_{\tau \to 0} H_\tau(x, u, p)  &= \lim_{\tau \to 0}  -\tau \ln\left(\int_A \exp\left(-\frac{b(x,a)^\top p -c(x,a)u+f(x,a)}{\tau }\right) \mu(da)\right)
\\
& = - \ln \left\|\exp\left(-({b(x,\cdot)^\top p -c(x,\cdot)u+f(x,\cdot)} )\right) \right\|_{L^\infty(A,\mu)}\,.
\end{split}
\end{align}
Now by the definition of the $\|\cdot\|_{L^\infty(A,\mu)}$-norm, there exists a set $N\subset A$ such that $\mu({N})=0$ and for all $a\in A\setminus N$,   
\[
\exp\left(-({b(x,a)^\top p -c(x,a)u+f(x,a)} )\right)\le \left\|\exp\left(-({b(x,\cdot)^\top p -c(x,\cdot)u+f(x,\cdot)} )\right) \right\|_{L^\infty(A,\mu)}\,.
\]
By  Assumption  \ref{assum:lsc} Item \ref{item:dense}, $A\setminus N$ is dense in $A$.     
Hence  for all $a\in A$, choosing   $(a_n)_{n\in \sN}\subset A\setminus N$ such that $\lim_{n\to \infty} a_n=a$
and using   the continuity of  
$a\mapsto  \exp\left(-({b(x,a)^\top p -c(x,a)u+f(x,a)} )\right)$ give that 
\begin{align*}
\exp\left(-({b(x,a)^\top p -c(x,a)u+f(x,a)} )\right)
&= \lim_{n\to \infty}  \exp\left(-({b(x,a_n)^\top p -c(x,a_n)u+f(x,a_n)} )\right)
\\
&\le   \left\|\exp\left(-({b(x,\cdot)^\top p -c(x,\cdot)u+f(x,\cdot)} )\right) \right\|_{L^\infty(A,\mu)}\,.
\end{align*}
This together with the compactness of $A$
shows that 
\begin{align*}
& \left\|\exp\left(-({b(x,\cdot)^\top p -c(x,\cdot)u+f(x,\cdot)} )\right) \right\|_{L^\infty(A,\mu)}
=\max_{a\in A} \exp\left(-({b(x,a)^\top p -c(x,a)u+f(x,a)} )\right)
\\
&\quad = \exp\left(-\min_{a\in A}({b(x,a)^\top p -c(x,a)u+f(x,a)} )\right)
=   \exp\left(-H(x, u,p) \right)\,,
\end{align*}
which along with 
\eqref{eq:lim_H_tau} implies that 
$\lim_{\tau \to 0} H_\tau(x, u, p) = H(x,  u, p)$. 
Consequently, for a.e.~$x\in \cO$,
\[
\lim_{\tau \to 0}  \left( H_\tau(x, v^*_0(x), Dv^*_0(x)) - H(x, v^*_0(x), Dv^*_0(x))\right)^+=0\,.
\]
By 
 \eqref{eq:Hamiltonian_tau}, the boundedness of coefficients and 
the Sobolev embedding 
$v^*_0 \in W^{2,p^*}(\cO)\subset  C^{1}(\overline{\cO})$, 
$$
\sup_{x\in \cO, \tau>0} |H_\tau(x, v^*_0(x), Dv^*_0(x))- H(x, v^*_0(x), Dv^*_0(x))|<\infty\,.
$$
Hence
$\lim_{\tau \to 0}\| ( H_\tau(\cdot,  {v}^*_0(\cdot), D {v}^*_0(\cdot))- H(\cdot,  {v}^*_0(\cdot), D {v}^*_0(\cdot)))^+\|_{L^{p^*}(\cO)}=0$ due to the dominated convergence theorem. 
This finishes the proof.  
\end{proof}

\section{Proofs of Theorems \ref{thm:conv_discrete_anneal} and \ref{thm:conv_general_anneal}}
\label{sec:proof_anneal}

\begin{proof}[Proof of Theorem \ref{thm:conv_discrete_anneal}]

By  \eqref{eqn:modified_conv} and
Theorems \ref{thm:convergece_tau} and \ref{ex:discrete_action_space},
there exists $C>0$ such that for all $s>0$, 
\begin{align}
\label{eq:error_anneal_discrete}
\begin{split}
0\le v^{{\boldsymbol{\pi}(Z_s)}  }_{0}(x)-v^{*}_{0}(x)
&\le v^{{\boldsymbol{\pi}(Z_s)}  }_{\btau_s}(x)-v^{*}_{\btau_s}(x)+ v^{*}_{\btau_s}(x)- v^{*}_{0}(x)
\\
& \leq
C \left( \frac{1 }
{\int_0^s e^{\int_0^{s'}\btau_{r}dr} d{s'}}
+\frac{\int_0^s(\btau_{s'}-{\btau_s})  e^{\int_0^{s'}\btau_{r}dr} ds'}{\int_0^s e^{\int_0^{s'}\btau_{r}dr} ds'} +\btau_s \right)\,.
\end{split}
\end{align}
Since $\btau_s =1/(1+s)$ for all $s>0$, 
$e^{\int_0^{s'}\btau_{r}dr}=e^{\ln (s'+1)}=s'+1$,  
$\int_0^s e^{\int_0^{s'}\btau_{r}dr} d{s'}=\frac{1}{2}s^2+s$, and 
$$
\frac{\int_0^s(\btau_{s'}-{\btau_s})^+ e^{\int_0^{s'}\btau_{r}dr} ds'}{\int_0^s e^{\int_0^{s'}\btau_{r}dr} ds'}
=\frac{\int_0^s\btau_{s'}  e^{\int_0^{s'}\btau_{r}dr} ds'}{\int_0^s e^{\int_0^{s'}\btau_{r}dr} ds'}-{\btau_s}
=\frac{1}{\frac{1}{2}s+1 }-\frac{1}{s+1}= \frac{ s}{(s+1)( s+2 )}\,.
$$
This along with \eqref{eq:error_anneal_discrete} proves the desired estimate. 
\end{proof}

\begin{proof}[Proof of Theorem \ref{thm:conv_general_anneal}]
By  \eqref{eqn:modified_conv},
there exists $C>0$ such that for all $s>0$, 
\begin{align}
\label{eq:error_anneal_general}
\begin{split}
0\le v^{{\boldsymbol{\pi}(Z_s)}  }_{0}(x)-v^{*}_{0}(x)
&\le v^{{\boldsymbol{\pi}(Z_s)}  }_{\btau_s}(x)-v^{*}_{\btau_s}(x)+ v^{*}_{\btau_s}(x)- v^{*}_{0}(x)
\\
& \leq
C \left( \frac{1 }
{\btau_s \int_0^s e^{\int_0^{s'}\btau_{r}dr} d{s'}}
+\frac{\int_0^s \btau_{s'}  e^{\int_0^{s'}\btau_{r}dr} ds'}{\btau_s \int_0^s e^{\int_0^{s'}\btau_{r}dr} ds'}-1 \right)
+v^{*}_{\btau_s}(x)- v^{*}_{0}(x) \,,
\end{split}
\end{align}
where we  used 
$\btau $ is uniformly bounded.
As   $\btau_s =1/\sqrt{s+1}$ for $s>0$, 
$
\int_0^s \btau_r \,dr = 2 { \sqrt{1+s}}-2$,
and    
\[
\int_0^s e^{\int_0^{s'} \btau_r \,dr}\,d{s'} = \int_0^s e^{2{ \sqrt{1+s'}}-2}\,ds' 
=\frac{e^{-2} }{2}\int_2^{2\sqrt{1+s}} e^{y}y dy
= \frac12 \left( e^{2{ \sqrt{1+s}}-2}\left(2{ \sqrt{1+s}} - 1\right) - 1\right)\,,
\]
where the last identity used  the integration by part formula.
Hence there exists $C>0$ and $S_0>0$  such that  the first term in \eqref{eq:error_anneal_general} can  be upper bounded by
\begin{equation*}
\frac{1}{\btau_s\int_0^s e^{\int_0^{s'}\btau_{r}dr}ds'} = \frac{\sqrt{1+s}}{\frac{1}{2} \left( e^{2{ \sqrt{1+s}}-2}\left(2{ \sqrt{1+s}} - 1\right) - 1\right)}  \le C e^{-2{ \sqrt{1+s}} },
\quad  \forall s\ge S_0\,.
\end{equation*}
For the second term  in \eqref{eq:error_anneal_general},
by setting $y=2\sqrt{1+s'}-2$ with $dy =\frac{1}{1+s'}ds'$,
\begin{equation*}
\int_0^s \btau_{s'} e^{\int_0^{s'} \btau_r \,dr}\,d{s'}  
=\int_0^s \frac{1}{\sqrt{1+{s'}}} e^{2\sqrt{1+{s'}}-2}\,d {s'} 
=\int_0^{2\sqrt{1+s}-2}  e^{y}\,dy=e^{2\sqrt{1+s}-2}-1 .
\end{equation*}
Hence for all sufficiently large $s>0$,
\begin{align}
\label{eq:dominating_term_general}
\frac{\int_0^s \btau_{s'}  e^{\int_0^{s'}\btau_{r}dr} ds'}{\btau_s \int_0^s e^{\int_0^{s'}\btau_{r}dr} ds'}-1
&=\frac{\sqrt{1+s}(e^{2\sqrt{1+s}-2}-1)}{
e^{2 \sqrt{1+s}-2} \left({ \sqrt{1+s}} - \frac{1}{2} \right) - \frac{1}{2}}-1
=\frac{ \frac{1}{2} e^{2 \sqrt{1+s}-2} -\sqrt{1+s} + \frac{1}{2}}{
e^{2 \sqrt{1+s}-2} \left({ \sqrt{1+s}} - \frac{1}{2} \right) - \frac{1}{2}}
\le \frac{C}{\sqrt{s}}.
\end{align}
Finally, 
the last term  in \eqref{eq:error_anneal_general} converges to zero as $s\to \infty $
under Assumptions~\ref{ass:data} and \ref{assum:lsc} due to Theorem \ref{thm:convergece_tau},
and 
is bounded by $C(\ln s)^\alpha /\sqrt{s}$ for all large $s$ under  Assumption \ref{assum:H_tau}. This completes the proof. 
\end{proof}

\section{Proofs of Propositions \ref{ex:unique_min_in_conv_comp_set} and   \ref{prop:EG_LQR}}
\label{proof:lapace_asymptotic}

Proposition \ref{ex:unique_min_in_conv_comp_set} follows directly from the following lemma
(with  $h$ defined in Proposition \ref{ex:unique_min_in_conv_comp_set}). 

\begin{lemma}
\label{lemma:minimiser_interior}
Let 
$\cO\subset \sR^d$ be a bounded domain,
$A\subset \sR^k$ be a nonempty convex and compact set and $\mu\in \cP(A)$ be the uniform distribution on $A$.
Let $h\in C(\overline{\cO}\times A;\sR)$ be    such that for all $x\in \overline{\cO}$,
$A\ni  a\mapsto  h(x,a)\in \sR$
admits a unique minimiser in the interior of $A$ and 
is twice  differentiable
with   derivative $D^2_{aa}h \in C(\overline{\cO}\times A;\sR^{k\times k})$.  
Then 
there exists $C\ge 0$ and $\tau_0>0$ such that for all   $x\in \overline{\cO}$
and all $\tau\in (0,\tau_0]$, 
$$
0\le -\tau\ln\left(\int_A \exp\left(-\frac{h(x,a)}{\tau}\right)\mu(da)\right) 
-\min_{a\in A} h(x,a)\le   C  \tau \ln\frac{1}{\tau}\,.
$$

\end{lemma} 

\begin{proof}
Throughout this proof, 
for any   $U\subset \sR^k$ and $\varepsilon>0$, we define 
${B}_\varepsilon(U)\coloneqq \{a\in A\mid d(a,U )< \varepsilon\}$
with $d(a,U)=\inf\{|a-x|\mid x\in U\}$,
and define $\bar{B}_\varepsilon(U) =\{a\in A\mid d(a,U )\le  \varepsilon\} $.
We denote by
$\textrm{int}(A)$ the interior of $A$.

Consider  the map $\phi:\overline{\cO}\to A$ such that 
$\phi(x)=\argmin_{a\in A}h(x,a)$ for all $x\in \overline{\cO}$.
Note that $\phi$ is 
a upper hemicontinuous corresponding due to Berge Maximum Theorem (see   \cite[Theorem 17.31]{aliprantis2006infinite}) and single-valued due to the assumption. 
Hence by    \cite[Lemma 17.6]{aliprantis2006infinite},
$\phi:\overline{\cO}\to A $ is a continuous function, which along with the compactness of
$\overline{\cO}$ implies that the image  
$\phi(\overline{\cO})\coloneqq \{\phi(x)\mid x\in \overline{\cO}\}$ is compact. 
As it is assumed that  
$\phi(\overline{\cO}) \subset \textrm{int}(A)$, 
there exists 
$\varepsilon>0$ such that  $\bar{B}_\varepsilon(\phi(\overline{\cO}))  \subset \textrm{int}(A)$.

Now fix  $x\in \phi(\overline{\cO})$ and observe that 
\begin{align}
\label{eq:A_decompose}
\int_A \exp\left(-\frac{h(x,a)}{\tau}\right)\mu(da)
= \int_{ {B}_\varepsilon(\phi(x))} \exp\left(-\frac{h(x,a)}{\tau}\right)\mu(da)
+\int_{A\setminus  {B}_\varepsilon(\phi(x))} \exp\left(-\frac{h(x,a)}{\tau}\right)\mu(da)\,.
\end{align}
We first   estimate the first   term on the right hand side of \eqref{eq:A_decompose}.
As $h(x,\cdot)$ is twice continuously differentiable and $\phi(x)\in \textrm{int}(A)$, 
$D_a h(x,\phi(x))=0$ and by the mean value theorem, 
\begin{align}
\label{eq:quadratic_approximation}
h(x,a)=h(x,\phi(x))+\frac{1}{2}(a-\phi(x))^\top D^2_{aa}h(x,a_x)(a-\phi(x))\,,
\end{align}
for some $a_x\in A$ on the line segment connecting $a$ and $\phi(x)$. 
Let $M_h>0$ be such that 
$v^\top D^2_{aa}h(x,a)v\le M_h|v|^2$
for all $v\in \sR^k$, $x \in  \phi(\overline{\cO})$ and $a\in \phi(\overline{\cO})$.
The existence of $M_h$ is ensured by the continuity of $D^2_{aa}h$  and the compactness of $\overline{\cO}\times \phi(\overline{\cO})$. Using  \eqref{eq:quadratic_approximation}, the fact that 
$\mu(da)=\frac{da}{|A|_{\lambda^k}} $ with $|A|_{\lambda^k}$ being the Lebesgue measure of $A$,
and the change of variables formula,
\begin{align*}
&  \int_{ {B}_\varepsilon(\phi(x))} \exp\left(-\frac{h(x,a)}{\tau}\right)\mu(da)
\ge   \int_{ {B}_\varepsilon(\phi(x))} \exp\left(-\frac{h(x,\phi(x))+\frac{M_h}{2}|a-\phi(x)|^2}{\tau}\right)\frac{da}{|A|_{\lambda^k}} 
\\
&= 
\exp\left(-\frac{h(x,\phi(x))}{\tau}\right)  \frac{1}{|A|_{\lambda^k}} 
\int_{ {B}_\varepsilon(0)} \exp\left(- {\frac{M_h  }{2\tau}|a|^2} \right)  da 
\\
&= 
\exp\left(-\frac{h(x,\phi(x))}{\tau}\right)  \frac{1}{|A|_{\lambda^k}}   \int_0^\varepsilon
\int_{\partial {B}_1(0)} \exp\left(- {\frac{M_h }{2\tau}t^2} \right)t^{k-1} dS  dt
\\
&= 
\exp\left(-\frac{h(x,\phi(x))}{\tau}\right)  \frac{|\partial {B}_1(0)|_S}{|A|_{\lambda^k}}  
\left(
\int_0^\infty
\exp\left(- {\frac{M_h }{2\tau}t^2} \right)t^{k-1}   dt
-       \int_\varepsilon^\infty
\exp\left(- {\frac{M_h }{2\tau}t^2} \right)t^{k-1}   dt
\right)\,, 
\end{align*}
where $dS$ denotes the surface measure on the boundary of $ {B}_1(0)$. 
Note that for all   $\nu>0$, 
\begin{align*}
\int_0^\infty t^{k-1} e^{-\nu t^2}dt=\frac{1}{2\nu}  \int_0^\infty \left(\frac{z}{\nu}\right)^{\frac{k-2}{2}} e^{-z}dz= 
\frac{1}{2}\nu^{-\frac{k}{2}}\Gamma\left( \frac{k}{2} \right)\,,
\end{align*}
where $\Gamma$ is the gamma function.
A straightfoward induction argument further  shows that   there exists $K_\varepsilon\ge 0$ such that
$\int_\varepsilon^\infty t^{k-1} e^{-\nu t^2}dt \le K_\varepsilon e^{-\nu \varepsilon^2}   
$  for all $\nu>1$.
Hence  for all $\tau<\frac{M_h}{2}$, 
\begin{align}
\label{eq:integral_decompose_near_min}
\begin{split}
& \int_{ {B}_\varepsilon(\phi(x))} \exp\left(-\frac{h(x,a)}{\tau}\right)\mu(da)
\\
& \ge 
\exp\left(-\frac{h(x,\phi(x))}{\tau}\right)  \frac{|\partial {B}_1(0)|_S}{|A|_{\lambda^k}} \left(
\frac{1}{2}\left(\frac{2\tau }{M_h}\right)^{\frac{k}{2}}\Gamma\left( \frac{k}{2} \right)- K_\varepsilon e^{-\frac{M_h }{2\tau} \varepsilon^2}
\right)\,.
\end{split}
\end{align} 

To estimate the second term   on the right hand side of \eqref{eq:A_decompose},
consider the set-valued map $\psi:\overline{\cO}\twoheadrightarrow A$
such that 
$\psi (x) \coloneqq A\setminus  {B}_\varepsilon(\phi(x))$ for all $x\in \overline{\cO}$.
By the compactness of $A$, 
$\psi(x)$ is compact for all $x\in \overline{\cO}$.
We claim that 
$\psi$ is   upper hemicontinuous. To see it,  let $\{(x_n,y_n)\}_{n\in \sN}$ be a sequence such that 
$x_n\in  \overline{\cO}$ and $y_n\in \psi(x_n)$ for all $n\in \sN$, 
and $\lim_{n\to \infty}x_n=x$ for some $x\in  \overline{\cO}$. 
This implies that $ |y_n-x_n|\ge  \varepsilon $ for all $n\in \sN$. As $A$ is compact, 
there exists a subsequence
$\{y_{n_k}\}_{k\in \sN}$
of $\{y_n\}_{n\in \sN}$ such that $\lim_{k\to \infty} y_{n_k}=y$ for some $y\in A$. 
Then $|y-x|=\lim_{k\to \infty} |y_{n_k}-x_{n_k}|\ge \varepsilon$, which shows that  $y\in \psi(x)$
and subsequently  the upper hemicontinuity of $\psi$ due to    \cite[Theorem 17.20]{aliprantis2006infinite}. 
Now consider    $m: \overline{\cO}\to \sR$ such that for all $x\in \overline{\cO}$, 
$$
m(x)\coloneqq \min_{a\in \psi(x)}(h(x,a)-h(x,\phi(x)))= - \max_{a\in \psi(x)}( h(x,\phi(x))-h(x,a)) \,.
$$
As   $h$ and $\phi$ are continuous, 
$(x,a)\mapsto h(x,\phi(x))-h(x,a)$ is continuous on $\overline{\cO}\times A$. 
This along with the upper hemicontinuity of $\psi$ and \cite[Lemma 17.30]{aliprantis2006infinite} implies that 
$m$ is lower semicontinuous. 
Note that
for each $x\in \overline{\cO}$, 
since  $\phi(x)$ is the  unique minimiser of $a\mapsto h(x,a)$, 
$ h(x,a)-h(x,\phi(x))>0$ for all $a\in \psi(x)=A\setminus   {B}_\varepsilon(\phi(x))$,
which along with the compactness of $\psi(x)$ implies $m(x)>0$. 
The lower semicontinuity of $m$ and the compactness of $\overline{\cO}$ 
then imply that there exists $m_\varepsilon>0$ such that 
$ h(x,a)-h(x,\phi(x))\ge m_\varepsilon$ for all $x\in \overline{\cO}$ and $a\in A\setminus     {B}_\varepsilon(\phi(x))$.
Hence
\begin{align}
\label{eq:integral_decompose_away_min}
\begin{split}
&\int_{A\setminus  {B}_\varepsilon(\phi(x))} \exp\left(-\frac{h(x,a)}{\tau}\right)\mu(da)
\\
& =  \exp\left(-\frac{h(x,\phi(x))}{\tau}\right)  \int_{A\setminus  {B}_\varepsilon(\phi(x))} \exp\left(-\frac{h(x,a)-h(x,\phi(x))}{\tau}\right)\mu(da)
\\
&
\le  \exp\left(-\frac{h(x,\phi(x))}{\tau}\right) \exp\left(-\frac{m_\varepsilon}{\tau}\right)\,.
\end{split}
\end{align}

Combining \eqref{eq:integral_decompose_near_min} and \eqref{eq:integral_decompose_away_min} gives  
for all $\tau<\frac{M_h}{2}$, 
\begin{align*}
& \int_A \exp\left(-\frac{h(x,a)}{\tau}\right)\mu(da)
\\
&\ge \exp\left(-\frac{h(x,\phi(x))}{\tau}\right)  \frac{|\partial {B}_1(0)|_S}{|A|_{\lambda^k}}
\bigg(
\frac{1}{2}\left(\frac{2\tau }{M_h}\right)^{\frac{k}{2}}\Gamma\left( \frac{k}{2} \right)
- K_\varepsilon e^{-\frac{M_h }{2\tau} \varepsilon^2}
-  \frac{|A|_{\lambda^k}}{|\partial {B}_1(0)|_S}  e^{-\frac{m_\varepsilon}{\tau}}\bigg)\,.
\end{align*}
Note that  for any $C>0$ and $k\in \sN$, $\lim_{\tau\to 0}\frac{e^{-\frac{C}{\tau}}}{\tau^k}
=\lim_{x\to \infty } e^{-Cx }x^k=0$. Hence there exists $\tau_0>0$ and $C>0$ such that 
for all $\tau\in (0,\tau_0]$, 
\begin{align*}
\int_A \exp\left(-\frac{h(x,a)}{\tau}\right)\mu(da)
&\ge C \exp\left(-\frac{h(x,\phi(x))}{\tau}\right) 
\tau^{\frac{m+1}{2}}\,,
\end{align*}
which implies that 
$$
-\tau\ln\left( \int_A \exp\left(-\frac{h(x,a)}{\tau}\right)\mu(da)\right)
\le -\tau \ln C +h(x,\phi(x))  -\tau \frac{m+1}{2} \ln \tau \,.
$$
This along with the fact that 
$h(x,\phi(x)) =\min_{a\in A} h(x,a)$ completes the proof.  
\end{proof}

The following lemma will be used to prove Proposition \ref{prop:EG_LQR}.

\begin{lemma}
\label{lemma:H_tau_interval}
Let $A=[\alpha,\beta]$ for some $-\infty<\alpha<\beta<+\infty$ and $\mu\in \cP(A)$ be the uniform distribution on $A$.
Let  $\mathfrak{h}_\tau: \sR\to \sR$,   $\tau>0$,  be such that 
$\mathfrak{h}_\tau(p)=-\tau\ln\left(\int_A \exp(-\frac{pa+\frac{1}{2}a^2}{\tau})\mu(da)\right)$
for all $p\in \sR$, 
and let $\mathfrak{h}:\sR\to \sR$ be  such that  
$\mathfrak{h}(p)=\min_{a\in A}(pa+\frac{1}{2}a^2) $
for all $p\in \sR$. 
Then for each $M> 0$, there exists    $\tau_0>0$ and $C\ge 0$ such that for all   $p\in [-M,M]$
and all $\tau\in (0,\tau_0]$, 
$$
0\le \mathfrak{h}_\tau(p)-\mathfrak{h}(p) \le C  \tau \ln\frac{1}{\tau}\,.
$$
\end{lemma}

\begin{proof}
It is easy to see from   $\mu\in \cP(A)$ and the definition of $\mathfrak{h}$ that   $\mathfrak{h}_\tau(p)\ge \mathfrak{h}(p)$. Hence it suffices to obtain an upper bound 
of $\mathfrak{h}_\tau(p)- \mathfrak{h}(p)$. 
Observe that by a change of variable, for all $p\in \sR$, 
\begin{align}
\label{eq:change_of_variable}
\begin{split}
\int_A e^{-\frac{pa+\frac{1}{2}a^2}{\tau}}\mu(da)
&=\frac{1}{\beta-\alpha}\int_{\alpha}^{\beta } e^{-\frac{pa+\frac{1}{2}a^2}{\tau}} da
=\frac{1}{\beta-\alpha}e^{\frac{p^2}{2\tau}}\int_{\alpha}^{\beta}  e^{-\frac{ (a+p)^2}{2\tau}}da
\\
&=
\frac{1}{\beta-\alpha}e^{\frac{p^2}{2\tau}}\sqrt{2\tau}  \int_{\frac{\alpha+p}{\sqrt{2\tau}}}^{\frac{\beta+p}{\sqrt{2\tau}}}  e^{-t^2} dt\,. 
\end{split}
\end{align}
In the sequel, we assume without loss generality that $M\ge \max\{|\alpha|,|\beta|\}$, fix   a sufficiently small $\tau_0\in (0,1)$,  
and    establish lower bounds   of \eqref{eq:change_of_variable} in terms of  $p\in [-M,M]$ and $\tau\in (0,\tau_0]$. 
For notational simplicity, we denote by $C$ a generic constant independent of $p$ and $\tau$.

We start by assuming that  $p\in [-\beta,-\alpha]$.
In this case, $[\alpha,\beta]\ni a\mapsto pa+\frac{1}{2}a^2\in \sR$ achieves its minimum at $a=-p$,
and hence  
$\mathfrak{h}(p)=-\frac{1}{2}p^2 $. As $\tau \in  (0,\tau_0]$ and $p\in [-\beta,-\alpha]$, 
$ 0\ge  \frac{\alpha+p}{\sqrt{2\tau_0 }} \ge  \frac{\alpha+p}{\sqrt{2\tau}}$
and 
$\frac{\beta+p}{\sqrt{2\tau}}\ge \frac{\beta+p}{\sqrt{2\tau_0 }}\ge 0$. This implies that for all $p\in  [-\beta,-\alpha]$,
\begin{align}
\label{eq:integral_case1}
\int_{\frac{\alpha+p}{\sqrt{2\tau}}}^{\frac{\beta+p}{\sqrt{2\tau}}}  e^{-t^2} dt\ge
\int_{\frac{\alpha+p}{\sqrt{2\tau_0}}}^{\frac{\beta+p}{\sqrt{2\tau_0 }}}  e^{-t^2} dt\,.
\end{align}
By the Leibniz integral rule, 
$$
\frac{d}{dp}\int_{\frac{\alpha+p}{\sqrt{2\tau_0}}}^{\frac{\beta+p}{\sqrt{2\tau_0}}}  e^{-t^2} dt
=\frac{1}{\sqrt{2\tau_0}} e^{-\frac{(\beta+p)^2}{2\tau_0}} - \frac{1}{\sqrt{2\tau_0}} e^{-\frac{(\alpha+p)^2}{2\tau_0}} 
=\frac{1}{\sqrt{2\tau_0}} e^{-\frac{(\beta+p)^2}{2\tau_0}} (1-  e^{\frac{\beta^2-\alpha^2+2(\beta-\alpha)p}{2\tau_0}} )\,.
$$ 
Thus 
$p\mapsto \int_{\frac{\alpha+p}{\sqrt{2\tau_0}}}^{\frac{\beta+p}{\sqrt{2\tau_0}}}  e^{-t^2} dt$
is increasing on $[-\beta,-\frac{\alpha+\beta}{2}]$, decreasing on 
$[-\frac{\alpha+\beta}{2},-\alpha]$, 
and has the minimum $ \int_{0}^{\frac{\beta-\alpha}{\sqrt{2 \tau_0}}}  e^{-t^2} dt$ on   $  [-\beta,-\alpha]$. 
This along with \eqref{eq:change_of_variable} and \eqref{eq:integral_case1} shows that 
\begin{align*}
\int_A e^{-\frac{pa+\frac{1}{2}a^2}{\tau}}\mu(da)\ge 
\frac{1}{\beta-\alpha}e^{\frac{p^2}{2\tau}}\sqrt{2\tau} \int_{0}^{\frac{\beta-\alpha}{\sqrt{2\tau_0}}}  e^{-t^2} dt
\,.
\end{align*}
Taking the logarithm on both sides of the above inequality and multiplying by $-\tau$ yield  
$$
\mathfrak{h}_\tau(p) =-\tau \ln\left(\int_A e^{-\frac{pa+\frac{1}{2}a^2}{\tau}}\mu(da)\right) 
\le   \tau C - \frac{p^2}{2 }-\frac{\tau}{2}  \ln   \tau
=\mathfrak{h}(p) 
+ \tau C  -\frac{\tau}{2}  \ln   \tau\,.   
$$
Hence $ \mathfrak{h}_\tau(p)-\mathfrak{h}(p)\le C  \tau \ln\frac{1}{\tau}$ for all
$p\in [-\beta,-\alpha]$ and $\tau \in (0,\tau_0]$. 

We then consider the case where  $p\in [ -\alpha, M]$.
Then $[\alpha,\beta]\ni a\mapsto pa+\frac{1}{2}a^2\in \sR$ achieves its minimum at $a=\alpha$,
and hence  
$\mathfrak{h}(p)=p\alpha + \frac{1}{2}\alpha^2 $.
By \eqref{eq:change_of_variable},
\begin{align}
\label{eq:integral_-p_less_alpha}
\begin{split}
\int_A e^{-\frac{pa+\frac{1}{2}a^2}{\tau}}\mu(da)
&=
\frac{\sqrt{2\tau}}{\beta-\alpha}e^{\frac{p^2}{2\tau}} 
\left(\int_{\frac{\alpha+p}{\sqrt{2\tau}}}^\infty  e^{-t^2} dt -
\int_{\frac{\beta+p}{\sqrt{2\tau}}}^{\infty}  e^{-t^2} dt
\right) 
\\
&=\frac{\sqrt{2\tau}}{\beta-\alpha}e^{-\frac{\mathfrak{h}(p)}{ \tau}} e^{\frac{(\alpha+p)^2}{2\tau}} 
\left(\int_{\frac{\alpha+p}{\sqrt{2\tau}}}^\infty  e^{-t^2} dt -
\int_{\frac{\beta+p}{\sqrt{2\tau}}}^{\infty}  e^{-t^2} dt
\right) 
\,. 
\end{split}
\end{align}
By \cite[Equation 7.8.3]{olver2010nist}, 
$\frac{\sqrt{\pi}}{2\sqrt{\pi}x+2}\le e^{x^2}\int_x^\infty e^{-t^2}dt<\frac{1}{x+1} $ for all $x\ge 0$. 
As $p+\alpha\ge 0$, $p\le M$ and $\tau\in (0,\tau_0]$,
\begin{align}
\label{eq:estimate_-p_less_alpha}
\begin{split}
e^{\frac{(\alpha+p)^2}{2\tau}} \int_{\frac{\alpha+p}{\sqrt{2\tau}}}^\infty  e^{-t^2} dt 
&\ge 
\frac{\sqrt{\pi}}{2\sqrt{\pi} \frac{\alpha+p}{\sqrt{2\tau}}+2}
\ge 
\frac{\sqrt{\pi}}{2\sqrt{\pi} \frac{M+\alpha}{\sqrt{2\tau}}+2}
\ge 
\frac{\sqrt{2\pi\tau }}{2\sqrt{\pi}  {(M+  \alpha )} +2\sqrt{2\tau_0}}\,,
\\
e^{\frac{(\alpha+p)^2}{2\tau}} \int_{\frac{\beta+p}{\sqrt{2\tau}}}^\infty  e^{-t^2} dt  
&=
e^{\frac{(\alpha-\beta) (\alpha+\beta+2p)}{2\tau}} e^{\frac{(\beta+p)^2}{2\tau}} \int_{\frac{\beta+p}{\sqrt{2\tau}}}^\infty  e^{-t^2} dt 
\le e^{\frac{-(\beta-\alpha)^2}{2\tau}} 
\frac{1}{\frac{\beta+p}{\sqrt{2\tau}}+1}
\le e^{\frac{-(\beta-\alpha)^2}{2\tau_0}} 
\frac{\sqrt{2\tau}}{\beta-\alpha }\,.
\end{split}
\end{align}
Suppose that $\tau_0>0$ is sufficiently small such that 
$$
\frac{1}{2}\frac{\sqrt{2\pi }}{2\sqrt{\pi}  {(M+  \alpha )} +2\sqrt{2\tau_0}}
\ge e^{\frac{-(\beta-\alpha)^2}{2\tau_0}} 
\frac{\sqrt{2}}{\beta-\alpha }\,.
$$
Then 
by \eqref{eq:integral_-p_less_alpha}
and \eqref{eq:estimate_-p_less_alpha},
there exists $C\ge 0$ such that 
$
\int_A e^{-\frac{pa+\frac{1}{2}a^2}{\tau}}\mu(da)
\ge  C  \tau   e^{-\frac{\mathfrak{h}(p)}{ \tau}}     
$ for all $p\in [-\alpha, M]$ and $\tau\in (0,\tau_0]$, which subsequently implies that 
\begin{align}
\label{eq:h_tau-h_-p_less_alpha}
-\tau \ln\left(\int_A e^{-\frac{pa+\frac{1}{2}a^2}{\tau}}\mu(da)\right) 
\le \mathfrak{h}(p) + C \tau  \ln \frac{1}{  \tau}\,.
\end{align}

Finally,  consider the case where  $p\in [ -M,-\beta]$.
Then $[\alpha,\beta]\ni a\mapsto pa+\frac{1}{2}a^2\in \sR$ achieves its minimum at $a=\beta$,
and hence  
$\mathfrak{h}(p)=p\beta + \frac{1}{2}\beta^2 $.
By \eqref{eq:change_of_variable},
\begin{align}
\label{eq:integral_-p_larger_beta}
\begin{split}
\int_A e^{-\frac{pa+\frac{1}{2}a^2}{\tau}}\mu(da)
&=
\frac{\sqrt{2\tau}}{\beta-\alpha}e^{\frac{p^2}{2\tau}} 
\left(\int_{-\frac{\beta+p}{\sqrt{2\tau}}}^\infty  e^{-t^2} dt -
\int_{-\frac{\alpha+p}{\sqrt{2\tau}}}^{\infty}  e^{-t^2} dt
\right) 
\\
&=\frac{\sqrt{2\tau}}{\beta-\alpha}e^{-\frac{\mathfrak{h}(p)}{ \tau}} e^{\frac{(\beta+p)^2}{2\tau}} 
\left(\int_{-\frac{\beta+p}{\sqrt{2\tau}}}^\infty  e^{-t^2} dt -
\int_{-\frac{\alpha+p}{\sqrt{2\tau}}}^{\infty}  e^{-t^2} dt
\right) 
\,. 
\end{split}
\end{align}
Similar to \eqref{eq:estimate_-p_less_alpha}, for all  $p\in [ -M,-\beta]$  and $\tau\in (0,\tau_0]$,
\begin{align*}
\begin{split}
e^{\frac{(\beta+p)^2}{2\tau}} \int_{-\frac{\beta+p}{\sqrt{2\tau}}}^\infty  e^{-t^2} dt 
&\ge 
\frac{\sqrt{\pi}}{2\sqrt{\pi} \frac{-(\beta+p)}{\sqrt{2\tau}}+2}
\ge 
\frac{\sqrt{2\pi\tau }}{2\sqrt{\pi}  {(M -  \beta )} +2\sqrt{2\tau_0}}\,,
\\
e^{\frac{(\beta+p)^2}{2\tau}} \int_{-\frac{\alpha+p}{\sqrt{2\tau}}}^\infty  e^{-t^2} dt  
&=
e^{\frac{(\beta -\alpha) (\alpha+\beta+2p)}{2\tau}} e^{\frac{(\alpha+p)^2}{2\tau}} \int_{-\frac{\alpha+p}{\sqrt{2\tau}}}^\infty  e^{-t^2} dt 
\le e^{\frac{-(\beta-\alpha)^2}{2\tau}} 
\frac{1}{-\frac{\alpha+p}{\sqrt{2\tau}}+1}
\le e^{\frac{-(\beta-\alpha)^2}{2\tau_0}} 
\frac{\sqrt{2\tau}}{\beta-\alpha }\,.
\end{split}
\end{align*}
This along with
\eqref{eq:integral_-p_larger_beta} implies 
\eqref{eq:h_tau-h_-p_less_alpha} also holds for  $p\in [ -M,-\beta]$.
This completes the proof. 
\end{proof}

\begin{proof}[Proof of Proposition~\ref{prop:EG_LQR}]
Throughout this proof,  let  $\tau_0>0$ be fixed, and $C\ge 0$ be a generic constant independent of $\tau$. 
Observe that for all $(x,u,p)\in \cO\times \sR\times \sR^p$, 
\begin{align*}
& H_\tau(x,u,p) 
\\
& =  \bar{b}(x)^\top p - \bar{c}(x)u + \bar{b}(x)
-\tau \ln\Bigg(\int_A \exp\Bigg(-\frac{ 
(2\widehat{f}(x))^{-1} \big( \widehat{b}(x)^\top p  - \widehat{c}(x)u +\tilde{f}(x) \big)   a +   \frac{1}{2} a^2
}{\tau (2\widehat{f}(x))^{-1} }\Bigg) \mu(da)\Bigg)\,
\end{align*}
By Proposition \ref{prop:wp_unregularized} and
the Sobolev embedding,  
$v^*_0 \in    C^{1}(\overline{\cO})$, and hence there exists $M\ge 0$ such that 
$|v^*_0(x)|+ |D v^*_0(x)|\le M$ for all $x\in \overline{\cO}$.
This along   with
$\inf_{x\in \cO}\widehat{f}(x)>0$ 
and the boundedness of $ \widehat{b}$, $ \widehat{c}$ and $ \widehat{f}$
implies that there exists $C\ge 0$ such that  
$$ 
\left|  (2\widehat{f}(x))^{-1}  \left(\widehat{b}(x)^\top Dv^*_0(x)  - \widehat{c}(x) v^*_0(x)  +\tilde{f}(x)\right)\right|\le C,
\quad \forall x\in \cO\,.
$$ 
Hence by Proposition \ref{lemma:H_tau_interval}, 
for all $\tau\in (0,\tau_0]$ and $x\in \cO$, 
\begin{align*}
&H_\tau(x,v^*_0(x),Dv^*_0(x)) 
\\
&
\le  \bar{b}(x)^\top p - \bar{c}(x)u + \bar{b}(x) 
+   2\widehat{f}(x)   \min_{a\in A} 
\left( (2\widehat{f}(x))^{-1} \big( \widehat{b}(x)^\top p  - \widehat{c}(x)u +\tilde{f}(x) \big)   a +   \frac{1}{2} a^2\right)
+C\tau \ln \frac{1}{\tau}
\\
& = 
H(x,v^*_0(x),Dv^*_0(x)) +C\tau \ln \frac{1}{\tau}\,.
\end{align*}
This completes the proof. 
\end{proof}

\appendix

\section{Proofs of Propositions~\ref{prop:Bellman_PDE_wellposedness},  \ref{prop:verification} and \ref{prop:wp_unregularized} }
\label{sec:well-posedness of HJB}

We first 
recall the following   $W^{2,p}$-estimate for  linear elliptic PDEs  proved in~\cite[Theorems 6.3 and 6.4]{chen1998second},
which will be used frequently in the subsequent analysis.

\begin{lemma}
\label{lemma:elliptic_regularity}
Let      $\cO$ be a bounded domain in  $\sR^d$
whose boundary $\partial \cO$ is of the class $C^{1,1}$. 
Let $a^{ij}, b^i:\sR^d\to \sR$, $i,j=1,\ldots, d$, 
and $c:\sR^d\to \sR$
be measurable functions
such that  
$a^{ij}\in C(\overline{\cO})$ for all $i,j=1,\ldots, d$,
and 
there exists  $\lambda, \Lambda>0$ such that 
$   \sum_{i,j=1}^d  a^{ij}(x)\xi_i\xi_j \ge \lambda |\xi|^2$ for all $x\in \cO$ and $\xi=(\xi_i)_{i=1}^d\in \sR^d$,  
$ \sum_{i,j=1}^d  \|a^{ij}\|_{B_b(\cO)} + \sum_{i=1}^d  \|b^i\|_{B_b(\cO)}+\|c\|_{B_b(\cO)}\le  \Lambda$
and  $c\ge 0$. 
Then for each $p\in (1,\infty)$ and $f\in L^p(\cO)$, 
there exists a unique solution   $u\in W^{2,p}(\cO)\cap W^{1,p}_0(\cO)$ 
to the following  boundary value problem 
$$
\sum_{i,j=1}^d  a^{ij}D_{ij}v+ \sum_{i=1}^d b^iD_i v-cv+f=0 \quad \textnormal{in $\cO$;}
\quad v=0 \quad \textnormal{on $\partial \cO$,}
$$
and  $\|u\|_{W^{2,p}(\cO)}\le C \|f\|_{L^p(\cO)}$, 
with a constant  $C$ depending  only on $d, p, \lambda, \Lambda$, $\cO$ and the modulus of continuity of $(a^{ij})_{i,j=1}^d$. 
\end{lemma}

\begin{proof}[Proof of Proposition~\ref{prop:Bellman_PDE_wellposedness}]
As $\pi\in \Pi_{\mu}$, \eqref{eq:on_policy_bellman} can be equivalently written as 
\begin{align}
\frac{1}{2}\tr(\sigma(x)\sigma(x)^\top D^2v(x) )+\tilde{b}(x )^\top Dv(x) -\tilde{c}(x)v(x)+\tilde{f}(x)+ \tau \tilde{h}(x)=0, \quad  x \in \cO\,,     
\end{align}
where 
$\tilde{b}:\cO\to \sR^d$, $\tilde{c}: \cO\to [0,\infty)$,  $\tilde{f}:\cO\to \sR$
and $\tilde{h}:\cO\to \sR$ are measurable functions given by
\begin{align}
\begin{split}
\tilde{b}(x)&= \int_A b(x,a) \pi(da|x), \quad 
\tilde{c}(x)= \int_A c(x,a) \pi(da|x), \quad
\tilde{f}(x)= \int_A f(x,a) \pi(da|x)\,. 
\\
\tilde{h}(x)&=\textrm{KL}(\pi |\mu)(x)
=\int_A\left(Z(x,a)-\ln\left(\int_A Z(x,a')\mu(da')\right)\right)\pi(da|x)\,,
\end{split}
\end{align} 
for some $Z\in B_b(\cO\times A)$.
As $\pi(A|x)=1$ for all $x\in \cO$,  
$\|\tilde{b }\|_{B_b(\cO)}\le \| {b }\|_{B_b(\cO)}$,
$\|\tilde{c}\|_{B_b(\cO)}\le \| c\|_{B_b(\cO)}$,
$\|\tilde{f }\|_{B_b(\cO)}\le \| f\|_{B_b(\cO)}$  
and 
$ \|\tilde{h}\|_{B_b(\cO)}
\le 2\|Z\|_{B_b(\cO\times A)}$. 
Hence by Lemma \ref{lemma:elliptic_regularity} and $g\in W^{2,p^*}(\cO)$, 
\eqref{eq:on_policy_bellman} admits a unique   solution $v:\cO\to \sR$
such that 
$v\in W^{2,p^*}(\cO)$ and $v-g\in W^{1,p^*}_0(\cO)$.
By using the Sobolev embedding $W^{2,p^*}(\cO)\subset  C^{1}(\overline{\cO})$ and 
\eqref{eq:on_policy_bellman}, one can deduce that 
$\tr(\sigma \sigma^\top D^2 v_\tau) \in L^\infty(\cO)$.
Finally,
applying the It\^o formula~\cite[Theorem 1, p.~122]{krylov2008controlled} for functions in $W^{2,2}(\cO)$ yields 
\[
v(x) = -\mathbb E^{\mathbb P^{x,\pi}}\left[\int_0^{\tau_\cO} \Gamma^\pi_t\int_A (\mathcal L^av)(X_t)\,\pi(da|X_t)\,dt + \Gamma^\pi_{\tau_\cO} v(X_{\tau_\cO})\right]\,, 
\]
which along with~\eqref{eq:value} and~\eqref{eq:on_policy_bellman} implies
\[
v(x) = \mathbb E^{\mathbb P^{x,\pi}}\left[\int_0^{\tau_\cO} \bigg( \Gamma^\pi_t\int_A f (X_t,a)\,\pi(da|X_t)+\tau \textrm{KL}(\pi|\mu)(X_t)\bigg)\,dt + \Gamma^\pi_{\tau_\cO} g(X_{\tau_\cO})\right] = v^\pi(x)\,.
\]
This concludes the proof.
\end{proof}

The following lemma states   some elementary properties of $H_\tau$. The proof follows   from
a straightforward computation 
and Assumption \ref{ass:data}  and   is therefore omitted. 

\begin{lemma}
\label{lemma:H_tau_regularity}
Suppose Assumption \ref{ass:data} holds, 
and  $\tau > 0$. For all $x\in \cO$,
$(u,p)\mapsto H_\tau(x,u,p)$ is   continuously  differentiable   
and 
\begin{align}
\partial_u H_\tau (x,u,p) = - \int_A  c(x,a) \boldsymbol{\pi}(h_{u,p} )(da|x),
\quad 
\partial_p H_\tau (x,u,p) =   \int_A  b(x,a) \boldsymbol{\pi}(h_{u,p})(da|x),
\end{align}
where $\boldsymbol{\pi} $
is defined by
\eqref{eq:operator_pi}, and   $  h_{u,p} (x,a) \coloneqq -\frac{b(x,a)^\top p -c(x,a)u+f(x,a)}{\tau } 
$. 
Moreover, 
there exists   $C\ge 0$, independent of $\tau$, such that for all $x\in \cO$, $u\in \sR$ and $p\in \sR^d$,
$|H_\tau (x,0,0)|\le C$, 
$-C\le \partial_u H_\tau(x,u,p)\le 0$ and $ |\partial_p H_\tau(x,u,p) |\le C$.
\end{lemma}

The next lemma establishes an a priori estimate for~\eqref{eq:semilinear}.  

\begin{lemma}
\label{lemma:hjb_a_priori}
Suppose Assumption \ref{ass:data} holds, 
and $\tau > 0$.
There exists  $C\ge 0$, independent of $\tau$,
such that if   
$v\in W^{2,p^*}(\cO) $ and   $\eta \in [0,1]$   satisfy 
\begin{equation}
\label{eq:hjb_eta}
\frac{1}{2}\tr(\sigma(x)\sigma(x)^\top D^2v(x) )+\eta  H_\tau (x,v(x),Dv(x)) =0,\quad \textnormal{a.e.~$x\in \cO$}\, ;  \quad 
v(x)=\eta  g(x),  \quad x \in \partial \cO\,, 
\end{equation}
then 
$\|v\|_{W^{2,{p^*}}(\cO)}\le C (1+\|g\|_{W^{2,{p^*}}(\cO)})$. 
\end{lemma}

\begin{proof}
Throughout this proof, let $C\ge 0$ be a generic constant 
which is 
independent of $\tau$, $\eta$ and $g$, and  may take a different value at each occurrence. 
As $g \in W^{2,{p^*}}(\cO)$,    $w\coloneqq v-\eta g\in W^{2,{p^*}}(\cO)\cap W^{1,{p^*}}_0(\cO)$   satisfies 
\begin{equation}
\label{eq:w_pde}
\frac{1}{2}\tr(\sigma(x)\sigma(x)^\top D^2(w+\eta g)(x) )+ \eta H_\tau (x,(w+\eta  g)(x),D(w+\eta  g)(x)) =0 \,,\quad  \textnormal{a.e.~$x\in \cO$}\,.
\end{equation}
By  \cite[Lemma 9.17]{gilbarg1977elliptic}, there exists $C\ge 0$ such that 
$\|w\|_{W^{2,{p^*}}(\cO)}\le C\| \frac{1}{2}\tr(\sigma \sigma^\top D^2w)\|_{L^{{p^*}}(\cO)}$, which along with \eqref{eq:w_pde}, 
Lemmas \ref{lemma:elliptic_regularity} and \ref{lemma:H_tau_regularity} and Assumption \ref{ass:data}  implies  
\begin{align}
\label{eq:w_estimate_1_p}
\begin{split}
\|w\|_{W^{2,{p^*}}(\cO)}
&\le C\left\|-\eta   H_\tau (\cdot,(w+\eta  g)(\cdot),D(w+\eta  g)(\cdot)) -  \eta   \frac{1}{2}\tr(\sigma \sigma^\top D^2g )\right\|_{L^{{p^*}}(\cO)}
\\
&\le C\left(   \| H_\tau (\cdot,0,0)\|_{L^{{p^*}}(\cO)}+\|w+\eta  g\|_{W^{1,{p^*}}(\cO) } +    \| g\|_{W^{2,{p^*}}(\cO) }\right)
\\
&\le C\left(1+ \| w\|_{W^{1,{p^*}}(\cO)}+\| g\|_{W^{2,{p^*}}(\cO)}  \right)\,.
\end{split}
\end{align}
By the interpolation inequality  \cite[Theorem 7.28]{gilbarg1977elliptic},
for all $\varepsilon>0$, there exists  $C_\varepsilon\ge 0$  such that 
$\| w\|_{W^{1,{p^*}}(\cO)} \le \varepsilon \| w\|_{W^{2,{p^*}}(\cO)} +C_\varepsilon  \| w\|_{L^{p^*}(\cO)}$ for all $w\in  W^{2,{p^*}}(\cO)$. 
Choosing a sufficiently small $\varepsilon$ and using \eqref{eq:w_estimate_1_p} yield
\begin{align}
\label{eq:w_estimate_p}
\|w\|_{W^{2,{p^*}}(\cO)}
\le C\left(1+ \| w\|_{L^{{p^*}}(\cO)}+\| g\|_{W^{2,{p^*}}(\cO)}  \right)\,.
\end{align}
To estimate $\| w\|_{L^{{p^*}}(\cO)}$, 
observe from \eqref{eq:w_pde} that 
\begin{equation}
\label{eq:w_pde_linear}
\frac{1}{2}\tr(\sigma(x)\sigma(x)^\top D^2w(x) )+
\sum_{i=1}^d \tilde{b}^i(x) D_i w(x)+ 
\tilde{c} (x) w(x)   =\tilde{h}(x)\,, \quad  \textnormal{a.e.~$x\in \cO$}\,,
\end{equation}
where for all $i=1,\ldots, d$,  
\begin{align*}
\tilde{b}^i(x)& =\int_0^1 \partial_{p_i} H_\tau (x,(w+\eta  g)(x),\eta  Dg(x)+tDw(x)) dt \,,
\\
\tilde{c}(x)& = \int_0^1 \partial_{u} H_\tau (x,\eta  g(x)+tw(x) ,\eta  Dg(x) ) dt\,,
\\
\tilde{h}(x)& =   - \eta  \frac{1}{2}\tr(\sigma(x) \sigma(x) ^\top D^2g(x)  )- H_\tau (x,\eta  g(x) ,\eta  Dg(x) )\,.
\end{align*}
By Lemma \ref{lemma:H_tau_regularity}, $|\tilde{b}(x)|\le C$ and   $  \tilde{c}(x)\le 0$ for all $x\in \cO$.
Hence  
as $p^*>d$ and $w\in W^{2,{p^*}}(\cO)\cap W^{1,{p^*}}_0(\cO)$, the   maximum principle  \cite[Theorem 9.1]{gilbarg1977elliptic} 
shows that 
$$
\|w\|_{L^\infty (\cO)}\le C\|\tilde{h}\|_{L^d(\cO)}\le  C\left(1+ \| g\|_{W^{2,{p^*}}(\cO)}  \right)\,,
$$
which along with \eqref{eq:w_estimate_p} yields the desired estimate.
\end{proof}

Now we are ready to prove Proposition  \ref{prop:verification}.

\begin{proof}[Proof of Proposition \ref{prop:verification}]

We start by proving the existence of a solution to \eqref{eq:semilinear} in $W^{2,p^*}(\cO)$,
with 
$p^*>0$    in Assumption \ref{ass:data}.
Define the map $T: W^{1,p^*}(\cO)\to W^{1,p^*}(\cO)$ such that for all $v\in W^{1,p^*}(\cO)$, $u=Tv$ is the unique solution to 
\begin{equation}
\frac{1}{2}\tr(\sigma(x)\sigma(x)^\top D^2u(x) )+ H_\tau (x,v(x),Dv(x)) =0,\quad \textnormal{a.e.~$x\in \cO$}\, ;  \quad 
u(x)=g(x),  \quad x \in \partial \cO\,, 
\end{equation}
By Lemmas \ref{lemma:elliptic_regularity} and \ref{lemma:H_tau_regularity}, 
$x\mapsto H_\tau (x,v(x),Dv(x))$ is in $L^{p^*}(\cO)$ and hence 
$u=Tv\in W^{2,p^*}(\cO)$ is well-defined. 

We claim that   $T:W^{1,p^*}(\cO)\to W^{1,p^*}(\cO)$ is continuous and compact. 
For the continuity of $T$,
observe that by Lemma \ref{lemma:H_tau_regularity},
$H_\tau$ is a Carath\'{e}odory function, i.e.,  $H_\tau $ is measurable in $x$ and continuous in $(u,p)$,
and for all $g\in L^{p^*}(\cO) $ and $h\in L^{p^*}(\cO)^d$, the function 
$x\mapsto H_\tau(x,g(x),h(x))$ is in $L^{p^*}(\cO)$.
Hence 
by \cite[Theorem 4]{goldberg1992nemytskij},
the Nemytskij operator $L^{p^*}(\cO) \times  L^{p^*}(\cO)^d \ni (g, h) \mapsto H_\tau(x,g(\cdot),h(\cdot))\in L^{p^*}(\cO)$ is continuous.
This   along with the continuity of 
$W^{1, p^*}(\cO)\ni g\mapsto (g,Dg) \in  L^{p^*}(\cO) \times  L^{p^*}(\cO)^d $ and Lemma \ref{lemma:elliptic_regularity}
implies  that $T: W^{1,p^*}(\cO)\to W^{1,p^*}(\cO)$ is continuous.
For the compactness of $T$, 
by Lemmas \ref{lemma:elliptic_regularity} and \ref{lemma:H_tau_regularity},
$T$ maps bounded sets in $W^{1,p^*}(\cO)$ to bounded sets in $W^{2,p^*}(\cO)$, 
which are precompact in $W^{1,p^*}(\cO)$, 
due to   $p^*>d$ and the Kondrachov   embedding theorem \cite[Theorem 7.26]{gilbarg1977elliptic}.

Finally, 
for all 
$\eta\in [0,1]$ and $v\in W^{1,p^*}(\cO)$ such that  $v=\eta T v$, it holds  that   $v$ is in $ W^{2,p^*}(\cO)$ and satisfies \eqref{eq:hjb_eta}. 
This along with
Lemma \ref{lemma:hjb_a_priori} implies that  there exists $C\ge 0$, independent of $\eta$ and $v$ such that $\|v\|_{ W^{2,p^*}(\cO)}\le C$. 
Hence by the Leray-Schauder Theorem \cite[Theorem 11.3]{gilbarg1977elliptic}, 
there exists $u\in W^{1,p^*}(\cO)$ such that $u=Tu$. 
This implies that 
\eqref{eq:semilinear} admits a solution $u\in W^{2,p^*}(\cO)$. 

To prove the uniqueness of solutions to \eqref{eq:semilinear}, 
let $u, v\in W^{2,p^*}(\cO)$ satisfy  \eqref{eq:semilinear}.
Then $w=u-v\in W^{2,p^*}(\cO)\cap W^{1,p^*}_0(\cO)$ satisfies 
\begin{equation*}
\frac{1}{2}\tr(\sigma(x)\sigma(x)^\top D^2w(x) )+
\sum_{i=1}^d \tilde{b}^i(x) D_i w(x)+ 
\tilde{c} (x) w(x)   =0\,, \quad  \textnormal{a.e.~$x\in \cO$}\,,
\end{equation*}
where   
\begin{align*}
\tilde{b}^i(x)& =\int_0^1 \partial_{p_i} H_\tau (x, u(x),   Dv(x)+tD(u-v)(x)) dt\,, \quad i=1,\ldots, d  \,,
\\
\tilde{c}(x)& = \int_0^1 \partial_{u} H_\tau (x,   v(x)+t(u-v) (x) ,   Dv (x) ) dt\,.
\end{align*}
As $\tilde{c}\le 0$,  the   maximum principle  \cite[Theorem 9.1]{gilbarg1977elliptic}  implies that $u\le v$.
Interchanging the roles of $u$ and $v$ shows that $u=v$, which implies the uniqueness of the solution to \eqref{eq:semilinear}. 

It remains to prove that the optimal  value function $v^*_\tau $ in \eqref{eq:value_optimal}  is the solution to 
\eqref{eq:semilinear} and  $ \pi^*_\tau $ is an optimal control.
This follows from the standard verification arguments    (see e.g., \cite[Theorem 2.2]{reisinger2021regularity})
using  the generalised   It\^{o}'s formula   \cite[Theorem 1, p.~122]{krylov2008controlled}.  The detailed steps are  omitted. 
\end{proof}

Finally, we prove Proposition~\ref{prop:wp_unregularized}.

\begin{proof}[Proof of Proposition~\ref{prop:wp_unregularized}]  
It is easy to see that for all $x\in \cO$ and $(u,p),(u',p')\in \sR\times \sR^d$,
$$
|H(x,u,p)-H(x,u',p')| \le \|b\|_{B_b(\sR^d\times A)}|p-p'|+  \|c\|_{B_b(\sR^d\times A)}|u-u'|. 
$$
Since the statement assumes   that
$H$  is measurable with respect to  $x$,  $H$ is a Carath\'{e}odory function. 
Hence the existence and uniqueness of a solution  $\bar{v} \in  W^{2,p^*}(\cO)$ to \eqref{eq:hjb_unregualarised} can be proved by the Leray-Schauder Theorem \cite[Theorem 11.3]{gilbarg1977elliptic} as in the proof of Proposition \ref{prop:verification}.

To estimate 
$ \|v^*_\tau - \bar{v} \|_{W^{2,p^*}(\cO)}$,
observe  that  $  v^*_\tau - \bar{v} \in W^{2,{p^*}}(\cO)\cap W^{1,{p^*}}_0(\cO)$   satisfies
for a.e.~$x\in \cO$,
\begin{align}
\begin{split}
&\frac{1}{2}\tr(\sigma(x)\sigma(x)^\top D^2 ( v^*_\tau - \bar{v} ) (x) )
+ 
H_\tau (x,  v^*_\tau (x),D v^*_\tau(x)) 
-H (x, \bar{v} (x),D\bar{v} (x)) 
\\
&=\frac{1}{2}\tr(\sigma(x)\sigma(x)^\top D^2 ( v^*_\tau - \bar{v} ) (x) )
+ 
H_\tau (x,  v^*_\tau (x),D v^*_\tau(x)) 
-H_\tau (x, \bar{v} (x),D\bar{v} (x)) 
\\
&\quad 
+ H_\tau (x, \bar{v} (x),D\bar{v} (x)) -H (x, \bar{v} (x),D\bar{v} (x)) 
=0 \,.
\end{split}
\end{align}
This implies that $w\coloneqq v^*_\tau - \bar{v}$  satisfies 
\begin{equation}
\frac{1}{2}\tr(\sigma(x)\sigma(x)^\top D^2w(x) )+
\sum_{i=1}^d \tilde{b}^i(x) D_i w(x)+ 
\tilde{c} (x) w(x)   +\tilde{h}(x)=0\,, \quad  \textnormal{a.e.~$x\in \cO$}\,,
\end{equation}
where for all $i=1,\ldots, d$,  
\begin{align*}
\tilde{b}^i(x)& =\int_0^1 \partial_{p_i} H_\tau (x, v^*_\tau (x),   D\bar{v}  (x)+t Dw (x)) dt \,,
\\
\tilde{c}(x)& = \int_0^1 \partial_{u} H_\tau (x,  \bar{v} (x)+tw(x) ,   D \bar{v} (x) ) dt\,,
\\
\tilde{h}(x)& =  H_\tau (x, \bar{v} (x),D\bar{v} (x)) -H (x, \bar{v} (x),D\bar{v} (x))\, \,.
\end{align*}
By Lemma \ref{lemma:H_tau_regularity}, 
there exists a constant $C\ge 0$, independent of $\tau$, such that 
$|\tilde{b}(x)|\le C$ and   $  \tilde{c}(x)\le 0$ for all $x\in \cO$.
Hence  the desired $W^{2,p^*}$-estimate follows from  Lemma \ref{lemma:elliptic_regularity}
and the inequality that for all $(x,u,p)\in \cO\times \sR\times \sR^d$, 
\begin{align*}
H_\tau (x,u,p)&\ge  -\tau \ln\left(\int_A \exp\left(-\frac{\min_{a\in A}(b(x,a)^\top p -c(x,a)u+f(x,a))}{\tau }\right) \mu(da)\right)
= H(x,u,p)\,.
\end{align*}
This finishes the proof. 
\end{proof}

\section{Proof of  Theorem~\ref{thm:Existence_Solutions}}
\label{sec:existence_of_sols}

The  main technical issue is that the nonlinearity  $B_b(\cO\times A) \ni Z \mapsto   \overline{\cL}^\cdot v_\tau^{\boldsymbol{\pi}(Z)} $ in \eqref{eq:md_intro}  is    merely    locally Lipschitz continuous (Proposition~\ref{prop:W^{2,p} bound}).
To address this issue, we proceed with   the following   three steps:   
We first   show that a truncated version of the mirror descent flow~\eqref{eq:md_intro} has a unique solution (Lemma~\ref{lem:local_existence}).
Then we  will use the linear growth  
$  Z \mapsto   \overline{\cL}^\cdot v_\tau^{\boldsymbol{\pi}(Z)} $ (Lemma \ref{lem:DV_s_bound})
to obtain an a priori estimate for solutions to~\eqref{eq:md_intro} (Lemma~\ref{lem:|Z_s|_bound}).
Finally, we'll combine these two intermediate steps to prove Theorem~\ref{thm:Existence_Solutions}.

Fix  $\btau\in C([0,\infty); (0,\infty))$.
Define the operation $\mathcal{H}:(0,\infty)\times B_b(\cO\times A)\rightarrow    B_b(\cO\times A)$  by 
\begin{align*}
\mathcal{H}(s,Z) = -(b^\top Dv^{\boldsymbol{\pi}(Z)}_{\btau_s} - cv^{\boldsymbol{\pi}(Z)}_{\btau_s} + f +{\btau_s} Z)\,.
\end{align*}
For all $N>0$, define the operator 
$\mathcal{H}_N:(0,\infty)\times B_b(\cO\times A)\rightarrow   B_b(\cO\times A)$ by
\begin{align*}
\mathcal{H}_N(s,Z)=\begin{cases}
\mathcal{H}(s,Z), & \|Z\|_{B_b(\cO\times A)}\leq N \\
\mathcal{H}\left(s,\frac{NZ}{\|Z\|_{B_b(\cO\times A)}}\right), & \|Z\|_{B_b(\cO\times A)}>N
\end{cases},
\end{align*}
Recall the constant $K$ defined at the beginning of Section \ref{sec:Differentiability_Stability}.
\begin{lemma}
\label{lem:local_existence}
Suppose Assumption \ref{ass:data} holds. 
For each $N\in \sN$,  $\btau\in C([0,\infty); (0,\infty))$  and $Z^0\in B_b(\cO\times A)$,
there exists a unique $Z\in \cap_{S>0}C^1([0,S];B_b(\cO\times A))$ satisfying 
$\partial_s Z_s=\mathcal{H}_N(s,Z_s)$ for all $s>0$ and $Z_0=Z^0$.
\end{lemma}

\begin{proof}
Throughout this proof let $\|\cdot\|=\|\cdot\|_{B_b(\cO\times A)}$.
We first show that there exists a constant $C>0$, depending on $N$, such that for all $Z_1,Z_2\in B_b(\cO\times A)$
and 
$s\in(0,\infty)$,
\begin{align}
\label{eqn:local_lip}
\|\mathcal{H}_N(s,Z_1)-\mathcal{H}_N(s,Z_2)\|\leq C(1+\mathcal{T}_s) \|Z_1 - Z_2\|\,,
\quad \textnormal{with $\mathcal{T}_s = \sup_{r\in[0,s]}\btau_r$}\,.
\end{align}
Without loss of generality, assume $Z_1,Z_2\in B_b(\cO\times A)$ satisfy $\|Z_1\|,\|Z_2\|\leq N$.
Then 
\begin{align*}
\|\mathcal{H}(s,Z_1)-\mathcal{H}(s,Z_2)\|&\leq K\|Dv_{\btau_{s}}^{\boldsymbol{\pi}(Z_1)}-Dv_{\btau_{s}}^{\boldsymbol{\pi}(Z_2)}\|_{C^0(\overline{\cO})}+K\|v_{\btau_{s}}^{\boldsymbol{\pi}(Z_1)}-v_{\btau_{s}}^{\boldsymbol{\pi}(Z_2)} \|_{C^0(\overline{\cO})}+  \btau_{s}\|Z_1-Z_2\|\\
&\leq 2C \|v_{\btau_s}^{\boldsymbol{\pi}(Z_1)}-v_{\btau_s}^{\boldsymbol{\pi}(Z_2)}\|_{W^{2,p^*}(\cO)}+\btau_s\|Z_1-Z_2\|,
\end{align*}
where the final inequality follows from Sobolev embedding
with a generic constant 
$C>0$ depending only on $d$, $p^*$, $\lambda$, $K$, $\cO$  and the modulus of continuity of $\sigma\sigma^\top$.
This along with  Proposition \ref{prop:W^{2,p} bound}  implies that
\begin{align*}
\|v_{\btau_s}^{\boldsymbol{\pi}(Z_1)}-v_{\btau_s}^{\boldsymbol{\pi}(Z_2)}\|_{W^{2,p^*}(\cO)}
&\leq C(1+\btau_s)(
1+\|v_{\btau_s}^{\boldsymbol{\pi}(Z_1)}\|_{C^1(\overline{\cO})}+\|Z_1\|
)\|Z_1-Z_2\|
\\
&
\leq C(1+\btau_s)(1+N )\|Z_1-Z_2\|\,,
\end{align*}
where the second inequality follows from Lemma \ref{lem:DV_s_bound}.
This proves \eqref{eqn:local_lip}. 

We are now ready to prove the desired well-posedness result.  
Fix an arbitrary $S>0$ and let 
$\eta>0$ be a constant to be determined later.
Let 
$X_{S,\eta}\coloneqq C\left([0,S];B_b(\cO\times A)\right)$ be equipped with the norm 
$\|k\|_{S,\eta}\coloneqq \sup_{s\in[0,S]}e^{-\eta s}\|k_s\|$.
Note that the norms $\|\cdot\|_{S,\eta}$ and
$\|\cdot\|_{S,0}$ are equivalent and 
since $X_{S,0}$ is a Banach space (see~\cite[Theorem 3.2-2]{ciarlet2013linear}), $X_{S,\eta}$ is also a Banach space.
Define $\psi:X_{S,\eta}\rightarrow X_{S,\eta}$ 
by $\psi(Z)_s = Z^0 + \int_0^s\mathcal{H}_{N}(r, Z_{r}) d r$.
We will show $\psi$ is a contraction on $X_{S,\eta}$ for an appropriate choice of $\eta$. 
To that end note that
\begin{align*}
\|\psi(Z)_s-\psi(\tilde{Z})_s\| & \leq\int_0^s\|\mathcal{H}_{N}(\bar{s},Z_{\bar{s}})-\mathcal{H}_{N}(\bar{s},\tilde{Z}_{\bar{s}})\|\,d\bar{s}\leq C(1+\mathcal{T}_S) \int_0^s\|Z_{\bar{s}}-\tilde{Z}_{\bar{s}}\|e^{-\eta \bar{s}}e^{\eta \bar{s}}\,d\bar{s}\\
&\leq C(1+\mathcal{T}_S)\sup_{r\in[0,S]}e^{-\eta r}\|Z_r-\tilde{Z}_r\|\int_0^se^{\eta \bar{s}}d\bar{s}
\leq\eta^{-1}C(1+\mathcal{T}_S) \|Z-\tilde{Z}\|_{S,\eta}e^{\eta s},
\end{align*}
with  the constant $C$  given   in \eqref{eqn:local_lip}.
Setting $C_0=C(1+\mathcal{T}_S) $ and  $\eta = C_0+1$, multiplying both sides of the above inequality by $e^{-\eta   s}$ and then taking a supremum over $[0,S]$ gives $\|\psi(Z)-\psi(\tilde{Z})\|_{S,\eta_0}\leq\frac{C_0}{C_0+1}\|Z-\tilde{Z}\|_{S,\eta }$.
Therefore from Banach's fixed point theorem there exists a unique $Z\in (C([0,S];B_b(\cO\times A)),\|\cdot\|_{S,\eta_0})$ such that $Z_s = Z_0 + \int_0^s\mathcal{H}_N(\bar{s},Z_{\bar{s}})\,d\bar{s}$, and from the equivalence of the norms $\|\cdot\|_{S,\eta_0}$ and $\|\cdot\|_{S,0}$ we have that $Z\in (C([0,S];B_b(\cO\times A)),\|\cdot\|_{S,0})$. 
By the fundamental theorem of calculus,  $Z$ is differentiable,
i.e. $Z\in C^{1}([0,S];B_b(\cO\times A))$. 
Since $S>0$ was arbitrary, we get that $Z\in \bigcap_{S>0}C^1([0,S];B_b(\cO\times A))$
and satisfies $\partial_s Z_s = \mathcal{H}_N(s,Z_s)$ for all $s>0$ and   $Z_0=Z^0$.
\end{proof}

\begin{lemma}
\label{lem:|Z_s|_bound}
Suppose Assumption~\ref{ass:data} holds.
Let $\btau\in C([0,\infty); (0,\infty))$ and $Z_0\in B_b(\cO\times A)$.
Then  there exists a constant $C>0$ depending on $d$, $p^*$, $\lambda$, $K$, $\cO$ and the modulus of continuity of   $\sigma\sigma^\top$ 
such that 
for all    $Z\in\cap_{S>0}C^1([0,S];B_b(\cO\times A))$ satisfying    \eqref{eq:md_intro}
and all  $S\geq 0$, 
\begin{align}
\label{eqn:Z_s_bound_along_flow}
\sup_{s\in[0,S]}\|Z_s\|_{B_b(\cO\times A)} \leq C(1+\mathcal{T}_S+ \|Z_0\|_{B_b(\cO\times A)})e^{C(1+\mathcal{T}_S)S}\,,\quad 
\textnormal{with $\mathcal{T}_s = \sup_{r\in[0,s]}\btau_r$}\,.
\end{align}
\end{lemma}
\begin{proof}
Throughout this proof, 
let $\pi_s=\boldsymbol{\pi}(Z_s)$ for all $s>0$,
and  let $C>0$ be a generic constant depending only on $d$, $p^*$, $\lambda$, $K$, $\cO$  and the modulus of continuity of $\sigma\sigma^\top$.
Integrating  \eqref{eq:md_intro}  from $0$ to $s$ yields for all $(x,a)\in \cO\times A$, 
$$
{Z_s}(x,a) =Z_0(s,a)-\int_0^s \left( \overline{ \mathcal L}^a v^{\pi_r}_{\btau_r} (x)+ f(x,a)  + \btau_r Z_r(x,a)\right)dr \,,
$$
Taking the $\|\cdot\|_{B_b(\cO\times A)}$ norm on both sides gives
\begin{align*}
\|Z_s\|_{B_b(\cO\times A)}\leq\|Z_0\|_{B_b(\cO\times A)}&+ \int_0^s C(   \| v^{\pi_r}_{\btau_r}\|_{C^1(\bar{\cO})}+1+ \cT_r \|Z_{r}\|_{B_b(\cO\times A)})dr\,,
\end{align*}
which along with Lemma \ref{lem:DV_s_bound} shows that 
\begin{align*}
\|Z_s\|_{B_b(\cO\times A)}\leq\|Z_0\|_{B_b(\cO\times A)}&+ \int_0^s C(  1+\cT_r  +(1+ \cT_r) \|Z_{r}\|_{B_b(\cO\times A)})dr\,.
\end{align*}
The desired estimate follows from    Gronwall's inequality.
\end{proof}
\begin{proof}
[Proof of Theorem
\ref{thm:Existence_Solutions}]
Fix an arbitrary  $S>0$, and let  $M>0$ be the right-hand side of \eqref{eqn:Z_s_bound_along_flow}.
By Lemma \ref{lem:local_existence}, there exists a unique  $\tilde Z \in \bigcap_{S>0} C^1([0,S];B_b(\mathcal O\times A))$ such that  $\partial_s \tilde Z_s = \mathcal H_{2M}(s, \tilde Z_s)$ for all $s>0$ and   $\tilde Z_0 = Z_0$. 
Let $S_M \coloneqq  \inf\{s\geq 0 : \|\tilde Z_s\|_{B_b(\mathcal O \times A)} \geq 2M\}$.
Assume for the moment that $S_M < S$.
On $[0,S_M]$,  we have $\partial_s \tilde Z_s = \mathcal H(s, \tilde Z_s)$ and hence by  Lemma~\ref{lem:|Z_s|_bound},  
$    \|\tilde Z_{S_M}\|_{B_b(\mathcal O \times A)}     \leq M
$
due to the assumption that $S_M \leq S$.
But that implies $2M \leq \|\tilde Z_{S_M}\|_{B_b(\mathcal O \times A)}  \leq M$ which is a contradiction.
Thus $S\leq S_M$ which means that on $[0,S]$,  $\tilde Z$ is the unique function in    $C^1([0,S];B_b(\mathcal O\times A))$ satisfying  \eqref{eq:md_intro} on $[0,S]$.
\end{proof}

\bibliographystyle{siam}
\bibliography{pg_control}

\end{document}